\documentclass[reqno, 11pt]{amsart}
\usepackage[utf8]{inputenc}
\usepackage[T1]{fontenc}
\usepackage{lmodern}
\usepackage{epsfig,amsmath, amsthm ,amsfonts,latexsym,stmaryrd, amssymb}
\usepackage[abs]{overpic}
\usepackage[usenames,dvipsnames]{xcolor}
\usepackage[linktocpage=true]{hyperref}
\usepackage{caption}
\usepackage{subcaption}
\usepackage{dsfont}
\usepackage{mathtools}
\usepackage{tikz-cd}
\usetikzlibrary{shapes,arrows}
\usepackage{comment}
\usepackage{enumitem}
\usepackage{marginnote}
\usepackage[all,cmtip]{xy}
\usepackage[backend=biber, maxnames=4, style=alphabetic]{biblatex}
\hypersetup{
  colorlinks,
  citecolor=teal,
  linkcolor=purple,
  urlcolor=teal}
\theoremstyle{plain}

\newtheorem{dummy}{anything}[section]
\newtheorem{thm}{Theorem}
\newtheorem{cor}{Corollary}
\newtheorem{lemma}[dummy]{Lemma}

\newtheorem{question}[dummy]{Question}

\newtheorem{proposition}{Proposition}[section]

\newtheorem{corollary}[dummy]{Corollary}
\newtheorem{conjecture}[dummy]{Conjecture}

\theoremstyle{definition}
\newtheorem{definition}[dummy]{Definition}

\newtheorem{remark}[dummy]{Remark}

\theoremstyle{remark}
\newcommand{\N}{\mathbb{N}}
\newcommand{\Z}{\mathbb{Z}}
\newcommand{\R}{\mathbb{R}}
\newcommand{\C}{\mathbb{C}}

\newcommand{\w}{\omega}

\def\R{\mathbb{R}}
\def\Z{\mathbb{Z}}
\def\C{\mathbb{C}}
\def\N{\mathbb{N}}
\def\HH{\mathrm{HH}}

\DeclareFontFamily{U}{mathx}{}
\DeclareFontShape{U}{mathx}{m}{n}{<-> mathx10}{}
\DeclareSymbolFont{mathx}{U}{mathx}{m}{n}
\DeclareMathAccent{\widecheck}{0}{mathx}{"71}

\newcommand{\eps}{\varepsilon}
\newcommand{\p}{\partial}

\begin{document}

\title[Floer theory of Anosov flows in dimension three]{Floer theory of Anosov flows in dimension three}

\author{Kai Cieliebak}

\address[K.\ Cieliebak]{Institut f\"ur Mathematik \\ Universit\"at Augsburg \\ Augsburg \\ Germany }

\email{\href{mailto:kai.cieliebak@math.uni-augsburg.de}{kai.cieliebak@math.uni-augsburg.de}}

\author{Oleg Lazarev}

\address[O.\ Lazarev]{Department of Mathematics \\ University of Massachusetts Boston \\ MA \\ USA }

\email{\href{mailto:oleg.lazarev@umb.edu}{oleg.lazarev@umb.edu}}

\author{Thomas Massoni}

\address[T.\ Massoni]{Department of Mathematics \\ Princeton University \\ NJ \\ USA}

\email{\href{mailto:tmassoni@princeton.edu}{tmassoni@princeton.edu}}

\author{Agustin Moreno}

\address[A.\ Moreno]{School of Mathematics \\ Institute for Advanced Study \\ NJ \\ USA / Universit\"at Heidelberg \\ Mathematisches Institut \\ Heidelberg \\ Germany}

\email{\href{mailto:agustin.moreno2191@gmail.com}{agustin.moreno2191@gmail.com}}

\date{\today}



\begin{abstract}
    A smooth Anosov flow on a closed oriented three manifold $M$ gives rise to a Liouville structure on the four manifold $[-1,1]\times M$ which is not Weinstein, by a construction of Mitsumatsu and Hozoori. We call it the associated \emph{Anosov Liouville domain}. It is well defined up to homotopy and only depends on the homotopy class of the original Anosov flow; its symplectic invariants are then invariants of the flow. We study the symplectic geometry of Anosov Liouville domains, via the wrapped Fukaya category, which we expect to be a powerful invariant of Anosov flows. The Lagrangian cylinders over the simple closed orbits span a natural $A_\infty$-subcategory, the \emph{orbit category} of the flow. We show that it does not satisfy Abouzaid's generation criterion; it is moreover ``very large'', in the sense that is not split-generated by any strict sub-family. This is in contrast with the Weinstein case, where critical points of a Morse function play the role of the orbits. For the domain corresponding to the suspension of a linear Anosov diffeomorphism on the torus, we show that there are no closed exact Lagrangians which are either orientable, projective planes or Klein bottles. By contrast, in the case of the geodesic flow on a hyperbolic surface of genus $g \geq 2$ (corresponding to the McDuff example), we construct an exact Lagrangian torus for each embedded closed geodesic, thus obtaining at least $3g-3$ tori which are not Hamiltonian isotopic to each other. For these two prototypical cases of Anosov flows, we explicitly compute the symplectic cohomology of the associated domains, as well as the wrapped Floer cohomology of the Lagrangian cylinders, and several pair-of-pants products. 
\end{abstract}

\maketitle

\tableofcontents

\section{Introduction}
The study of Lagrangian submanifolds is a central topic in symplectic geometry. These submanifolds (equipped with extra structure and adjectives) are usually packaged as the objects of an $A_\infty$-category, the \emph{Fukaya category}, whose morphisms consist of suitable Floer chain complexes. These are endowed with operations which are based on counts of punctured disks solving Floer's equation, and underlie the rich algebraic structure of the theory. This category is central to the far-reaching homological mirror symmetry conjecture of Kontsevich, a bridge between symplectic and algebraic geometry, which has its origins in string theory.

On the symplectic side, the notion of a \emph{Weinstein manifold}, introduced by Eliashberg and Gromov~\cite{EG91} based on a construction of Weinstein~\cite{W91}, gives a particularly rich class of examples of symplectic manifolds, all of which are modelled on the cylindrical end over a contact manifold at infinity. The corresponding theory is the symplectician's approach to Morse theory. The interplay between the geometry and the dynamics imposes strong topological restrictions on a Weinstein domain, i.e., its homotopy type is that of a CW-complex of at most half the dimension. In particular, except in dimension two, its boundary is always non-empty and connected. The (wrapped) Fukaya category of a Weinstein domain admits a finite collection of building blocks, e.g., it is generated by the unstable manifolds of the critical points of the chosen Morse--Lyapunov function \cite{CDGG, GPS2}. This is a fundamental structural result for the underlying symplectic geometry of these examples. 

On the other hand, there exist examples of exact symplectic manifolds with contact boundary (i.e., Liouville domains) which do not admit a Weinstein structure, and which are especially interesting from a dynamical perspective. Instances of such were constructed in \cite{McD91,Ge95,M95,MNW13}, all of the form $V=[-1,1]\times M$ for some closed manifold $M$, and in particular with disconnected contact-type boundary $$\partial V=(M_+, \xi_+=\ker \alpha_+)\sqcup (M_-, \xi_-=\ker \alpha_-),$$ with $M_\pm:=\{\pm 1\}\times M$. 

\begin{figure}
    \centering
    \includegraphics[width=0.7\linewidth]{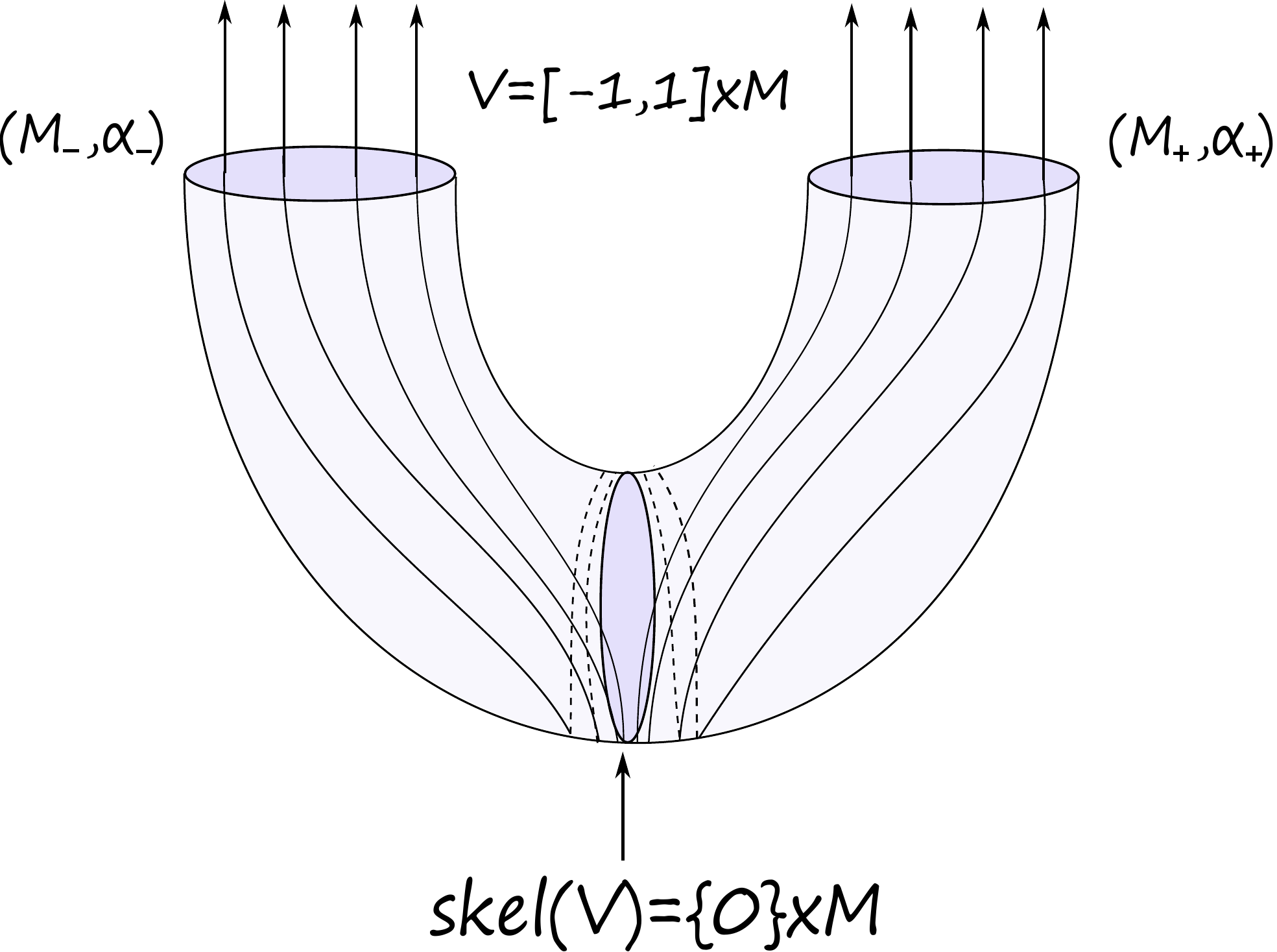}
    \caption{A Liouville domain of the form $V=[-1,1]\times M$. For the McDuff and torus bundle domains described below, $V$ retracts onto the skeleton $\mathrm{skel}(V)=M_0=\{0\}\times M$ under the negative Liouville flow. The Liouville flow is tangent to $M_0$ and is Anosov on $M_0$. For the more general constructions in~\cite{H22a, Mas22}, the skeleton is the graph of a $\mathcal{C}^1$ but not necessarily smooth function on $M_0$; in fact, the regularity of the skeleton is related to the regularity of the weak-stable foliation~\cite{H}.}
    \label{semifill}
\end{figure}

\subsection{Bridge to Anosov theory}

We focus our attention on $4$-dimensional examples. If a closed oriented $3$-manifold $M$ admits an smooth Anosov flow, then $[-1,1]\times M$ carries a Liouville structure with contact-type boundary by \cite[Theorem 1.1]{H22a}. This construction generalizes work of Eliashberg--Thurston \cite{ET98} and Mitsumatsu \cite[Theorem 3]{M95}, and builds a bridge between Anosov theory and \emph{bi-contact topology}. The direction of the flow spans the $1$-dimensional intersection $\xi_-\cap \xi_+$ of the contact structures at each end. Moreover, the Liouville structure obtained this way only depends on the underlying Anosov flow up to Liouville homotopy, and a $1$-parameter family of Anosov flows induces a $1$-parameter family of Liouville structures, see~\cite{Mas22}. As a result, all the symplectic invariants (e.g., symplectic cohomology, Rabinowitz Floer cohomology, wrapped Fukaya category) of these Liouville structures are \emph{invariants of the underlying Anosov flow}, only depending on its homotopy class in the space of Anosov flows. One may restrict the focus to the class of \emph{smooth} Anosov flows, as Anosov's structural stability implies that every $\mathcal{C}^1$ Anosov flow can be smoothened without changing its homotopy class in the space of Anosov flows, and the new flow is topologically equivalent to the original one. In this article, we initiate a systematic exploration of this bridge between symplectic geometry and Anosov theory, via the study of the wrapped Fukaya categories of the associated Liouville structures. The constructions mentioned above can be summarized as follows:

\vspace{0.4cm}

\begin{center}
\tikzstyle{block} = [rectangle, draw, fill=blue!15, 
    text width=7.7em, text centered, rounded corners, minimum height=4.5em]
\tikzstyle{line} = [draw, -latex']
    
\begin{tikzpicture}[node distance = 4cm, auto]
    \node [block] (A) {Smooth Anosov flow on $M$};
    \node [block, right of=A] (B) {Suitable pair of contact forms $(\alpha_-, \alpha_+)$ on $M$};
    \node [block, right of=B] (C) {Liouville form $\lambda = e^{-s} \alpha_- + e^s \alpha_+$ on $V=\R \times M$};
    \node [block, right of=C] (D) {Symplectic invariants of $(V,\lambda)$};
    \path [line] (A) -- (B);
    \path [line] (B) -- (C);
    \path [line] (C) -- (D);
    \path [line,dashed] (D) |- ([yshift=-1cm] A.south) -- (A);
\end{tikzpicture}
\end{center}

The first two arrows are constructed in~\cite{H22a} (and~\cite{Mas22} for this specific type of Liouville form). The present paper focuses on the last box of the flowchart. The dashed arrow represents the bridge to Anosov theory, i.e., the fact that the symplectic invariants are invariants of the flow. Below, we will explicitly compute relevant invariants for the two simplest examples of Anosov flows: suspensions of Anosov diffeomorphisms on the torus, and geodesic flows of hyperbolic surfaces. 

\subsection{The orbit category of Anosov flows}

The Liouville structures we consider on $V=\R \times M$ are of the form $\lambda=e^{-s} \alpha_- + e^s \alpha_+$, where $\alpha_-$ and $\alpha_+$ are two contact forms on $M$ with opposite orientations. Their underlying contact structures $\xi_-=\ker \alpha_-$ and $\xi_+=\ker \alpha_+$ are transverse and their intersection $\xi_- \cap \xi_+$ is spanned by an Anosov vector field (see Definition~\ref{def:alstructure}). We will call $(V, \lambda)$ obtained this way an \emph{Anosov Liouville manifold}, and the underlying domain $([-1, 1] \times M, \lambda)$ an \emph{Anosov Liouville domain}.\footnote{In this paper all Liouville manifolds will be of finite type, i.e., completions of Liouville domains. Therefore, we can freely go back and forth between Liouville domains and manifolds, with the boundary of the Liouville domain corresponding to the ideal boundary (at infinity) of the corresponding Liouville manifold. \label{footnote1}} 

We may find an abundance of simple closed orbits for the Anosov flow generated by this vector field, and any such orbit $\Lambda$ is Legendrian for both contact structures. Therefore,
$$
\mathcal{L}_\Lambda= \R \times \Lambda \subset V=\R \times M
$$
is a strictly exact Lagrangian (i.e., $\lambda\vert_{\mathcal{L}_\Lambda}\equiv 0$) with Legendrian boundary at infinity in $\partial_\infty V=\{\pm \infty\}\times M$. 

Next, we consider the full $A_\infty$-subcategory $\mathcal{W}_0(V)$ of the wrapped Fukaya category $\mathcal{W}(V)$ whose objects are the Lagrangians $\{\mathcal{L}_\Lambda\}$; this subcategory contains all information pertaining to the orbits of the Anosov flow. We call $\mathcal{W}_0(V)$ the \emph{orbit category} of the Anosov flow, as we may think of it as the set of simple closed orbits of the flow endowed with the structure of an $A_\infty$-category.\footnote{The objects in $\mathcal{W}(V)$ and $\mathcal{W}_0(V)$ are geometric, i.e., Lagrangians. Concretely, we do not consider the derived category obtained by passing to the split-closure plus the idempotent completion.}

In the case where the skeleton is $M_0 = \{0\}\times M$, the Lagrangian $\mathcal{L}_\Lambda$ is contained in the unstable manifold of the orbit $\Lambda \subset M_0$, i.e., $\mathcal{L}_\Lambda$ retracts to $\Lambda$ under the negative Liouville flow. This phenomenon generalizes to the case where the skeleton is the graph of a $\mathcal{C}^1$ function $f : M \rightarrow \R$ over $M_0$, with $\Lambda$ replaced by the graph of $f$ over $\Lambda$. Note that in the Weinstein case, a critical point of the Morse function can be viewed as an orbit of the Liouville flow, and the corresponding Lagrangian co-core is the whole unstable manifold. By analogy to the Weinstein case, it is therefore natural to ask whether the collection $\{\mathcal{L}_\Lambda\}$ of these Lagrangians recovers all of the wrapped Fukaya category of $V$, i.e., whether $\mathcal{W}(V)$ lies in the split-closure of $\mathcal{W}_0(V)$.

For this purpose, one would usually try to appeal to Abouzaid's generation criterion \cite{A10}, which ensures split-generation under the hypothesis that the \emph{open-closed map}\footnote{We use the grading conventions from~\cite{G19,GPS1}, which differ from~\cite{A10}.} 
$$\mathcal{OC}: \HH_{*-2}(\mathcal{W}(V))\rightarrow SH^*(V)$$ hits the unit when restricted to the orbit category $\mathcal{W}_0(V)$. Here, $SH^*(V)$ denotes the symplectic cohomology of $V$, and $\HH_{*}(\mathcal{W}(V))$, the Hochschild homology of $\mathcal{W}(V)$. There is always a splitting  $SH^*(V)=SH^*_c(V)\oplus SH^*_{nc}(V)$
as $\mathbb{Z}$-modules, where $SH^*_c(V)$ and $SH^*_{nc}(V)$ denote the summands of symplectic cohomology generated by the contractible and non-contractible Hamiltonian orbits, respectively. For a $4$-dimensional Anosov Liouville manifold $V=\R \times M$, the contact forms $\alpha_\pm$ at both components of the boundary at infinity can be chosen \emph{hypertight} (see \cite[Theorem 1.1]{H22a}), i.e., they admit no contractible Reeb orbits, which implies that $SH^*_c(V)\cong H^*(M)$. In particular, $SH^*_c(V)$ contains the unit, lying in the degree zero part $H^0(M)$. We shall prove the following, which is in contrast to the Weinstein case:

\begin{thm}[Open-closed map]\label{thm:unit} Let $V=\R \times M$ be a $4$-dimensional Anosov Liouville manifold. Let $\mathcal{W}_0(V)$ be the orbit category of the Anosov flow, generated by its simple closed orbits. Then $\mathcal{W}_0(V)$ does \textbf{not} satisfy Abouzaid's generation criterion, i.e., the restriction of the open-closed map to $\mathcal{W}_0(V)$, $$\mathcal{OC}_0 : \HH_{*-2}(\mathcal{W}_0(V)) \rightarrow SH^*(V),$$ does \textbf{not} hit the unit. 

More precisely, there is a splitting $\HH_*(\mathcal{W}_0(V)) = \HH_*^c \oplus \HH_*^{nc}$ such that $\mathcal{OC}_0$ splits as a sum of two maps
\begin{align*}
    \mathcal{OC}^{c}_0 : \HH_{*-2}^c &\rightarrow SH^*_c(V) \cong H^*(M), \\
    \mathcal{OC}^{nc}_0 : \HH_{*-2}^{nc} &\rightarrow SH^*_{nc}(V),
\end{align*}
and $$\mathrm{Im}(\mathcal{OC}^c_0) \subseteq H^2(M; \Z) \oplus H^3(M; \Z).$$
Moreover, there is an isomorphism $$\HH_*^c \cong \bigoplus_{\Lambda} W_*$$
where $W_*:=\HH_*(C^*(S^1))$, the Hochschild homology of the singular cochain dg-algebra of the circle, has infinite rank and is supported in degrees $0$ and $1$, and the sum runs over the simple closed orbits of the Anosov flow. Hence, the kernel of $\mathcal{OC}_0$ has infinite rank.
\end{thm}

In fact, more is true. The following theorem says that the subcategory $\mathcal{W}_0(V)$ is ``maximally non-finitely split-generated''. Loosely speaking, it is very large. 

\begin{thm}[Non-finite split-generation] \label{thm:nonsplit} Let $\mathcal{A}$ be a full $A_\infty$-subcategory of the orbit category $\mathcal{W}_0(V)$, generated by a collection $\mathcal{C}$ of simple closed orbits of the flow. If $L=\mathcal{L}_\Lambda$ is a Lagrangian cylinder corresponding to an orbit $\Lambda\notin \mathcal{C}$, then $L$ is \textbf{not} split-generated by $\mathcal{A}$.
In particular,
\begin{enumerate}
    \item Any two Lagrangians $\mathcal{L}_\Lambda,\mathcal{L}_{\Lambda'}$ with $\Lambda\neq \Lambda'$ are \textbf{not} quasi-isomorphic in $\mathcal{W}(V)$,
    \item $\mathcal{W}_0(V)$ is \textbf{not} split-generated by finitely many objects,
    \item $\mathcal{W}_0(V)$ is \textbf{not} homologically smooth.
\end{enumerate}
\end{thm}

Here, recall that an $A_\infty$-category is \emph{homologically smooth} if the diagonal bi-module is perfect, i.e., it is split-generated by tensor products of left and right Yoneda modules (see, e.g.,~\cite{G19}). We remark that, although failing to satisfy Abouzaid's criterion, the family $\{\mathcal{L}_\Lambda\}$ might still constitute a split-generating collection for $\mathcal{W}(V)$. If it does split-generate, the above results would imply that the whole category $\mathcal{W}(V)$ is not homologically smooth, and the total open-closed $\mathcal{OC}$ map fails to be an isomorphism. This would be in stark contrast with the Weinstein case, in which it is an isomorphism, a fact originally envisioned by Seidel in his ICM address \cite{Sei02} (see also \cite{Sei09}), and which follows from \cite{A10,G13,CDGG,GPS2}. In fact, we do not know how to construct other exact Lagrangians which are not expected to lie in the split-closure of $\mathcal{W}_0(V)$. 

\begin{remark} The following remarks are in order.
\begin{enumerate}
    \item The main geometric input in the proof of Theorems~\ref{thm:unit}, \ref{thm:nonsplit}, and \ref{thm:WFH} is the Topological Disk Lemma~\ref{lemma:disk}, which implies that for suitable Floer data, a punctured Floer disk with at least one nonconstant chord at its inputs must have a nonconstant chord at its output. This lemma has algebraic consequences beyond the ones stated here, cf.~the discussion in~\S\ref{sec:discussion}. The proof of this lemma relies crucially on the \textit{tautness} of the Anosov weak-stable foliation. 
    \item The collection $\{\mathcal{L}_\Lambda\}$ constitutes a fairly general class of Lagrangians in $V$. As  mentioned earlier, the $\mathcal{L}_\Lambda$'s are \textit{strictly} exact, namely the Liouville form vanishes on them, hence the Liouville vector field is everywhere tangent to them. In fact, in the cases where $\mbox{skel}(V)=\{0\}\times M$ is smooth (e.g., for the McDuff and torus bundle domains defined below), any connected, strictly exact Lagrangian in $V$ which is transverse to the skeleton must be one of the $\mathcal{L}_\Lambda$'s. Indeed, any such Lagrangian must intersect the skeleton along a closed orbit and must be invariant under the Liouville flow which is proper away from the skeleton. More generally, we expect that Proposition 1.37 of  \cite{GPS2} can be used to show that any exact Lagrangian which is strictly exact \textit{near} the skeleton (but not necessarily elsewhere) is actually generated by a finite sub-collection of $\{\mathcal{L}_\Lambda\}$; the proof of this proposition requires a further technical  condition called \textit{thinness}, although it is expected to be unnecessary.
    \item  More generally, one can ask whether an arbitrary exact Lagrangian, not necessarily strictly exact near the skeleton, is (split-)generated by the $\mathcal{L}_\Lambda$'s. Assuming the skeleton is smooth, there are no restrictions on the loop arising from a transverse intersection of a (local, non-strictly exact) Lagrangian with the skeleton, beyond the fact that the loop is \emph{non-characteristic}, i.e., nowhere tangent to the $1$-dimensional characteristic distribution. In particular, any homotopy class of loops is achieved as such a transverse intersection. If one could use the Anosov flow to flow an arbitrary non-characteristic loop in the skeleton to a periodic orbit, then one would have generation for an arbitrary (non-strictly) exact Lagrangian in $V$. Indeed, this is one approach to proving generation in the Weinstein case, where a generic point in the skeleton converges to a maximum of the Weinstein Morse function \cite{CDGG}. The difficulty here is dynamical: it is not true that arbitrary non-characteristic loops converge to closed orbits under the Anosov flow. 
\end{enumerate}
\end{remark}


\subsection{Closed Lagrangians} A further natural question concerns the existence of nontrivial \emph{closed} Lagrangians. We shall address this for some concrete examples, which we now describe. 

\medskip

\textbf{McDuff domains.} The first example of a Liouville domain of the form $(V=[-1,1]\times M,\lambda)$ is due to McDuff \cite{McD91}, where $M$ is the unit cotangent bundle of a hyperbolic surface $\Sigma$, and $V$ is a subdomain of the cotangent bundle of $\Sigma$ with symplectic form twisted with a magnetic field. It corresponds to an Anosov Liouville domain where the Anosov flow is (a rotated version of) the geodesic flow on $\Sigma$, a prototypical example of Anosov flow. The skeleton $M$ has $\mathrm{SL}(2,\mathbb{R})$-geometry. We will refer to these examples as the {\em McDuff domains}. The Lagrangians $\mathcal{L}_\Lambda$ described above correspond to the positive conormal bundle of a closed oriented geodesic $\gamma$; we denote the corresponding orbit $\Lambda$ by $\Lambda_\gamma$ (the unit conormal lift of $\gamma$), and the corresponding Lagrangian by $\mathcal{L}_\gamma$. In Section \ref{sec:closed_Lag_McDuff}, we will show the following:

\begin{thm}[McDuff domains: closed exact Lagrangians] In every McDuff domain, there exist $3g-3$ pairwise disjoint exact Lagrangian tori in distinct homotopy classes, where $g$ denotes the genus of $\Sigma$. 
\end{thm}

Indeed, we construct an embedded exact Lagrangian torus $\mathbb T_\gamma$ for each closed embedded geodesic $\gamma \subset \Sigma$. Taking $3g-3$ disjoint closed geodesics in a pair of pants decomposition of $\Sigma$, we obtain the above. The torus $\mathbb T_\gamma$ may be constructed from two Hamiltonian isotopic copies of $\mathcal{L}_\gamma$ by $S^1$-equivariant versions of Polterovich surgeries, although we provide a simpler and more explicit construction. The torus $\mathbb T_\gamma$ is expected to be split-generated by $\mathcal{L}_\gamma$, but this would require an algebraic implementation of such surgery in terms of iterated cones, which we will not pursue. It is also possible to associate \emph{non-exact} (but weakly exact) Lagrangian tori to each closed embedded geodesic by a much simpler construction; see Section \ref{sec:Mcduff_domains}.

\medskip

\textbf{Torus bundle domains.} The torus bundle domains were first described by Geiges~\cite{Ge95} and independently by Mitsumatsu~\cite{M95}. They correspond to the case of Sol geometry. As a smooth manifold, they are of the form $[-1,1] \times M$, where $M$ is a $\mathbb{T}^2$-bundle over $S^1$ whose monodromy is given by a hyperbolic matrix $A \in \mathrm{SL}(2,\Z)$. These are the Anosov Liouville domains with Anosov flow given by the suspension of $A$, viewed as an Anosov diffeomorphism of the torus. The Lagrangians $\mathcal{L}_\Lambda$ arise from periodic orbits $\mathcal{O}$ of $A$; we denote the corresponding orbit $\Lambda$ by $\Lambda_\mathcal{O}$, and the corresponding Lagrangian by $\mathcal{L}_\mathcal{O}$. The Reeb vector fields of $\alpha_\pm$ on $M_\pm$ give linear flows on each $\mathbb{T}^2$-fiber with slope varying from fiber to fiber.\footnote{Breen and Christian~\cite{BC21} recently proved that the \emph{stabilization} of this domain is Weinstein.}

\begin{thm}[Torus bundle domains: closed exact Lagrangians]\label{thm:closed}
In every torus bundle domain, there are no closed exact Lagrangian submanifolds which are either orientable, projective planes, or Klein bottles.
\end{thm}

However, similarly to the McDuff domains, the torus bundle domains admit non-exact but weakly exact Lagrangian tori, namely the $\mathbb{T}^2$-fibers.

\medskip

 The McDuff domains (and finite quotients thereof corresponding to unit cotangent bundles of hyperbolic orbifold surfaces) and the torus bundle domains correspond to \emph{algebraic Anosov flows}. Those are the flows generated by Anosov vector fields which are invariant vector fields on manifold quotients of the Lie groups $\widetilde{\mathrm{SL}}\big(2,\R\big)$ (the universal cover of $\mathrm{SL}\big(2,\R\big)$) and $\mathrm{Sol}^3$ (the semi-direct product of $\R$ with $\R^2$ for the action $z\cdot (x,y) = (e^z x, e^{-z} y)$), respectively. Algebraic Anosov flows enjoy several important properties which are relevant in the construction of their associated Liouville structures: they are \emph{volume preserving}, and their stable and unstable foliations are \emph{smooth}. By classical theorems of Ghys~\cite{Gh92, Gh93}, those are essentially the only (volume preserving) Anosov flows with smooth Anosov splitting in dimension $3$.

\subsection{Symplectic cohomology and Rabinowitz Floer cohomology}

We are also interested in computing relevant symplectic invariants for Anosov Liouville domains, as they are homotopy invariants for the Anosov flow. We begin with their symplectic cohomology and Rabinowitz Floer cohomology, which can be studied in arbitrary dimension. The hypertightness of the contact boundary, which follows from the tautness of the weak-stable foliation of the Anosov flow for Anosov Liouville domains, plays a crucial role in the computation. Under a suitable assumption on the free homotopy classes of the closed Reeb orbits, we can entirely compute the symplectic cohomology \emph{as a ring}. We exploit the relations between symplectic cohomology and Rabinowitz Floer cohomology to drastically simplify the study of the ring structure on $SH^*$ and reduce it to the ring structure on the Rabinowitz Floer cohomology of the two contact boundary components. The latter are independent of the Liouville filling and can be computed independently of each other.\footnote{Note that in the case of Anosov Liouville domains, the skeleton is \emph{not} a stable Hamiltonian hypersurface. It is not possible to apply the usual neck-stretching argument to study holomorphic curves crossing it.}

\begin{thm}[Rabinowitz Floer cohomology and symplectic cohomology rings]\label{thm:RFH-SH} 
Let $V=[-1,1]\times M$ be a Liouville domain of dimension $2n$ such that each boundary component $M_\pm=\{\pm 1\}\times V$ is hypertight. Then its Rabinowitz Floer cohomology splits as a ring,
$$
   RFH^*(V)\cong RFH^*(M_-)\oplus RFH^*(M_+).
$$
Moreover, its symplectic cohomology splits as a $\Z$-module,
$$
   SH^*(V) = SH^*_-(V)\oplus SH^*_0(V) \oplus SH^*_+(V),
$$ 
where $SH^*_\pm(V)$ is generated by closed Reeb orbits on $M_\pm$ and $SH^*_0(V)\cong H^*(M)$ is generated by critical points on $V$.

Assume in addition that the free homotopy classes of positively or negatively parametrized closed Reeb orbits on $M_+$ are distinct from those on $M_-$. Then the summands $SH^*_\pm(V)$ and $SH_0(V)$, as well as $SH^*_{0-}(V)=SH^*_-(V)\oplus SH^*_0(V)$ and $SH^*_{0+}(V)=SH^*_0(V) \oplus SH^*_+(V)$, are subrings given by
\begin{itemize}
    \item $SH^*_0(V)\cong H^*(M)$ with the cup product;
    \item $SH^*_\pm(V)\cong RFH^*_{>0}(M_\pm)$;
    \item $SH^*_{0\pm}(V)\cong RFH^*_{\geq 0}(M_\pm)$. 
\end{itemize}
Here $RFH^*_{>0}(M_\pm)$ and $RFH^*_{\geq 0}(M_\pm)$ are the positive and nonnegative action parts of Rabinowitz Floer cohomology, respectively, with the product described in~\cite{CO18}. 
\end{thm}

\begin{remark}\label{rk:SH} Some remarks are in order.
\begin{enumerate}
    \item The assumption on hypertightness for $M_\pm$ holds for all examples of Liouville domains $V=[-1,1]\times M$ that we are aware of. We do not know if the assumption on free homotopy classes of Reeb orbits holds for any Anosov Liouville domain. In Lemma~\ref{lem:hypertight}, we verify that it holds for McDuff domains and torus bundle domains.
    \item We call symplectic/Rabinowitz Floer \emph{cohomology} what is called homology in~\cite{CFO10,CO18} and grade it by $n$ minus the Conley--Zehnder index, i.e., degree $*$ is replaced by $n-*$. 
    \item Here, we use the non $S^1$-equivariant version of symplectic cohomology. If $c_1(V)=0$ (which is the case in our examples, see Lemma \ref{lemma:symptriv}) it has an integer grading, and we fix such a choice.
    \item In some cases, similar arguments also impose strong restrictions on the higher structures (e.g., $A_\infty$, $L_\infty$-structure) on the symplectic cochain complex. However, these can be quite complicated in general, and involve curves having inputs/outputs from both $M_+$ and $M_-$.
    \item Under the assumption on the homotopy classes of Reeb orbits, Theorem~\ref{thm:RFH-SH} implies a rather curious algebraic description for $SH^*$: $A_0 := SH^*_0$ and $A_{\pm} := SH^*_{0 \pm}$ are sub-$\Z$-algebras of $A:= SH^*$, $I_\pm := SH^*_{\pm} \subset A_{\pm} \subset A$ are \emph{ideals} such that $I_- \cap I_+ =0$ and the quotients $A_{\pm} \slash I_\pm$ are isomorphic to $A_0$. Therefore, $SH^*$ is the fiber product of $\Z$-algebras
    $$\begin{tikzcd}
    A \arrow[r] \arrow[d]	    \arrow[dr, phantom, "\lrcorner", very near start] &		A_{-} \arrow[d]\\
    A_{+}	\arrow[r]               &	    A_0
    \end{tikzcd}$$ where the maps $A_{\pm} \rightarrow A_0$ are the quotient maps by $I_\pm$. This corresponds to ``gluing the affine schemes $\mathrm{Spec}(A_-)$ and $\mathrm{Spec}(A_+)$ along the subscheme $\mathrm{Spec}(A_0)$''.
\end{enumerate}
\end{remark}

For McDuff domains and torus bundle domains, which satisfy the assumptions of Theorem~\ref{thm:RFH-SH} by Lemma~\ref{lem:hypertight} below, the ring structures can be described more explicitly:

\begin{cor}[Symplectic cohomology rings of McDuff domains]\label{cor:SH-McDuff}
Let $V=[-1,1]\times M$ be a McDuff domain modeled on $M=S^*\Sigma$. Then the splitting in Theorem~\ref{thm:RFH-SH} becomes
$$
   SH^*(V) \cong tH^*(M)[t]\oplus H^*(M)\oplus H_{2-*}(\mathcal{L}^{nc}\Sigma),
$$ 
where $t$ is a variable of degree zero representing the $S^1$-fibre and $\mathcal{L}^{nc}\Sigma$ is the space of non-contractible loops on $\Sigma$. The subrings are the following:
\begin{itemize}
    \item $SH^*_{0-}(V)\cong H^*(M)[t]$ and $SH^*_-(V)\cong tH^*(M)[t]$ with product as polynomial rings;
    \item $SH^*_{0+}(V)\cong \widecheck{H}_{2-*}^{\geq 0}(\mathcal{L}\Sigma)$, the nonnegative action part of Rabinowitz loop homology with product described in~\cite{CHO};
    \item $SH^*_+(V)\cong H_{2-*}(\mathcal{L}^{nc}\Sigma)$ with the loop product. 
\end{itemize}
\end{cor}

\begin{cor}[Symplectic cohomology of torus bundle domains]\label{cor:SH-torusbundle}
Let $V=[-1,1]\times M$ be a torus bundle domain. Then the splitting in Theorem~\ref{thm:RFH-SH} becomes
$$
   SH^*(V)\cong \bigoplus_{\Gamma} H^*\big(S^1\big) \oplus H^{*}(M) \oplus \bigoplus_{\Gamma} H^*\big(S^1\big), 
$$
where $\Gamma$ is in bijection with $\mathbb Q\cap [0,1)$ and corresponds to the $\mathbb T^2$-fibers along which the $\alpha_\pm$-orbits have rational slope (i.e., the rational lines inside the cone of Figure \ref{fig:chords} below). 
\end{cor}

\begin{remark}[Torus bundle domains: products]\label{rem:products_torus_bundle}
The product in $SH^*$ of two $\alpha_\pm$-orbits $\gamma_1,\gamma_2$ lying in two different $\mathbb{T}^2$-fibers with respective rational slopes $s_1,s_2$ (measured in a fundamental domain $\mathbb{T}^2 \times [0,\nu)$ of the torus bundle, see below for the construction) can only be a linear combination of orbits lying in  $\mathbb{T}^2$-fibers with specific slopes determined by $s_1$ and $s_2$. However, we do not know if there are any Floer solutions contributing to this product. 
\end{remark}

\subsection{Wrapped Floer cohomology}

The following is an open string version of Theorem~\ref{thm:RFH-SH}; unlike in the closed string case, which holds in more generality, the proof here requires the Topological Disk Lemma \ref{lemma:disk} and hence the result holds only for Anosov Liouville domains. For the purpose of exposition, the explicit computations for the McDuff domains and torus bundle domains are carried out in Section \ref{sec:WFH}.

\begin{thm}[Wrapped Floer cohomology]\label{thm:WFH} 
Let $V=[-1,1]\times M$ be a $4$-dimensional Anosov Liouville domain. Consider the $\Z$-algebra 
$$A=\bigoplus_{\Lambda,\Lambda'} HW^*(\mathcal{L}_\Lambda,\mathcal{L}_{\Lambda'}),
$$
where the sum runs over pairs of simple closed orbits of the Anosov flow.
\begin{enumerate}
    \item 
We have a natural splitting of $A$ as a $\Z$-module
$$
A=I_-\oplus A_0\oplus I_+,
$$
where 
$$A_0 := \bigoplus_{\Lambda} HW_0^*(\mathcal{L}_\Lambda,\mathcal{L}_{\Lambda}),
$$
$$
I_\pm:=\bigoplus_{\Lambda,\Lambda'} HW_\pm^*(\mathcal{L}_\Lambda,\mathcal{L}_{\Lambda'}),
$$
with $HW_0^*(\mathcal{L}_\Lambda,\mathcal{L}_{\Lambda}) \cong H^*\big(S^1;\Z\big)$ generated by intersection points of $\mathcal{L}_\Lambda$ with a perturbation of itself, and $HW_\pm^*(\mathcal{L}_\Lambda,\mathcal{L}_{\Lambda'})$ generated by chords from $\Lambda$ to $\Lambda'$ in $M_\pm$.
\item 
If we let
$$
A_{\pm}:=A_0\oplus I_\pm,
$$
then $A_0$ and $A_\pm$ are subrings of $A$, $I_\pm$ are ideals of $A$ such that $I_- \cap I_+ = 0$, and the product on $$A_0=\bigoplus_\Lambda H^*(\mathcal{L}_\Lambda; \Z)\cong \bigoplus_\Lambda H^*(S^1;\Z)$$ is the direct sum of the cup product on each $H^*(S^1;\Z)$.
\end{enumerate}
\end{thm}

\begin{remark}
The $\Z$-algebra $A$ satisfies the same fiber product description as in Remark~\ref{rk:SH} (5) in terms of the $\Z$-algebras $A_\pm$, $A_0$ and the ideals $I_\pm$.
\end{remark}

\textbf{A word on gradings and coefficients.} Anosov Liouville manifolds have vanishing first Chern class since their symplectic tangent bundle is trivial (Lemma~\ref{lemma:symptriv}), so their wrapped Fukaya category comes with a natural $\Z$-grading. All of the Lagrangians that we consider are spin, and the Lagrangian cylinders have a canonical spin structure. Throughout the paper, we will use $\Z$ coefficients unless stated otherwise (we will need Novikov coefficients when dealing with weakly exact Lagrangians); our results remain valid over $\Z \slash 2 \Z$ or $\R$. 

\subsection*{Acknowledgments} 

O.~Lazarev was supported by NSF postdoctoral fellowship, award \#1705128, and by the Simons Foundation through grant \#385573, the Simons Collaboration on Homological Mirror Symmetry. 
T.~Massoni is grateful to his PhD advisor John Pardon for his constant support and encouragement; to Jonathan Zung for insightful conversations about Anosov flows; to Surena Hozoori for numerous valuable discussions about his work; to Patrick Massot for sharing a note on torus bundle domains. 
A.~Moreno is grateful to Vivek Shende, for insightful conversations during the author's fellowship at the Mittag-Leffler Institute in Djursholm, and in Uppsala, Sweden; to Alex Ritter, for productive discussions at Stanford University; to Rafael Potrie, for helpful inputs on Anosov theory. This author is supported by the National Science Foundation under Grant No.\ DMS-1926686.

The first and fourth authors thank the Institute for Advanced Study, Princeton, for its hospitality in the academic year 2021/22, during which most of this work was carried out.
The authors are also grateful to Georgios Dimitroglou Rizell for very helpful inputs and discussions, particularly about Section~\ref{sec:closed_Lag_McDuff}, and to Mohammed Abouzaid for his comments about the Lagrangian foliation by cotangent fibers in the McDuff example.

\section{Non-Weinstein Liouville domains}

In this section, we review the construction, described by Mitsumatsu~\cite{M95} and generalized by Hozoori~\cite{H22a}, of Liouville domains with disconnected contact-type boundary associated to an Anosov flow on a $3$-manifold. In particular, we will describe the McDuff domains and the torus bundle domains in more detail. In~\cite{Mas22}, the third author introduces the notion of \emph{Anosov Liouville structure}, a generalization of the previous constructions, and shows a correspondence (up to homotopy) between these structures and Anosov flows on $3$-manifolds. This gives rise to a well-defined class of \emph{Anosov Liouville manifolds} for which our main results apply. 

\subsection{Anosov Liouville manifolds}\label{sec:anosov_Liouville_manifolds} 

We will only sketch the main construction and refer to~\cite{H22a} and ~\cite{Mas22} for a more complete exposition. Consider $\phi_t:M\rightarrow M$ a smooth Anosov flow generated by a smooth non-singular vector field $X$ on a closed oriented $3$-manifold $M$. Let
\begin{align}
TM=\langle X\rangle \oplus E^s \oplus E^u \label{splitting}
\end{align}
be the Anosov splitting, i.e., it is a \emph{continuous}, $d\phi_t$-invariant splitting of $TM$ such that $\phi_t$ is exponentially contracting on $E^s$ (the stable subbundle) and exponentially expanding on $E^u$ (the unstable subbundle). More precisely, for some (any) Riemannian metric $g$ on $M$, there exist constants $C > 0$ and $\mu > 0$ such that for every $v \in E^s$ and $t \in \mathbb{R}$,
$$
\Vert d\phi_t(v)\Vert \leq C e^{- \mu t} \Vert v \Vert,
$$
and for every $v \in E^u$ and $t \in \mathbb{R}$,
$$
\Vert d\phi_t(v)\Vert \geq C e^{\mu t} \Vert v \Vert.
$$

The metric $g$ is said to be \emph{adapted} to $\phi_t$ if the constant $C$ above equals $1$. An Anosov flow on a closed manifold always admits a smooth adapted Riemannian metric, see~\cite[Proposition 5.1.5]{FH19}.

We will further assume that $E^s$ and $E^u$ are oriented, and that the orientation of $M$ coincides with the orientation of the splitting~\eqref{splitting}. This can always be achieved after passing to a suitable double cover of $M$. We will call such a flow an \emph{oriented} Anosov flow. Denote $$E^{ws}:=\langle X\rangle \oplus E^s, \qquad E^{wu}:=\langle X \rangle \oplus E^u$$ the weak-stable and weak-unstable subbundles, respectively. Even though $E^s$ and $E^u$ are only continuous in general, a classical result of Hasselblatt implies that for a smooth Anosov flow in dimension $3$, $E^{ws}$ and $E^{wu}$ are of class $\mathcal{C}^1$. They integrate to codimension one foliations $\mathcal{F}^{ws}$ and $\mathcal{F}^{wu}$ called the weak-stable and weak-unstable foliations, respectively, and they are both \emph{taut foliations}.\footnote{A codimension one foliation on a closed manifold is \emph{taut} if every leaf intersects a transverse circle, i.e., a closed loop transverse to the foliation. By Novikov's theorem, transverse circles to a taut foliation are non-contractible 
(see~\cite[Theorem 4.37]{C07}, noting that taut foliations have no Reeb components). As explained below, there is a Reeb vector field everywhere transverse to $\mathcal{F}^{ws}$. This vector field preserves a (contact) volume form, so the tautness of $\mathcal{F}^{ws}$ follows from~\cite[Theorem 4.29]{C07}. Similarly, $\mathcal{F}^{wu}$ is taut as it is the weak stable foliation of the reverse of the Anosov flow.}

For an adapted metric $g$, let $e_s$ and $e_u$ be two unit vector fields of class $\mathcal{C}^1$ spanning $E^s$ and $E^u$, respectively. There exist continuous functions $r_s, r_u : M \rightarrow \R$ such that
$$ [X, e_s] = -r_s e_s, \qquad [X, e_u] = -r_u e_u,$$
and since the metric is adapted, $r_s < 0 < r_u$. The functions $r_s$ and $r_u$ are called the \emph{expansion rates} of the flow in the stable and unstable direction, respectively (see~\cite[Section 3]{H22a}). These quantities depend on the choice of the metric.

By duality, there exist $1$-forms $\alpha_s$ and $\alpha_u$ of class $\mathcal{C}^1$ with $\ker \alpha_u=E^{ws}, \ker \alpha_s=E^{wu}$, $\alpha_s(e_s) = \alpha_u(e_u) = 1$, and satisfying
\begin{align*}
\mathcal{L}_X \alpha_s = r_s \, \alpha_s, \quad \mathcal{L}_X \alpha_u = r_u \, \alpha_u.
\end{align*}
One easily checks that
\begin{align} \label{eq:standard}
    \alpha_- := \alpha_u + \alpha_s, \quad \alpha_+ := \alpha_u - \alpha_s,
\end{align}
are $1$-forms of class $\mathcal{C}^1$ defining two contact structures $\xi_\pm=\ker \alpha_\pm$ intersecting transversely along $\langle X \rangle$. Moreover, the $1$-form $\lambda$ on $V= \R_s \times M$ given by
\begin{align*}
\lambda :=e^{-s} \alpha_- + e^s \alpha_+
\end{align*}
is a Liouville form (i.e., the \emph{continuous} $2$-form $\omega := d \lambda$ is non-degenerate), whose Liouville vector field has a positive $\partial_s$ component for $s> 0$, a negative $\partial_s$ component for $s < 0$, and is tangent to $\{0\} \times M$. Moreover, the Reeb vector fields $R_\pm$ of $\alpha_\pm$ satisfy
\begin{align*}
\alpha_u(R_+) = \alpha_u(R_-) = \frac{-r_s}{r_u - r_s}>0,
\end{align*}
hence they are both positively transverse to the weak-stable foliation $\mathcal{F}^{ws}$, so $\alpha_\pm$ are \emph{hypertight} by the tautness of $\mathcal{F}^{ws}$. We call this construction of a ($\mathcal{C}^1$) Anosov Liouville structure the \emph{standard construction}.

At this point, $\lambda$ is only of class $\mathcal{C}^1$. It can be smoothened to obtain a Liouville structure on the domain $[-A, A] \times M$, for some arbitrary $A > 0$. Being more careful, it is possible to smoothen $\alpha_\pm$ and obtain $\widetilde{\alpha}_\pm$ such their contact structures $\widetilde{\xi}_\pm=\ker \widetilde{\alpha}_\pm$ still intersect transversally along $\langle X \rangle$, and their  Reeb vector fields are still positively transverse to $\mathcal{F}^{ws}$, see~\cite{Mas22}.
The latter property will be crucial in the proofs of Theorem~\ref{thm:unit} and Theorem~\ref{thm:nonsplit}.

\begin{figure}
    \centering
    \includegraphics[width=0.6\linewidth]{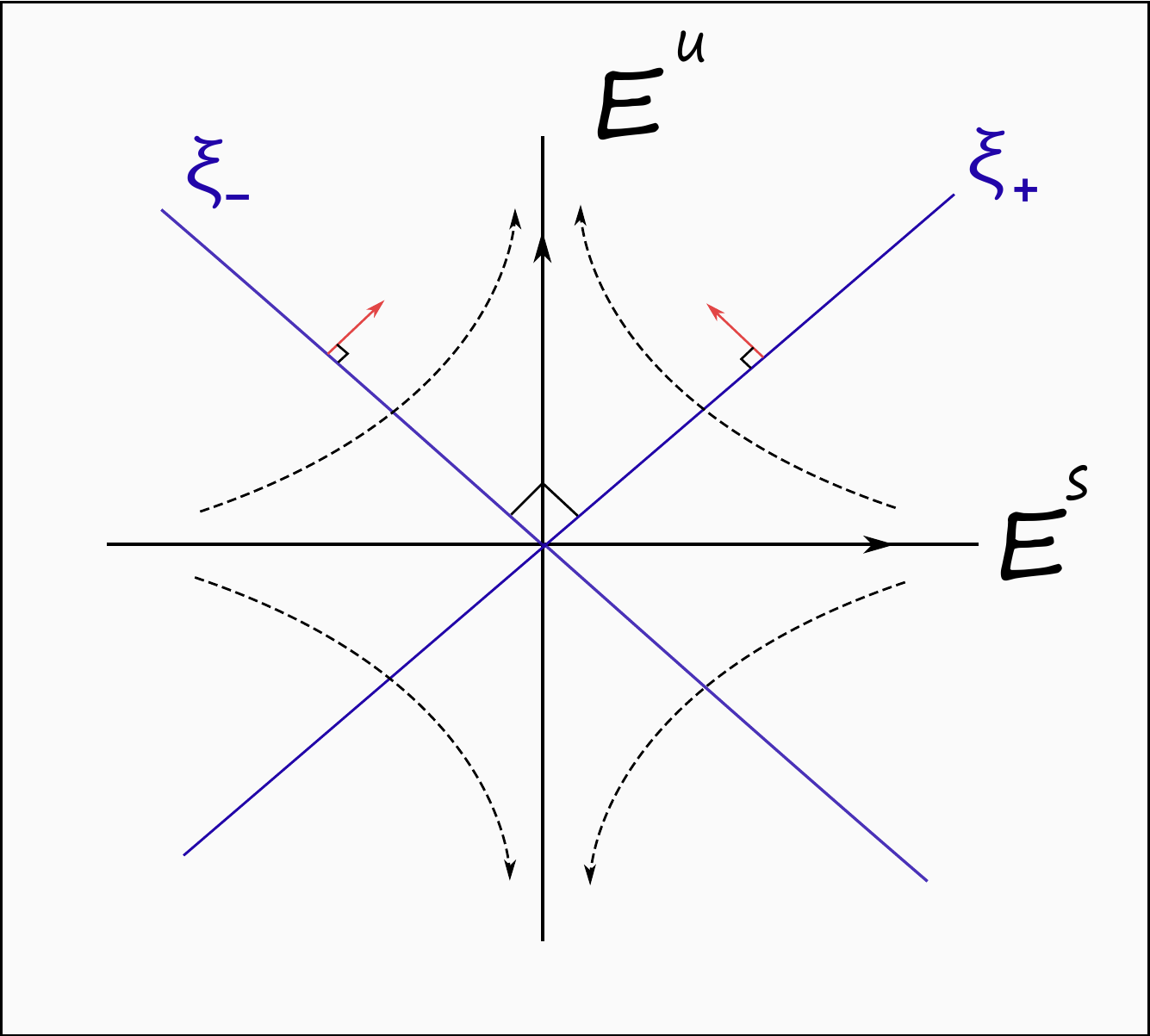}
    \caption{The Anosov splitting of $N=TM/\langle X \rangle$. The contact structures $\xi_\pm$, when flowed along $X$, approach $E^u$. The red arrows show the co-orientations.}
    \label{fig:Anosovsplitting}
\end{figure}

This construction fits in the more general framework of the following
\begin{definition} \label{def:alstructure}
An \emph{Anosov Liouville structure} on $\R \times M$ is a Liouville structure induced by a smooth Liouville form $\lambda$ of the form
\begin{align} \label{eq:exponentialanosov}
\lambda = e^{-s} \alpha_- + e^s \alpha_+
\end{align}
where $(\alpha_-, \alpha_+)$ is a pair of contact forms satisfying the following two conditions.
\begin{enumerate}
    \item The contact plane fields $\xi_\pm = \ker \alpha_\pm$ are transverse everywhere.
    \item The line distribution $\xi_- \cap \xi_+$ is spanned by an Anosov vector field.
\end{enumerate}
\end{definition}

The pair $(\xi_-,\xi_+)$ constitutes a \emph{bi-contact structure} on $M$, i.e., a
transverse pair of contact plane fields with opposite orientations, meaning that $\alpha_\pm \wedge d\alpha_\pm$ induce volume forms on $M$ with opposite orientations. See Figure \ref{fig:Anosovsplitting}.

We say that the Anosov Liouville structure \emph{supports} an Anosov flow $\phi_t$ if the vector field $X$ generating $\phi_t$ satisfies $X \in \xi_- \cap \xi_+$ (and the orientations are compatible). The following theorem is proved in~\cite{Mas22}.

\begin{thm}[\cite{Mas22}]
For any smooth oriented Anosov flow $\phi_t$ on a closed oriented $3$-manifold $M$, the space of Anosov Liouville structures on $V = \R \times M$ supporting $\phi_t$ is non-empty and contractible. Moreover, the map that sends an Anosov Liouville structure to the underlying Anosov flow (up to positive time reparametrization) is a Serre fibration with contractible fibers, hence a homotopy equivalence.
\end{thm}

It follows that any Anosov Liouville structure is Liouville homotopic to one such that the Reeb vector fields of the contact boundaries are both \emph{positively transverse to some taut foliation} (the weak-stable foliation of the underlying Anosov flow). Moreover, the above theorem implies that all the symplectic invariants (symplectic cohomology, wrapped Fukaya category, etc.) of an Anosov Liouville manifold/domain are in fact \emph{invariants of the underlying Anosov flow}.

\begin{remark}
(cf.~footnote~\ref{footnote1}).
In this paper, we work with Liouville domains and Liouville manifolds. Anosov Liouville manifolds are finite type Liouville manifolds, and hence are Liouville isomorphic to the completions of Liouville domains obtained by truncating them. Conversely, some of our examples are Liouville domains (like the McDuff and torus bundle domains), whose completions are Anosov Liouville manifolds. 

We use Anosov Liouville manifolds in Definition \ref{def:alstructure} because an Anosov vector field has a canonically associated Liouville manifold, up to exact symplectomorphism. Due to various choices (like the metric and the smoothing), there is \textit{not} a canonical Liouville domain one could associate to an Anosov flow (and thereby obtain quantitative symplectic invariants from it). 
\end{remark}

By the following lemma, the wrapped Fukaya category of an Anosov Liouville manifold has a natural $\Z$-grading.

\begin{lemma} \label{lemma:symptriv}
The symplectic tangent bundle $(TV, \omega)$ of an Anosov Liouville manifold is trivial. Therefore, $c_1(TV, \omega) = 0$.
\end{lemma}

\begin{proof} 
Let $X$ be a non-singular vector field generating the underlying Anosov flow and $\theta$ be a smooth $1$-form on $M$ satisfying $\theta(X) = 1$. We define two vector fields $X_s$ and $X_\theta$ on $V=\R_s\times M$ by
\begin{align*}
    \omega(\cdot,X_s) = ds, \qquad
    \omega(\cdot,X_\theta) = \theta.
\end{align*}
From $\omega=d\lambda=ds\wedge(e^s\alpha_+-e^{-s}\alpha_-) + e^{-s} d\alpha_- + e^s d\alpha_+$, we get
$$\big(e^s \alpha_+ - e^{-s} \alpha_-\big)(X_s) = \omega(\partial_s,X_s)= 1,$$ 
so replacing $\theta$ with $\theta - \theta(X_s) \left(e^s \alpha_+ - e^{-s} \alpha_- \right)$, we can assume that $\theta(X_s) = 0$. 
Then, one easily checks that $\{\partial_s, X_s, X, X_\theta \}$ is a symplectic trivialization of $(TV, \omega)$.
\end{proof}

As mentioned in the Introduction, an Anosov Liouville manifold $(V = \R \times M, \lambda)$ supporting an Anosov flow $\phi_t$ admits many interesting \emph{exact} Lagrangian submanifolds obtained from closed orbits of the flow in the following way. Let $\Lambda \subset M$ be a simple closed orbit of $\phi_t$ and let $\mathcal L_\Lambda := \R \times \Lambda \subset V$. Since $T\Lambda \subset \xi_- \cap \xi_+$, $\Lambda$ is \emph{Legendrian} for both $\xi_-$ and $\xi_+$, and $\lambda_{|\mathcal L_\Lambda} \equiv 0$, so $\mathcal L_\Lambda$ is a \emph{strictly exact} Lagrangian submanifold of $V$. In particular, the Liouville vector field $X_\lambda$ is tangent to $\mathcal L_\Lambda$, and $\mathcal L_\Lambda$ is cylindrical. The boundary at infinity of $\mathcal L_\Lambda$ is simply given by two copies of $\Lambda$, one in each component of the boundary at infinity of $V$. Since $\phi_t$ possesses infinitely many simple closed orbits, we obtain infinitely many objects in the wrapped Fukaya category of $(V, \lambda)$. These Lagrangians are obviously spin.

It is also worth noting that for an Anosov Liouville structure $\lambda$ obtained from the standard construction, the leaves of the weak-stable foliation $\mathcal{F}^{ws} \subset \{0\} \times M$ are also strictly exact Lagrangians, i.e., $\lambda$ vanishes on $E^{ws}$ along $M_0 = \{0\} \times M$. Indeed, $\lambda = 2 \alpha_u$ along $M_0$ by~\eqref{eq:standard}.
However, due to classical results in Anosov theory \cite{V74}, the Lagrangians leaves of $\mathcal{F}^{ws}$ are copies of either $\mathbb{R}^2$ or a cylinder (depending on whether they contain a closed orbit or not). In particular they are not closed submanifolds, or properly embedded with Legendrian boundary, and so do \emph{not} define elements in the wrapped Fukaya category. Besides, the regularity of the foliation $\mathcal{F}^{ws}$ is no better than $\mathcal{C}^1$ in general. Nevertheless, if it is smooth, it is possible to smoothen $\lambda$ in such a way that the leaves of $\mathcal{F}^{ws}$ are still strictly exact Lagrangians.

\begin{remark}
In~\cite{M95} and~\cite{H22a}, $\lambda$ is defined as 
\begin{align} \label{eq:linearanosov} 
(1-t) \alpha_- + (1+t) \alpha_+
\end{align} on $[-1, 1]_t \times M$, yielding a Liouville domain rather than a Liouville manifold. We warn the reader that there are (at least) two notions of \emph{Liouville pairs} in the literature, one using a linear interpolation between $\alpha_+$ and $\alpha_-$ as in~\eqref{eq:linearanosov}, and the other using the exponential interpolation as in~\eqref{eq:exponentialanosov}. Although closely related, these two notions are \emph{different}: there exist pairs of contact forms $(\alpha_-, \alpha_+)$ such that $e^{-s} \alpha_- + e^s \alpha_+$ is a Liouville form on $\R_s \times M$, but $(1-t) \alpha_- +(1+t) \alpha_+$ is \emph{not} a Liouville form on $[-1,1]_t \times M$, see~\cite{Mas22}. In the present paper, we stick to the exponential version.
\end{remark}

\begin{remark}
In~\cite{Mas22}, the third author, building upon the work of Hozoori~\cite{H22a}, defines a space of pairs of contact forms called \emph{Anosov Liouville pairs}. Those are pairs of contact forms $(\alpha_-, \alpha_+)$ such that both $(\alpha_-, \alpha_+)$ and $(-\alpha_-, \alpha_+)$ are Liouville pairs (for the exponential version of the definition~\eqref{eq:exponentialanosov}). The space of Anosov Liouville pairs is homotopy equivalent to the space of (smooth) Anosov flows on a three manifold. Anosov Liouville pairs give rise to Anosov Liouville structures, but the converse is not necessarily true. Besides, the definition of Anosov Liouville pairs makes no reference to Anosov flows, as opposed to Anosov Liouville structures. In that sense, Anosov Liouville pairs are purely contact topological objects.
\end{remark}

        \subsection{McDuff domains}\label{sec:Mcduff_domains}

We now review the construction due to McDuff \cite{McD91}, corresponding to manifolds with $\mathrm{SL}(2,\mathbb{R})$ geometry. Let $\Sigma$ be a closed orientable surface of genus $g\geq 2$. Fix a hyperbolic metric on $\Sigma$ and let $\sigma$ be the  area form associated to the metric, with total area $-\chi(\Sigma)$. Consider the twisted cotangent bundle $$(T^*\Sigma, \omega_\sigma=d\lambda_{\mathrm{can}}+\pi^*\sigma),$$ 
where $\pi:T^*\Sigma\rightarrow \Sigma$ is the projection and $\lambda_{\mathrm{can}}=p\,dq$ is the canonical Liouville form. 
For $0<a < 1 < b$, consider the subdomain $$(V,\omega)=(\{p \in T^*\Sigma:a\leq |p|=r \leq b\}, \omega_{\sigma}\vert_V)\subset (T^*\Sigma,\omega_\sigma).$$ Since $\pi^*\sigma$ is exact away from the zero section, $\omega$ is an exact symplectic form on $V$, and one checks that the associated Liouville vector field points outwards at each boundary component. Therefore, $(V,\omega)$ is a Liouville domain of the form $([a,b]\times S^*\Sigma,d\lambda_{\sigma})$, where $\lambda_{\sigma}=\lambda_{\mathrm{can}}+\lambda_{\mathrm{pre}}$, with $d\lambda_{\mathrm{pre}}=\pi^*\sigma$. The $1$-form $\lambda_{\mathrm{pre}}$ is a prequantization form, i.e., a connection for the principal $S^1$-bundle $\pi:S^*\Sigma\rightarrow \Sigma$. The canonical contact structure is $\xi_\mathrm{can}=\ker \alpha_\mathrm{can}$ on $S^*\Sigma\subset T^*M$, and $\xi_\mathrm{pre}=\ker \alpha_\mathrm{pre}$ is the $S^1$-invariant prequantization contact structure on $S^*\Sigma\to\Sigma$, where we denote $\alpha_\mathrm{can}=\lambda_{\mathrm{can}}\vert_{S^*\Sigma}$ the standard contact form on $S^*\Sigma$, and $\alpha_\mathrm{pre}=\lambda_{\mathrm{pre}}\vert_{S^*\Sigma}$. 

We may understand this construction in terms of \emph{ideal Liouville domains} \cite[Section 4.2]{MNW13}. 
Observe that $\lambda_\sigma=r\alpha_\mathrm{can} + \alpha_\mathrm{pre},$ where $r \in (0,+\infty)$. If we denote $$M_r:=\{p \in T^*\Sigma: |p|=r\},$$ we have
$$
\omega_\sigma\vert_{M_r}=rd\alpha_\mathrm{can} +d\alpha_\mathrm{pre}=r\left(d\alpha_\mathrm{can}+\frac{1}{r}d\alpha_\mathrm{pre}\right).
$$

It follows that the characteristic line field of $M_r$ is the kernel of $d\alpha_\mathrm{can} +\frac{1}{r}d\alpha_\mathrm{pre}$, whose integral curves correspond to \emph{magnetic geodesics} of geodesic curvature $k_g\equiv 1/r$. This interpolates between the vertical flow generated by the $S^1$-fiber $R_\mathrm{pre}$ (the Reeb vector field of $\alpha_\mathrm{pre}$), and the geodesic flow generated by $R_\mathrm{can}$ (the Reeb vector field of $\alpha_\mathrm{can}$), going through the horocycle flow at $r=1$ (which is not a Reeb flow). See Figure \ref{fig:charflow}.

\begin{figure}
    \centering
    \includegraphics[width=1\linewidth]{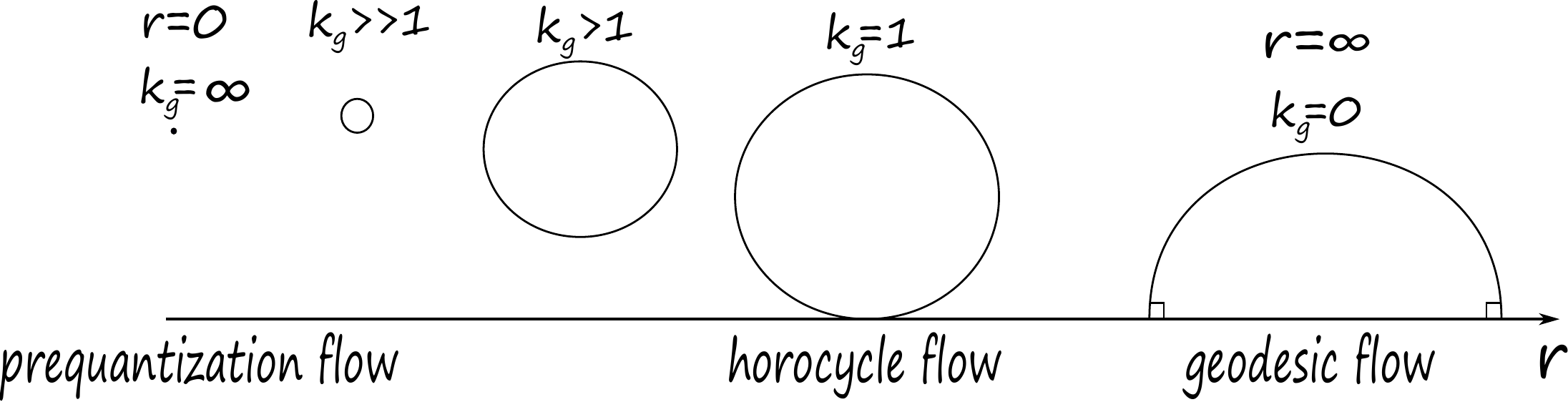}
    \caption{The characteristic flow at each $r$-slice $M_r$, viewed in the universal cover $\mathbb{H}$.}
    \label{fig:charflow}
\end{figure}

Identifying $S^*\Sigma$ with $\mathrm{PSL}(2,\R)\slash \pi_1(\Sigma)$, the contact forms $\alpha_\mathrm{can}$ and $\alpha_\mathrm{pre}$ can also be obtained as the quotient of two invariant contact forms on $\mathrm{PSL}(2,\R)$, see~\cite{Ge95} for more details. Alternatively, $\alpha_\mathrm{can}$ and $\alpha_\mathrm{pre}$ lift to contact forms $\widetilde{\alpha}_\mathrm{can}$ and $\widetilde{\alpha}_\mathrm{pre}$ on the unit cotangent bundle $S^*\mathbb{H}$ of the hyperbolic half-plane $\mathbb{H}$. The usual $\mathrm{PSL}(2,\R)$ action on $\mathbb{H}$ induces an action on $T\mathbb{H}$, which induces an action on $T^*\mathbb{H}$ using the musical isomorphism $T\mathbb{H} \cong T^*\mathbb{H}$ for the hyperbolic metric, and this action can be restricted to $S^*\mathbb{H} \subset T^*\mathbb{H}$. Here, we use the bundle metric on $T^*\mathbb{H}$ induced by the hyperbolic bundle metric on $T\mathbb{H}$ by the musical isomorphism. The forms $\widetilde{\alpha}_\mathrm{can}$ and $\widetilde{\alpha}_\mathrm{pre}$ are both invariant under this $\mathrm{PSL}(2,\R)$ action on $S^*\mathbb{H}$. Using the standard coordinates $(x,y) \cong x+iy =z$ on $\mathbb{H}$ and writing $(p_x, p_y) \cong p_x + ip_y = p \in \C$ for the dual coordinates, the bundle metric on $T^* \mathbb{H} \cong \mathbb{H} \times \C$ is simply $y^2 \vert dp\vert^2$. If $\varphi \in \R \slash 2 \pi \Z$ denotes the angular coordinate on the fibers for this metric, we have
\begin{align}
    \widetilde{\alpha}_\mathrm{can} &= \frac{1}{y}\big(\cos(\varphi) dx + \sin(\varphi) dy \big), \\
    \widetilde{\alpha}_\mathrm{pre} &= \frac{1}{y} dx + d\varphi. \label{eq:preinH}
\end{align}
Note that $d \widetilde{\alpha}_\mathrm{pre} = \frac{dx \wedge dy}{y^2}$ is the hyperbolic volume form on $\mathbb{H}$. Here we normalize the prequantization form so that it integrates to $2\pi$ along the fibers. It will be useful to write these forms using the \emph{Fermi coordinates} $(r, t)$ on $\mathbb{H}$ for the geodesic $\gamma : t \mapsto i e^t$ (see, e.g.,~\cite{Hu97}). Here, $r, t \in \R$ and in these coordinates, the hyperbolic metric becomes $dr^2 + \cosh(r)^2 dt^2 $. The usual $(x,y)$ coordinates are related to the Fermi coordinates by 
$$(x,y) = \left( \tanh(r) e^t, \frac{1}{\cosh(r)}e^t \right).$$ For a fixed $t_0 \in \R$, the points in $\mathbb{H}$ with Fermi coordinates $(r,t_0)$ for $r\in \R$ describe the geodesic in $\mathbb{H}$ orthogonal to $\gamma$ and passing through $\gamma(t_0)$. The Fermi coordinates also induce an angular coordinate $\theta \in \Z \slash 2\pi \Z$ on the fibers of $S^*\mathbb{H} \cong \R^2 \times S^1$,\footnote{The correspondence between $\varphi$ and $\theta$ is given by $\theta = \varphi + \arctan \sinh(r)$.} 
and the forms $\widetilde{\alpha}_\mathrm{can}$ and $\widetilde{\alpha}_\mathrm{pre}$ become
\begin{align}
    \widetilde{\alpha}_\mathrm{can} &= \cos(\theta) dr + \cosh(r) \sin(\theta) dt, \label{eqalphacanfermi}\\
    \widetilde{\alpha}_\mathrm{pre} &= \sinh(r) dt + d\theta. \label{eqalphaprefermi}
\end{align}
Recall that $\mathrm{PSL}(2,\R)$ is generated by the matrices
\begin{align*}
T_\tau := \begin{pmatrix} 1 & \tau \\ 0 & 1  \end{pmatrix}, \quad S:= \begin{pmatrix} 0 & 1 \\ -1 & 0  \end{pmatrix},
\end{align*}
where $\tau \in \R$. The action of $T_\tau$ on $S^*\mathbb{H}$ in the usual coordinates is given by 
$$ T_\tau \cdot (z, \varphi) = (z+\tau, \varphi),$$ 
and the action of $S$ in Fermi coordinates is given by $$ S \cdot (r,t,\theta) = (-r, -t, \theta + \pi).$$
It easily follows that the forms $\widetilde{\alpha}_\mathrm{can}$ and $\widetilde{\alpha}_\mathrm{pre}$ are invariant under the $\mathrm{PSL}(2,\R)$ action on $S^* \mathbb{H}$ and descend to $1$-forms on any manifold quotient of $S^* \mathbb{H}$, including $S^* \Sigma$. Moreover, writing $\alpha_+ = \alpha_\mathrm{can}$ and $\alpha_-= \alpha_\mathrm{pre}$, it follows from the above formulas that
\begin{align}
\alpha_+ \wedge d\alpha_+ &= - \alpha_- \wedge d\alpha_- > 0, \label{eq:geiges1}\\
\alpha_- \wedge d\alpha_+ &= \alpha_+ \wedge d\alpha_- = 0,  \label{eq:geiges2}
\end{align}
where the orientation of $S^*\Sigma$ comes from the canonical contact form $\alpha_\mathrm{can}$ (i.e., the fibers of $S^*\Sigma$ are \emph{negatively} oriented). In the terminology of~\cite{MNW13}, $(\alpha_-, \alpha_+)$ is a \emph{Geiges pair}. One can easily check that the Reeb vector fields $R_\pm$ satisfy $R_\pm \in \xi_{\mp}$. The (ideal) Liouville form $\lambda_\sigma = r \alpha_+ + \alpha_-$ does not quite match Definition~\ref{def:alstructure}, but it can be modified as follows. For $0 <a < 1 < b$, we consider the $1$-parameter family of $1$-forms $\{\lambda_t\}_{t \in [0,1]}$ on $[a,b] \times M$ defined by $$\lambda_t := r \alpha_+ + \left(\frac{t}{r} + 1-t \right)\alpha_-.$$ For $t=0$, it is simply $\lambda_\sigma$ and for $t=1$, it corresponds to $e^{-s} \alpha_- + e^s \alpha_+$ on $[\ln a, \ln b] \times M$ after the change of variable $r=e^s$. It easily follows from~\eqref{eq:geiges1} and ~\eqref{eq:geiges2} that $\{\lambda_t\}_t$ is a Liouville homotopy.
This deformation allows us to view the McDuff domain as $(V=[-1,1]\times M, \lambda=e^{-s} \alpha_- + e^s \alpha_+)$.

Alternatively, one can also write $$ \lambda_\sigma = r \alpha_+ + \alpha_- = \sqrt{r} \left( \sqrt{r} \alpha_+ + \frac{1}{\sqrt{r}} \alpha_-\right) = e^s \big( e^s \alpha_+ + e^{-s} \alpha_- \big) = e^s \lambda,$$ where we let $\sqrt{r} = e^s$. Note that the Liouville vector fields of $\lambda$ and $e^s \lambda$ are positively proportional to each other, and therefore share the same skeleton. Using the Geiges relations~\eqref{eq:geiges1} and~\eqref{eq:geiges2}, one can show that the Liouville vector field of $\lambda$ is of the form $$X_\lambda = \tanh(2s) \partial_s - \frac{1}{\cosh(2s)} X,$$ where $X$ is a suitable non-singular vector field on $M$ satisfying $X \in \xi_- \cap \xi_+$. In particular, $X$ is tangent to the underlying Anosov flow and the skeleton of $\lambda$ is exactly $\{0\} \times M$. Hence, the Liouville vector field $X_\lambda$ is tangent to the skeleton where it coincides with (the opposite of) the underlying Anosov flow. This implies that the skeleton of $\lambda_\sigma$ is exactly $S^*\Sigma \subset T^*\Sigma \setminus 0_\Sigma$.

More generally, one can take further manifold quotients of $S^*\mathbb{H}$ to obtain unit cotangent bundles of hyperbolic orbifold surfaces.

\medskip

\textbf{The underlying Anosov flow.} For the McDuff domains, the underlying Anosov flow generated by $\xi_-\cap \xi_+$ may be described as follows. Let $G\in \xi_-\subset TM$ be the vector field generating the geodesic flow, and let $X$ be the Anosov vector field spanning $\xi_-\cap \xi_+$. If we take a complex structure $J$ on $\Sigma$ which is compatible with the metric and the orientation, then we can lift it to $\xi_-$ so that $JG=X$, which is compatible with orientations. Alternatively, we can view $J$ as a diffeomorphism $J:M\rightarrow M$ whose differential satisfies $J_*G=X$, and so if $\Phi:\mathbb{R}\times M\rightarrow M$ denotes the geodesic flow, then the flow of $X$ is $\Psi_t=J\circ \Phi_t \circ J^{-1}$. Concretely, this means that the time-$t$ flow of $X$ starting at $(q,v)\in M$ is $(q,v)\mapsto \Psi_t(q,v)=J\Phi_t(q,J^{-1}v)$. Since the geodesic flow preserves angles, we see that an orbit of $X$ through $(q,v)\in M$ is the positive conormal lift of the geodesic $\gamma(t)=\pi(\Phi_t(q,J^{-1}v)) \in \Sigma$ that passes through $(q,J^{-1}v)$, where $\pi:M\rightarrow \Sigma$ is the bundle projection. 

\medskip 

\textbf{McDuff domains: exact Lagrangians.} The Lagrangians described in the Introduction have a precise description in this case, as the positive halves of the conormal bundles of oriented geodesics. Indeed, consider a closed orbit $\Lambda_\gamma \subset M$ of the Anosov flow, where $\gamma=\pi(\Lambda_\gamma)$ denotes its geodesic footprint in $\Sigma$. As we already noted, $\Lambda_\gamma$ is Legendrian for both contact structures $\xi_\pm$, hence 
$$
\mathcal{L}_\gamma=[-1,1]\times \Lambda_\gamma \subset V=[-1,1]\times M
$$
is an exact Lagrangian which has Legendrian boundary in $\partial V$. By construction, 
we have
$$\mathcal{L}_\gamma = \left\{ p \in \gamma^* T^*\Sigma :  p(\dot{\gamma})=0, \ \mathrm{det}\big(\dot{\gamma},  p^{\sharp}\big) > 0 \right\} \subset T^*\Sigma \setminus 0_\Sigma$$
the positive conormal bundle of the oriented geodesic $\gamma\subset \Sigma$, where $p^\sharp$ is the image of $p$ under the musical isomorphism $\sharp : T^*\Sigma \rightarrow T \Sigma$.

\medskip

\textbf{McDuff domains: weakly exact Lagrangian tori.} A Lagrangian $L \subset V$ is \emph{weakly exact} if $\omega \cdot \pi_2(V,L) =0$. While there are no weakly exact closed Lagrangians in the standard end of the McDuff domain (see Lemma~\ref{lemma:weakly}), we can find weakly exact Lagrangian tori in the prequantization end, viewed as the positive symplectization $([0,+\infty)\times S^*\Sigma, d(e^s \alpha_{\mathrm{pre}}))$ of the prequantization form. These tori then lie in the McDuff manifold away from the skeleton. Indeed, consider an embedded geodesic $\gamma\subset \Sigma$, and let 
$$
T_\gamma=S^*\Sigma\vert_\gamma \rightarrow \gamma
$$
viewed as an $S^1$-bundle over $\gamma$. To see that $T_\gamma$ is Lagrangian, we can take a collar neighbourhood $\mathcal{N}(\gamma)=\gamma \times [-\epsilon,\epsilon]\subset \Sigma$ of $\gamma$ with coordinates $(r,t)$, so that the prequantization form on $S^*\mathcal{N}(\gamma)=S^1\times \mathcal{N}(\gamma)$ looks like $\alpha_{\mathrm{pre}}=\sinh(r)dt +d\theta$. Then one readily sees that $d(e^s\alpha_{\mathrm{pre}})$ vanishes on $T_\gamma,$ which is parametrized by $(t,\theta)$ in these coordinates. Since $\gamma$ and the $S^1$-fiber are non-contractible in $V$, we see that $\pi_2(V,T_\gamma)=0$, which implies that $T_\gamma$ is weakly exact; note that $T_\gamma$ is not exact, as $\alpha_{\mathrm{pre}}$ is not exact along $T_\gamma$. However, in Section \ref{sec:closed_Lag_McDuff}, we will construct an exact Lagrangian torus in the isotopy class of $T_\gamma$.

        \subsection{Torus bundle domains}

We now describe the torus bundle domains, which correspond to closed quotients with Sol geometry.

On $\R^3$ with coordinates $(x,y,z)$, we consider the pair of $1$-forms
\begin{align} \label{form}
\alpha_\pm := \pm e^z dx + e^{-z}dy.
\end{align}
These form a Geiges pair (they satisfy~\eqref{eq:geiges1} and~\eqref{eq:geiges2}) whose Reeb vector fields are
\begin{align} \label{reeb}
R_\pm = \frac{1}{2}\left( \pm e^{-z} \partial_x + e^z \partial_y \right).
\end{align}
Moreover, $\alpha_\pm(R_\mp) = \alpha_\pm(\partial_z) = 0$ and so $\xi_\pm = \ker \alpha_\pm=\langle R_\mp, \partial_z\rangle$. On $\R^4  = \R_s \times \R^3$, we consider the $1$-form
\begin{align*}
\lambda_0 := &  e^{-s} \alpha_- + e^s \alpha_+ \\
= & 2\sinh(s) e^z dx + 2\cosh(s) e^{-z} dy.
\end{align*}
An elementary computation shows that $\w_0 := d\lambda_0$ is an exact symplectic form on $\R^4$. The Liouville vector field for $\lambda_0$ is
\begin{align} \label{liouv}
X_0 = \tanh(2s) \partial_s - \frac{1}{\cosh(2s)} \partial_z.
\end{align}

Let $A \in \mathrm{SL}(2,\Z)$ be a positive hyperbolic matrix (satisfying $\mathrm{tr}(A) > 2$). Then there exists $P \in \mathrm{SL}(2,\R)$ and $\nu \in \R \setminus \{0\}$ such that 
$$ PAP^{-1} = \begin{pmatrix}
e^{\nu} & 0\\
0 & e^{-\nu}
\end{pmatrix} =: D_\nu$$
and we have a commuting diagram
$$
\xymatrix{\R^2/\Z^2\times\R \ar[d]^{P\times\mathrm{id}} \ar[r]^\phi & \R^2/\Z^2\times\R \ar[d]^{P\times\mathrm{id}}\\
\R^2/P(\Z^2)\times\R \ar[r]^\psi & \R^2/P(\R^2)\times\R}
$$
with the maps $\phi(v,z)=(Av,z-\nu)$ and $\psi(v,z)=(D_\nu v,z-\nu)$. 
The $1$-forms $\alpha_\pm$ are invariant under translation by the lattice $P\big(\Z^2\big) \subset \R^2_{x,y}$ as well as under the map $\psi$, so they descend to the mapping torus $M=((\R^2/P(\Z^2)\times\R_z)/\psi$.
Hence, $\lambda_0$ descends to a Liouville  form $\lambda$ on $V =\R_s \times M$. The expression~\eqref{liouv} for the Liouville vector field is still valid in this quotient, so $\R \times M$ is a Liouville domain with disconnected end, and $[-1,1] \times M$ is a Liouville domain with disconnected boundary. The expression~\eqref{liouv} also implies that the skeleton of $(V, \lambda)$ is $\{0\} \times M$. The Liouville vector field is tangent to it, where it corresponds to (the opposite of) the underlying Anosov flow. Note that by~\eqref{reeb}, $R_\pm$ are tangent to the $\mathbb{T}^2$-fibers and restrict to linear flows on them.

By the diagram above, $M$ is diffeomorphic (via $P\times\mathrm{id}$) to the $\mathbb{T}^2$-bundle over $S^1$ with monodromy $A$ and the underlying Anosov flow is simply the suspension of the Anosov diffeomorphism of the torus induced by $A$. In this model, the Reeb vector fields take the form
\begin{align} \label{reeb2}
R_\pm(z) = \frac{1}{2}\left( \pm e^{-z} v_x + e^z v_y \right)
\end{align}
with the eigenvectors $v_x=P^{-1}\partial_x$ and $v_y=P^{-1}\partial_y$ of $A$ to the eigenvalues $e^{\pm\nu}$. 

\medskip

\textbf{Torus bundle domains: exact Lagrangians.} The Lagrangians $\mathcal{L}_\Lambda$ described in the Introduction also appear in the torus bundle domains, and admit a more explicit description. Indeed, the hyperbolic matrix $A \in \mathrm{SL}(2,\Z)$ naturally acts on $\mathbb{T}^2 = \R^2 \slash \Z^2$, and this action has a lot of periodic points. A finite orbit $\mathcal{O}$ for this action gives rise to a closed orbit $\Lambda=\Lambda_\mathcal{O}$ for the vector field $\partial_z$ on $M$. $\Lambda_\mathcal{O}$ is then a closed, connected Legendrian submanifold of $(M, \xi_\pm)$ and $$\mathcal L_\mathcal{O} = \mathcal{L}_\Lambda=[-1,1] \times \Lambda_\mathcal{O} \subset V$$ is an exact cylindrical Lagrangian submanifold of $(V, \lambda)$. 

\medskip

\textbf{Torus bundle domains: Lagrangian torus fibration.} As will be shown, $\left(V, \lambda \right)$ has no closed exact Lagrangian tori, yet it does have interesting non-exact but weakly exact Lagrangian tori. Namely, the torus fibration $M \rightarrow S^1$ gives rise to a torus fibration $V \rightarrow [-1,1] \times S^1$ whose fibers are Lagrangian. There is no clear analogue of these in the McDuff domains, as the corresponding plane distribution spanned by $\{R_-, R_+\}$ is contact. 

Let $\mathbb{T}_{s,z}$ denote the fiber over $(s,z) \in [-1,1] \times S^1$. The fibration homotopy exact sequence implies that $\pi_2(V, \mathbb{T}_{s,z})=0$, so these tori are weakly exact. Their Floer homology is well-defined \emph{over the Novikov ring $\Lambda_\mathbb{Z}$},\footnote{The Novikov ring $\Lambda_\mathbb{Z}$ is defined as the subring of $\Z \llbracket \R \rrbracket$ whose elements are of the form $\sum_i n_i t^{a_i}$, where $n_i \in \Z$, $a_i \in \R$, and $a_1 < a_2 < \dots$, $a_i \rightarrow +\infty$.} invariant under Hamiltonian isotopies, and straightforward to compute: $HF^*(\mathbb{T}_{s,z}) = H^*(\mathbb{T}^2 ; \Z) \otimes \Lambda_\mathbb{Z}$, and $HF^*(\mathbb{T}_{s,z}, \mathbb{T}_{s',z'})=0$ if $(s,z) \neq (s',z')$. Therefore, these Lagrangian tori are pairwise non-Hamiltonian isotopic. Moreover, the Floer homology of $\mathcal{L}_\mathcal{O}$ with any $\mathbb T_{s,z}$ is independent of $(s,z)$, given by  
$$
HF^*(\mathcal{L}_\mathcal{O},\mathbb T_{s,z})=\bigoplus_{p \in \mathcal{O}} \Lambda_\mathbb{Z}\cdot p,
$$
since $\mathcal{L}_\mathcal{O}\cap\mathbb T_{s,z}=\mathcal{O}$ and it is easy to show that there are no non-trivial Floer strips between these intersection points. Indeed, consider a Floer strip $u:D\to\R\times M$ joining $p$ to $p'$ with lower boundary $u_-$ on  $\mathcal{L}_\mathcal{O}$ and upper boundary $u_+$ on $\mathbb{T}_{s,z}$. Its composition $v$ with the projection $\R\times M\to M$ has lower boundary $v_-$ on $\Lambda_\mathcal{O}$ and upper boundary $v_+$ on $\mathbb{T}_{s,z}$.
Note that the projection $\pi:M\to S^1$ maps $\mathbb{T}_{s,z}$ to one point and restricts to a finite covering $\Lambda_\mathcal{O}\to S^1$. Thus, $\pi\circ v_-$ is a contractible loop in $S^1$, and therefore the endpoints of its lift $v_-$ agree, $p=p'$. It follows that $v_+$ is a contractible loop at $p$ in $\mathbb{T}_{s,z}$, so $u$ is topologically trivial, and therefore constant, and is not counted in the differential.

        \subsection{Homotopy classes of Reeb orbits}

The following lemma will be useful for computing the product structure on the symplectic cohomology of the McDuff and torus bundle domains.

\begin{lemma}\label{lem:hypertight}
Let $V=[-1,1]\times M$ be a McDuff domain or a torus bundle domain. Then the free homotopy classes of positively or negatively parametrized closed Reeb orbits on $M_+$ are distinct from those on $M_-$.
\end{lemma}

\begin{proof}
Consider first a McDuff domain modeled on $S^*\Sigma$, where $\Sigma$ is a closed oriented surface of genus $g \geq 2$. Closed Reeb orbits in $M_+$ are lifts of closed geodesics on $\Sigma$, so their free homotopy classes project non-trivially onto $\Sigma$, in particular they are nontrivial. The free homotopy classes of positively or negatively parametrized closed Reeb orbits in $M_-$ are nonzero multiples of the fibre class, so they are also nontrivial and distinct from the classes of Reeb orbits in $M_+$. 

Consider next a torus bundle domain $M$ associated to the hyperbolic monodromy matrix $A$. The lifts of the Reeb vector fields to $\mathbb{T}^2 \times \R$ are $\widetilde{R}_\pm=\frac{1}{2} \big(\pm e^{-z}v_x+e^{z}v_y \big)$, which are tangent to the torus fibers spanned by $v_x,v_y$ (the chosen eigenvectors of $A$). Therefore, the closed $R_\pm$-orbits arise whenever $\widetilde{R}_\pm$ has rational slope. From this we see that the free homotopy classes of positively or negatively parametrized Reeb orbits in $M_+$ and $M_-$ are nontrivial and distinct.
\end{proof}

    \section{Open-closed map and non split-generation}

In this section, we prove Theorem \ref{thm:unit} and  Theorem~\ref{thm:nonsplit} from the Introduction. Both will be consequences of the following elementary topological lemma:

\begin{lemma}[Topological Disk Lemma] \label{lemma:disk}
Let $X$ be a $\mathcal{C}^1$ Anosov vector field on a closed $3$-manifold $M$ with oriented weak-stable foliation $\mathcal{F}^{ws}$. Then there exists no continuous map $u : \overline{\mathbb{D}} \rightarrow M$ satisfying the following conditions (see Figure \ref{fig:topological_disk}):
\begin{itemize}
    \item There exist $n \geq 1$ and $a_1 < b_1 < a_2 < \dots < a_n < b_n < a_{n+1}=a_1$ cyclically ordered points on $\partial{\overline{\mathbb{D}}}$ such that $u$ is $\mathcal{C}^1$ on $\partial \overline{\mathbb{D}} \setminus \{a_1, b_1, \dots, a_n, b_n\}$,
    \item $u$ is tangent to $X$ on $(a_i, b_i)$ for every $1 \leq i \leq n$, 
    \item $u$ is positively transverse to $\mathcal{F}^{ws}$ on $(b_i, a_{i+1})$ for every $1 \leq i \leq n$.
\end{itemize}
\end{lemma}

\begin{figure}
    \centering
    \includegraphics[width=0.5\linewidth]{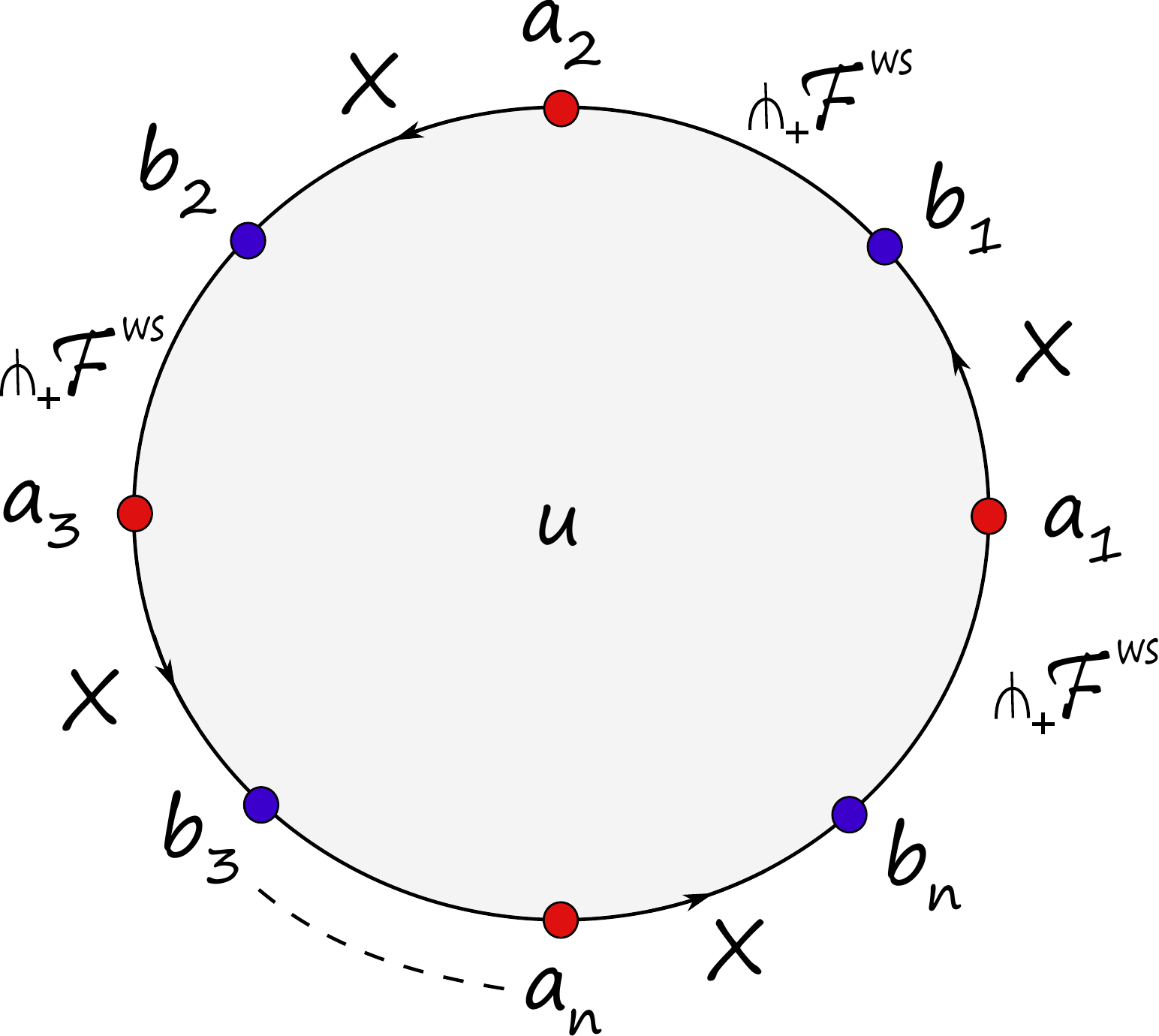}
    \caption{The tautness of the weak-stable foliation $\mathcal{F}^{ws}$ prevents the existence of disks as above.}
    \label{fig:topological_disk}
\end{figure}

\begin{proof}
We argue by contradiction. Let $u : \overline{\mathbb{D}} \rightarrow M$ be such a disk and let $\gamma := u_{| \partial \mathbb{D}}$ be its boundary loop. On each interval $(a_i, b_i)$ of $\partial \overline{\mathbb{D}}$, $\gamma$ is tangent to $X$ and can be modified so that it becomes positively transverse to $\mathcal{F}^{ws}$. Indeed, it suffices to push it along the weak-unstable foliation and smoothen it, see Figure \ref{fig:smoothing}. The new loop $\gamma'$ obtained this way is still contractible and positively transverse to $\mathcal{F}^{ws}$ \emph{everywhere}, which is impossible since $\mathcal{F}^{ws}$ is taut. 
\end{proof}

\begin{figure}
    \centering
    \includegraphics[width=0.8\linewidth]{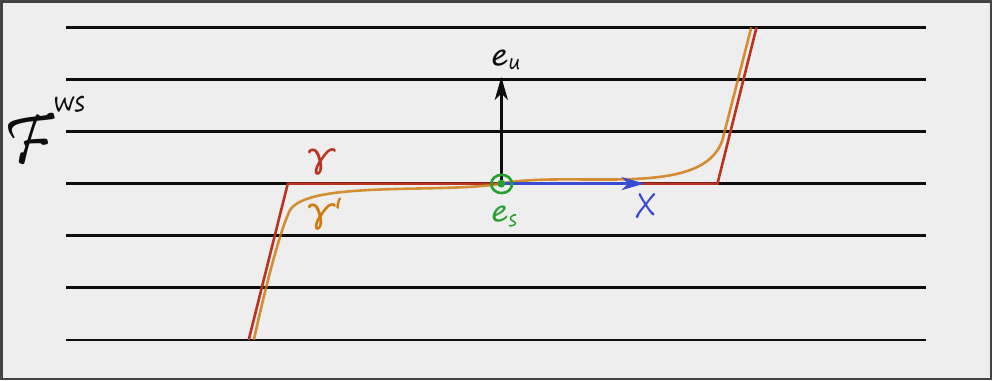}
    \caption{The smoothened pushoff $\gamma'$ of $\gamma$, transverse to $\mathcal{F}^{ws}$.}
    \label{fig:smoothing}
\end{figure}

\begin{proof}[Proof of Theorem \ref{thm:unit}]
The models that we use for the wrapped Fukaya category, the symplectic cohomology and the open-closed map are the ones described in~\cite{A10} (see also~\cite{G13}). As explained in Section \ref{sec:anosov_Liouville_manifolds}, we can assume that the Reeb vector field on $\partial_\infty V$ is positively transverse to $\mathcal{F}^{ws}$ on both components of $\partial_\infty V$, and this property will still hold after a small perturbation of the Liouville form on the boundary (to ensure non-degeneracy of the Reeb chords and orbits). The generators of $CW^*(\mathcal{L}_\Lambda, \mathcal{L}_{\Lambda'})$ are given by:
\begin{enumerate}
    \item Hamiltonian chords from (a small cylindrical Hamiltonian perturbation of) $\mathcal{L}_\Lambda$ to (a small cylindrical Hamiltonian perturbation of) $\mathcal{L}_{\Lambda'}$ for a Hamiltonian that is quadratic at infinity; these correspond to Reeb chords on the contact boundary, and
    \item if $\Lambda = \Lambda'$, two intersection points between $\mathcal{L}_\Lambda$ and a small Hamiltonian perturbation of $\mathcal{L}_{\Lambda'}$ corresponding to a small perfect Morse function on the cylinder.
\end{enumerate}
We can further assume that all the generators of type (2) lie in a fixed compact region of $V$, and all the generators of type (1) lie away from this region.

By definition, the underlying chain complex of $\HH_*(\mathcal{W}_0)$, denoted by $CC_* = CC_*(\mathcal{W}_0)$, is generated by elements of the form $a_1 \otimes a_2 \otimes \dots \otimes a_d$ where $a_i \in CW^*(\mathcal{L}_{\Lambda_i}, \mathcal{L}_{\Lambda_{i+1}})$.\footnote{Some authors use the reverse ordering on the string of generators in order to simplify some sign computations. It won't be relevant here and we opt for the easiest ordering.} We define $CC_*^c$ as the subgroup of $CC_*$ generated by strings only involving intersection points as in (2) above, and $CC_*^{nc}$ as the subgroup of $CC_*$ generated by strings involving at least one chord as in (1) above. Recall that the differential of $CC_*$ on a generator $a_1 \otimes \dots \otimes a_d$ has the form: 
\begin{align} \label{eqdCC}
\begin{multlined}
    d_{CC}(a_1 \otimes \dots \otimes a_d) = 
    \sum_{i+j \leq d} \pm \, a_i \otimes \dots \otimes a_{i+j} \otimes \mu^{d-j-1}(a_{i+j+1} , \dots , a_d , a_1 , \dots , a_{i-1}) \\
    + \sum_{i+j \leq d} \pm \, a_1 \otimes \dots \otimes a_{i-1} \otimes \mu^{j+1}(a_i, \dots, a_{i+j}) \otimes a_{i+j+1} \otimes \dots \otimes a_d,
\end{multlined}
\end{align}
where the $\mu^k$'s denote the $A_\infty$ products on $\mathcal{W}$. We decompose the proof into four steps.

\medskip
\textit{Step 1.} $CC_*^c$ and $CC_*^{nc}$ are subcomplexes of $CC_*$.

Let $a_1 \otimes \dots \otimes a_d$ be a generator of $CC_*^c$. By definition, all of the $a_i$'s are intersection points between small perturbations of a single $\mathcal{L}_\Lambda$. A maximum principle argument implies that the $A_\infty$ product $\mu^{d}(a_1, \dots, a_d)$ is a linear combination of intersection points only (no chords are involved). This shows that $CC_*^c$ is a subcomplex of $CC_*$. Let's now assume that $a_1 \otimes \dots \otimes a_d$ is a generator of $CC_*^{nc}$, so that one of the $a_i's$ is a chord. For simplicity, let's assume that $a_1$ is a chord. Then all of the terms in the right hand side of~\eqref{eqdCC} involve either $a_1$ or an $A_\infty$ product of $a_1$ with other generators. Hence, it is enough to establish the following property: 
\begin{center}
    \emph{(C) The $A_\infty$ product of generators $a_1, \dots, a_d$, one of which is a chord, is a linear combination of chords (no self-intersections are involved).}
\end{center}
Assume by contradiction that there exist generators $a_1, \dots, a_d$, one of which is a chord, such that the product $\mu^d(a_1, \dots, a_d) \in CW^*(\mathcal{L}_{\Lambda_1}, \mathcal{L}_{\Lambda_{d+1}})$ is not a linear combination of chords only (in particular, $\mathcal{L}_{\Lambda_1} = \mathcal{L}_{\Lambda_{d+1}}$). By definition of the $A_\infty$ products, this implies that there exists a Floer disk with positive boundary punctures asymptotic to $a_1, \dots, a_d$, and a negative boundary puncture asymptotic to a self-intersection point $b$. We can undo the perturbations of the Lagrangians $\mathcal{L}_{\Lambda_i}$ and reverse the small Hamiltonian isotopies, making the self-intersections disappear, and turn the asymptotics into boundary components. After projecting the new disk to $M$, we obtain a topological disk whose boundary is alternatively tangent to the Anosov flow (coming from the Lagrangian boundary condition) or to a Reeb orbit (coming from the asymptotics). Since the Reeb flow is positively transverse to $\mathcal{F}^{ws}$, this disk satisfies the assumptions of Lemma~\ref{lemma:disk}, which is impossible. This proves Property (C) and finishes this step.

\medskip
\textit{Step 2.} The image of the restriction of $\mathcal{OC}_0$ to $\HH_*^{nc}$ is contained in $SH^{*+2}_{nc}$.

Assume by contradiction that there exist generators $a_1, \dots, a_d$, one of which is a chord, such that $\mathcal{OC}_0(a_1 \otimes \dots \otimes a_d)$ is not in $SC^{*+2}_{nc}$ (the subcomplex of $SC^{*+2}$ generated by non-contractible Hamiltonian orbits). By definition of the open-closed map, this implies that there exists a Floer disk with boundary punctures asymptotic to $a_1, \dots, a_d$ and an interior puncture asymptotic to a contractible closed Hamiltonian orbit $x$. We cap this puncture and perform the same modifications as in Step 1 to obtain a topological disk satisfying the assumptions of Lemma~\ref{lemma:disk}, which again is impossible.

\medskip
\textit{Step 3.} The image of the restriction of $\mathcal{OC}_0$  to $\HH_*^c$ is contained in $SH^{*+2}_c = H^{*+2}(M)$.

Since both components of the contact boundary of $V$ are hypertight and remain hypertight after a small perturbation (since the Reeb vector fields on both components are positively transverse to $\mathcal{F}^{ws}$), for a careful choice of Floer data for the definition of $SH^*$, the subcomplex $SC^*_c$ of $SC^*$ coincides with the Morse chain complex of a Morse function on $M$, and all the generators of $SC^*_{nc}$ are contained in the cylindrical ends of $V$. Recall that the generators of $CC_*^c$ involve Lagrangian intersection points contained in a compact subset of $V$ away from the cylindrical ends, so this step is a consequence of~\cite[Lemmas 2.2 and 2.3]{CO18}.

\medskip
\textit{Step 4.} The image of $\mathcal{OC}_0^c$ is contained in $H^2(M; \Z) \oplus H^3(M;\Z)$.

Since the Lagrangians $\mathcal{L}_\Lambda$ are disjoint,  $\HH^c_*$ decomposes as
\begin{align} \label{HHsplitting}
    \HH^c_* = \bigoplus_{\Lambda} \HH_*(C^*(\mathcal{L}_\Lambda)) \cong \bigoplus_{\Lambda} \HH_*(C^*(S^1)),
\end{align}
where $C^*(\mathcal{L}_\Lambda)$ is the dg-algebra of singular cochains on $\mathcal{L}_\Lambda$ and the direct sums run over the simple closed orbits of the Anosov flow. Here, we are using the $A_\infty$ equivalence between the low energy part of the Floer $A_\infty$-algebra $CW^*(\mathcal{L}_\Lambda, \mathcal{L}_\Lambda)$ with $C^*(\mathcal{L}_\Lambda)$ due to Abouzaid~\cite[Theorem 1]{A11b} in the case of closed Lagrangians, which can easily be extended to our setting by the maximum principle. Finally, $W_* := \HH_*(C^*(S^1))$ is supported in degrees $0$ and $1$,\footnote{There is an isomorphism of dg-algebras $C^*(S^1) \cong \Z[x] \slash x^2 =: R$ where $\mathrm{deg}(x) =1$ and $R$ has no differential. Hence, its Hochschild homology is the Hochschild homology of its underlying graded algebra. To compute it, we can use a graded version of the periodic resolution of $R$ as a $R^e$-module, where $R^e := R \otimes_\Z R$ (see~\cite[Exercise 9.1.4]{W97}). Setting $u := x \otimes 1 - 1 \otimes x \in R^e$ and $v := x \otimes 1 + 1 \otimes x \in R^e$, this resolution is $$ \dots \overset{v}{\longrightarrow} R^e \overset{u}{\longrightarrow} R^e \overset{v}{\longrightarrow} R^e \overset{u}{\longrightarrow} R^e \overset{\mu}{\longrightarrow} R \longrightarrow 0,$$ where $\mu : R^e \rightarrow R$ denotes the product. Here, we use \emph{cohomological} grading conventions. After tensoring the resolution with $R$ over $R^e$, we obtain the periodic complex
$$ \dots \overset{2x}{\longrightarrow} R \overset{0}{\longrightarrow} R \overset{2x}{\longrightarrow} R \overset{0}{\longrightarrow} R.$$
This complex computes $\HH_*(R,R)$ and is supported in degrees $0$ and $1$ (recall that we are computing the Hochschild homology of a \emph{graded} algebra). Besides, $\HH_*(R,R)$ has infinite rank.} and Step $4$ follows.
\end{proof}

\begin{remark}
With more work, we expect that it is possible to show that $$\mathrm{Im}(\mathcal{OC}^c_0) = \Gamma \oplus H^3(M; \Z),$$
where $\Gamma \subseteq H^2(M;\Z)$ is generated by the Poincar\'e duals of the closed orbits of the Anosov flow.\footnote{By a theorem of Fried~\cite{F82}, $\Gamma = H^2(M;\Z)$ when the Anosov flow is transitive (in particular, when it is volume preserving).} Indeed, for every simple closed orbit $\Lambda$ of the flow, the composition $$H^*(\Lambda;\Z) = H^*(\mathcal{L}_\Lambda;\Z) \rightarrow \HH_*(C^*(\mathcal{L}_\Lambda)) \overset{\mathcal{OC}_0}{\longrightarrow} SH^{*+2}_c = H^{*+2}(M;\Z)$$
is the shriek map $H^*(\Lambda;\Z) \overset{i^!}{\rightarrow} H^{*+2}(M;\Z)$ induced by the inclusion $\Lambda \subset M$ and its image is exactly $\langle \mathrm{PD}([\Lambda]) \rangle \oplus H^3(M;\Z)$. Therefore, $\Gamma \oplus H^3(M;\Z)$ is contained in the image of $\mathcal{OC}_0$. Moreover, $\mathcal{OC}_0 : \HH_*(C^*(\mathcal{L}_\Lambda)) \rightarrow  H^{*+2}(M;\Z)$ should factor as $$\HH_*(C^*(\mathcal{L}_\Lambda)) \rightarrow H^*(\mathcal{L}_\Lambda ; \Z) \overset{i^!}{\rightarrow} H^{*+2}(M;\Z),$$
where the first map corresponds to a map $\HH_*(C^*(S^1)) \rightarrow H^*(S^1; \Z)$ obtained by composing the ``Jones map'' $\HH_*(C^*(S^1))  \rightarrow H^*(LS^1, \Z)$, where $L S^1$ denotes the free loop space of $S^1$, with the ``restriction to constant loops'' $H^*(L S^1 ; \Z) \rightarrow H^*(S^1 ; \Z)$. Hence, the image of $\mathcal{OC}_0$ should be contained in $\Gamma \oplus H^3(M;\Z)$. To establish such a factorization, one could degenerate the holomorphic annuli involved in the definition of $\mathcal{OC}_0$ to some Morse flow-trees by sending the Hamiltonian data to $0$, and obtain a purely Morse-theoretic model for $\mathcal{OC}_0$. In this model, the desired factorization should be more manifest.
\end{remark}

\begin{remark}[Non-contractible component]\label{rem:non_contractible}
It is possible to obtain information on the non-contractible summand $\mathcal{OC}_0^{nc}: \HH^{nc}_{*-2}\rightarrow SH^*_{nc}(V)$ of the open-closed map $\mathcal{OC}_0$, by using the fact that $\mathcal{OC}_0$ is a module map for suitable actions of $SH^*(V)$, coinciding with the pair of pants product on the target (see \cite[Section 8]{RS17}). Namely, consider the splitting
$$
SH^*(V)=\bigoplus_{\beta \in \pi^{\mathrm{free}}_1(V)}SH^*_\beta(V)=SH^*_c(V)\oplus SH^*_{nc}(V) 
$$
where $SH^*_\beta(V)$ is generated by Hamiltonian orbits in the free homotopy class $\beta$, and $SH^*_{nc}(V)=\bigoplus_{\beta \neq 0}SH^*_\beta(V)$. The pair of pants product satisfies $$SH^*(V)_{nc}\cdot SH^*(V)_{c} \subseteq SH^*(V)_{nc}.$$ It follows that 
\begin{equation}\label{eq:image}
SH_{nc}^*(V)\cdot \mbox{Im}(\mathcal{OC}_0^c)\subseteq \mbox{Im}(\mathcal{OC}_0^{nc}). 
\end{equation}
Moreover, using Lemma \ref{lemma:disk}, one can explicitly check that the action of $SH^*(V)$ on $\HH_{*-2}$ is compatible with the splittings, i.e., 
\begin{align*}
SH_c^*(V)\cdot \HH_{*-2}^c&\subseteq \HH_{*}^{c}, \\
SH_c^*(V)\cdot \HH_{*-2}^{nc}&\subseteq \HH_{*}^{nc}, \\
SH_{nc}^*(V)\cdot \HH_{*-2}^{c}&\subseteq \HH_{*}^{nc}.
\end{align*}
Then, we have
\begin{equation}\label{eq:kernel} 
SH^*_{nc}(V)\cdot \ker \mathcal{OC}_0^{c} \subseteq \ker \mathcal{OC}_0^{nc}.
\end{equation}
Note that $SH^*_{nc}(V)=SH_-^*(V)\oplus SH_+^*(V)$, $SH^*_c(V)=H^*(M)$, by hypertightness of the contact forms on $M_\pm$. Under the further assumption that the free homotopy classes of orbits in $M_+$ differ from those in $M_-$, by Theorem~\ref{thm:RFH-SH} we have 
$$
SH^*_{\pm,\beta}(V)\cdot \mbox{Im}(\mathcal{OC}_0^c)\subseteq \mbox{Im}(\mathcal{OC}_0^{nc}) \cap SH_{\pm,\beta}^*(V), 
$$
where $SH^*_{\pm,\beta}(V)$ is generated by orbits in $M_\pm$ in the non-trivial free homotopy class $\beta\neq 0$. Note that there are no invertible elements in $SH_{nc}^*(V)$, as $SH_\pm^*(V)$ are subrings by Theorem~\ref{thm:RFH-SH}, so that we may not take inverses in the inclusions (\ref{eq:image}) and (\ref{eq:kernel}). If the assumption on homotopy classes is dropped, there could \emph{a priori} be invertible elements in $SH^*_{nc}(V)$. If this is the case, these are necessarily not in the image of $\mathcal{OC}_0^{nc},$ as this is easily seen to contradict Theorem \ref{thm:unit}.
\end{remark}

\begin{proof}[Proof of Theorem~\ref{thm:nonsplit}] Let $\mathcal{A}$ be a full subcategory of $\mathcal{W}_0(V)$ corresponding to a collection $\mathcal{C}$ of simple closed orbits of the Anosov flow, and $L = \mathcal{L}_\Lambda$, for a simple closed orbit $\Lambda \notin \mathcal{C}$. Let $\mathcal{A}'$ be the full subcategory of $\mathcal{W}_0(V)$ generated by $\mathcal{A}$ and $L$, and $\mathcal{C}' := \mathcal{C} \cup \{ \Lambda\}$. It follows from the proof of Theorem~\ref{thm:unit} that $\HH_*(\mathcal{A})$ splits as $$\HH_*(\mathcal{A}) = \HH_*^{c}(\mathcal{A}) \oplus \HH_*^{nc}(\mathcal{A}),$$ and furthermore $$\HH_*^c(\mathcal{A}) = \bigoplus_{\Lambda \in \mathcal{C}} W_*,$$ and similarly for $\HH_*(\mathcal{A}')$. Moreover, the natural map $\iota_* : \HH_*(\mathcal{A}) \rightarrow \HH_*(\mathcal{A}')$ induced by the inclusion $\mathcal{A} \subset \mathcal{A}'$ splits as a sum of a map $\HH_*^{nc}(\mathcal{A}) \rightarrow \HH_*^{nc}(\mathcal{A}')$ and the map $\HH_*^c(\mathcal{A}) \rightarrow \HH_*^c(\mathcal{A}) \oplus W_* = \HH_*^c(\mathcal{A}')$ given by the inclusion of the first factor 
due to the fact that $\mathcal{L}_\Lambda$ is disjoint from the other Lagrangians in $\mathcal{A}$. Hence, $\iota_*$ is not an isomorphism since $W_*$ is non-zero. If $L$ were split-generated by $\mathcal{A}$, then $\mathcal{A}'$ would be split-generated by $\mathcal{A}$ and $\iota_*$ would be an isomorphism (see for instance~\cite[Lemma 5.2.]{GPS1}).

(1) follows immediately since for $\mathcal{A} := \{\mathcal{L}_\Lambda\}$, $\mathcal{L}_{\Lambda'} \notin \mathcal{A}$ is not split-generated by $\mathcal{L}_\Lambda$ and is a fortiori not isomorphic to $\mathcal{L}_\Lambda$. For (2), if $\mathcal{A}$ is a finite collection of objects of $\mathcal{W}_0$, there exists an object $\mathcal{L}_\Lambda \in \mathcal{W}_0 \setminus \mathcal{A}$, thus $\mathcal{L}_\Lambda$ is not split-generated by $\mathcal{A}$. Finally, a homologically smooth $A_\infty$-category whose morphism chain complexes are \emph{cofibrant} is necessarily finitely split-generated\footnote{See~\cite[Section 3.1.1]{GPS1} for the definition and relevance of the cofibrancy assumption; it is used in~\cite[Lemma A.2]{GPS2} to show the $A_\infty$ version of the Yoneda lemma, which implies that perfect modules are compact objects in the module category, see~\cite[Section A.3]{GPS3}. Then, the argument of~\cite[Lemma 3.14]{R08}, which is stated for field coefficients, can easily be adapted to show finite split-generation for $\Z$ coefficients.} and (3) follows. 
\end{proof}

\section{Closed Lagrangian submanifolds} 

We now investigate the existence of closed (weakly) exact Lagrangians in the McDuff and torus bundle domains. Note that any closed oriented Lagrangian $L$ in a four dimensional Liouville domain of the form $[-1,1] \times M$ is necessarily a $2$-torus. Indeed, $$0 = [L] \cdot [L] = - \chi(L),$$ where the first equality follows from the (smooth) displaceability of $L$ and the second equality follows from the Weinstein neighborhood theorem.

\subsection{Closed Lagrangians in the McDuff domains}\label{sec:closed_Lag_McDuff}

In this section, we show that the McDuff domains or manifolds contain \emph{exact} embedded Lagrangian tori, one for each embedded closed geodesic. The idea behind the construction is due to Georgios Dimitroglou Rizell.

Let $\gamma$ be an oriented simple closed geodesic in $\Sigma$ of length $\ell$ and $S^1_\ell = \R \slash \ell \Z$.

\begin{lemma}
For $\delta>0$ sufficiently small, there exists a neighborhood $U_\delta$ of $T^*\Sigma_{|\gamma} \setminus 0_\Sigma$ diffeomorphic to $V_\delta:=T^*C_\delta \setminus 0_{C_\delta} \cong C_\delta \times \left(\R^2 \setminus \{0\}\right)$, where $C_\delta=(-\delta, \delta)\times S^1_\ell$, with coordinates $(r, t, x, y)$, in which the McDuff Liouville form becomes
\begin{align} \label{lambdageod}
\lambda_C = x dr + \big(\cosh(r) y + \sinh(r)\big) dt + d\theta,
\end{align}
where $\theta$ denotes the angular coordinate in the $(x,y)$-plane.
\end{lemma}

\begin{proof}
For $\delta>0$ sufficiently small, there exists a neighborhood of $\gamma$ in $\Sigma$ diffeomorphic to $C_\delta=(-\delta, \delta) \times S^1_\ell$ with Fermi coordinates $(r,t)$ in which the hyperbolic metric writes $dr^2 + \cosh(r)^2 dt^2$, see Section~\ref{sec:Mcduff_domains}. If $(\rho, \tau)$ denotes the dual coordinates of $(r,t)$, the canonical Liouville form on $T^*C_\delta$ is $\lambda_\mathrm{can} = \rho dr+ \tau dt$. Changing the fiber coordinates to $(x,y):= \big(\rho, \frac{\tau}{\cosh(r)}\big)$, the bundle metric on $T^*C_\delta$ becomes the standard Euclidean metric $dx^2 + dy^2$ and by~\eqref{eqalphaprefermi}, the preqantization contact form writes $\alpha_\mathrm{pre} = \sinh(r)dt + d\theta$, where $\theta$ is the angular coordinate in the $(x,y)$-plane. Recall that the McDuff Liouville form is defined as $\lambda_\mathrm{can} + \alpha_\mathrm{pre}$, so~\eqref{lambdageod} follows.
\end{proof}

We now construct exact Lagrangian tori in $(V_\delta, \lambda_C)$ using the above coordinates. Let $\epsilon :=\tanh(\delta) < 1$ and let $\beta = (f,g) : S^1 \rightarrow \R^2 \setminus \{0\}$ be a smooth Jordan curve satisfying the following assumptions:
\begin{enumerate}
    \item $\beta$ winds around $0$ once positively,
    \item $\beta$ is contained in the strip $\R \times (-\epsilon, \epsilon)$,
    \item The domain $D$ bounded by $\beta$ in $\R^2$ satisfies $$ \int_D \frac{dx \wedge dy}{1-y^2}\, = 2 \pi.$$
\end{enumerate}
See Figure \ref{fig:beta}.
\begin{figure}
    \centering
    \includegraphics[width=0.5\linewidth]{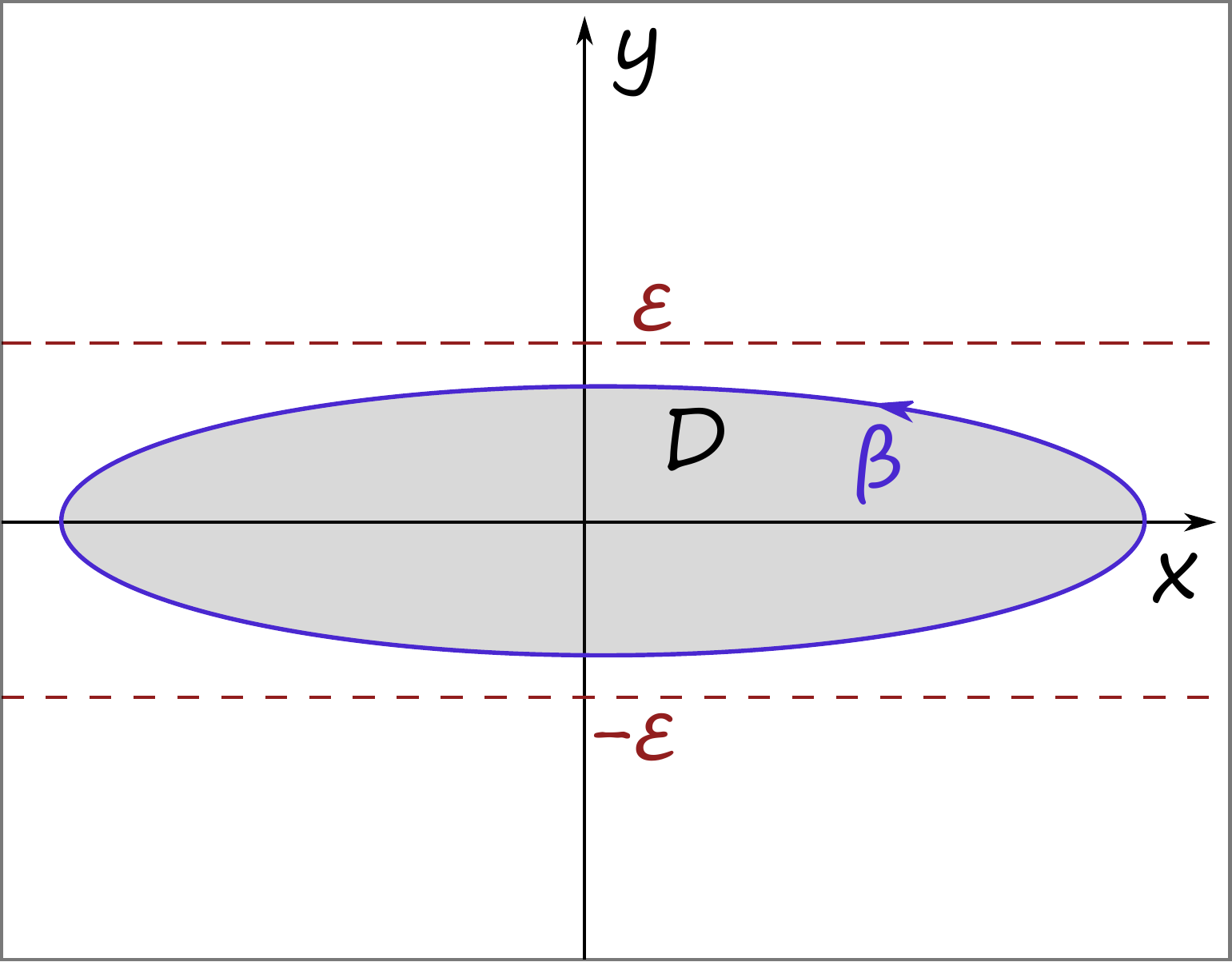}
    \caption{The Jordan curve $\beta$.}
    \label{fig:beta}
\end{figure}
We define an embedding $h : \mathbb{T}^2 \cong S^1_\ell \times S^1 \hookrightarrow V_\delta$ by $$h(t_1, t_2) := \left( -\tanh^{-1}(g(t_2)), t_1, \beta(t_2)\right).$$ Note that since $|g(t_2)| < \epsilon < 1$, $h$ is well-defined.

\begin{lemma}
The embedding $h$ is an exact Lagrangian embedding.
\end{lemma}

\begin{proof}
A straightforward computation\footnote{Recall that $\frac{d}{dt} \tanh^{-1}(t) = \frac{1}{1-t^2}$.} shows 
\begin{align*}
h^* \lambda_C = - (\beta\circ \pi_2)^* \mu + (\beta \circ \pi_2)^* d\theta,
\end{align*}
where $\mu := \frac{x}{1-y^2}\, dy$ and $\pi_2 : S^1_\ell \times S^1 \rightarrow S^1$ is the projection onto the second factor. Note that the $1$-form $h^* \lambda_C$ is closed, and it is exact if and only if its integral along $\{0\} \times S^1$ is $0$. By Stokes' theorem, the latter is equivalent to $$- \int_D d\mu + \int_\beta d\theta = -\int_D \frac{dx\wedge dy}{1-y^2} + 2\pi =0,$$ which is exactly condition (3) above.
\end{proof}

\begin{remark}
This construction yields an exact Lagrangian torus $\mathbb{T}_\gamma$ for every simple closed geodesic $\gamma$ on $\Sigma$ which is well defined up to exact Lagrangian isotopy. Indeed, any choice of curve $\beta$ satisfying the conditions (1), (2) and (3) above yields an exact Lagrangian torus and an isotopy of $\beta$ through curves satisfying these conditions yields a exact Lagrangian isotopy. Moreover, $\Sigma$ admits $3g-3$ simple closed geodesics that are pairwise disjoint, and the construction can be applied to each of those in such a way that the $3g-3$ exact Lagrangian tori that we obtain are pairwise disjoint. This construction can also be performed along an \emph{immersed} geodesic on $\Sigma$ to obtain an \emph{immersed} exact Lagrangian torus.
\end{remark}

\begin{figure}
    \centering
    \includegraphics[width=0.6\linewidth]{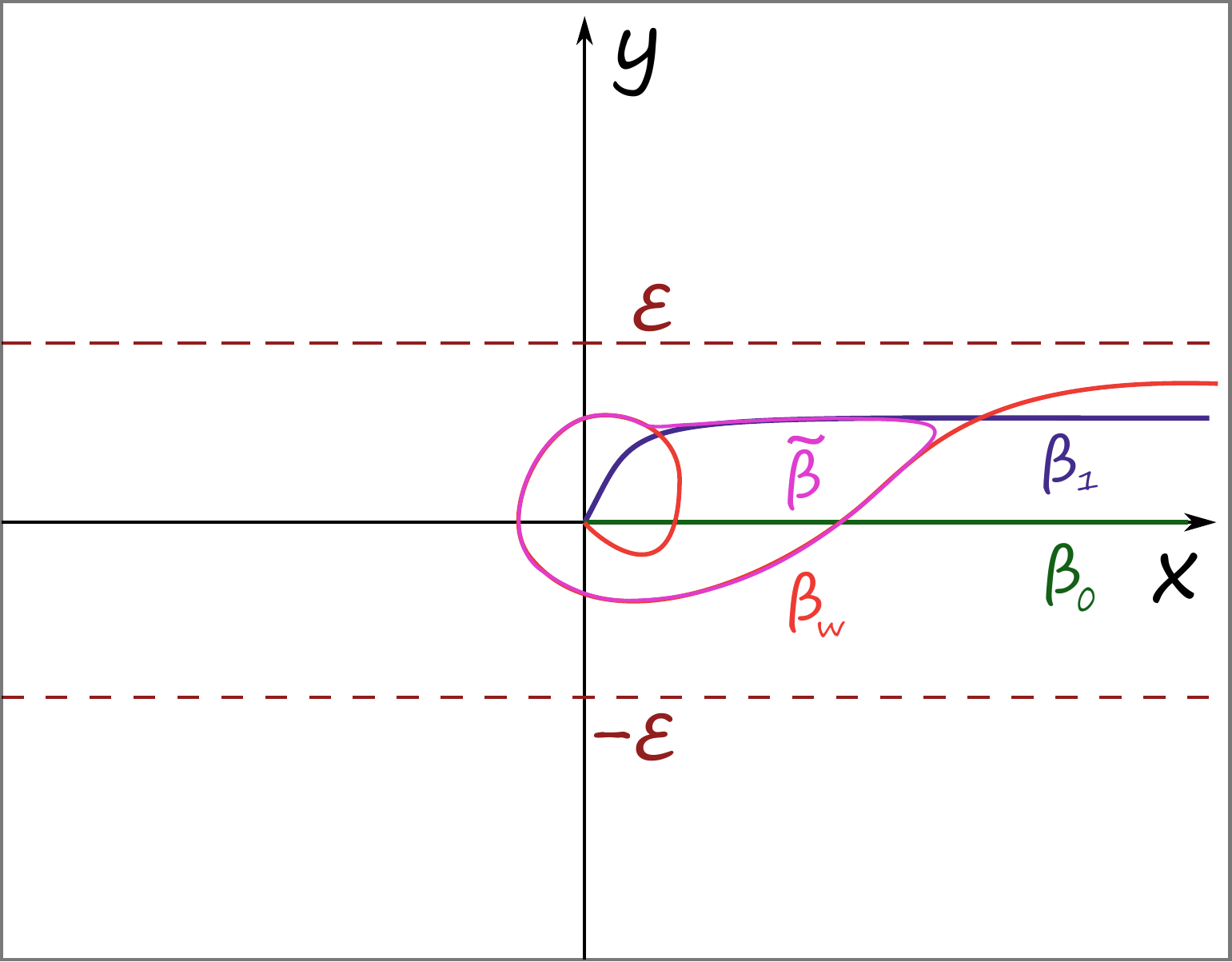}
    \caption{The curves $\beta_0,\beta_1, \beta_w,\widetilde \beta$.}
    \label{fig:more_curves}
\end{figure}

The exact Lagrangian torus $\mathbb{T}_\gamma$ can also be obtained in the following way. Let $\mathcal L_\gamma$ be the positive conormal lift of $\gamma$ in $T^* \Sigma \setminus 0_\Sigma$, which we know is an exact cylindrical Lagrangian. In the above coordinates, it is simply $$\mathcal L_\gamma = \big\{(0, t, x, 0 ) : t \in S^1_\ell, x \in \R_{>0} \big\}.$$ Note that for a choice of $\beta$ in the construction of $\mathbb{T}_\gamma$ as in Figure \ref{fig:beta}, $\mathbb{T}_\gamma$ and $\mathcal L_\gamma$ intersect cleanly along a circle. For any smooth simple curve $\beta = (f,g): \R_{> 0} \rightarrow \R^2 \setminus \{0\}$ contained in the strip $\R \times (-\epsilon, \epsilon)$, the subset $$ \mathcal L_{\gamma; \beta} := \left\{ \left(-\tanh^{-1}(g(s)), t, \beta(s) \right) : s \in \R_{>0}, \, t \in S^1_\ell \right\}$$ is an exact Lagrangian submanifold of $(V_\delta, \lambda_C)$. For $\beta_0(s) = (s,0)$, we have $\mathcal L_\gamma =\mathcal  L_{\gamma, \beta_0}$ as in Figure \ref{fig:more_curves}. Under suitable assumptions on the behavior of $\beta$ for $s$ near $0$ and $s$ near $+\infty$, $\mathcal L_{\gamma;\beta}$ is also cylindrical at infinity. To be more precise, let $X$ be the vector field on $\R^2 \setminus \{0\}$ defined by $$ X(x,y) = x \big( 1-x^2 -2y^2\big) \partial_x + y (1-y^2) \, \partial_y.$$

\begin{figure}
    \centering
    \includegraphics[width=1\linewidth]{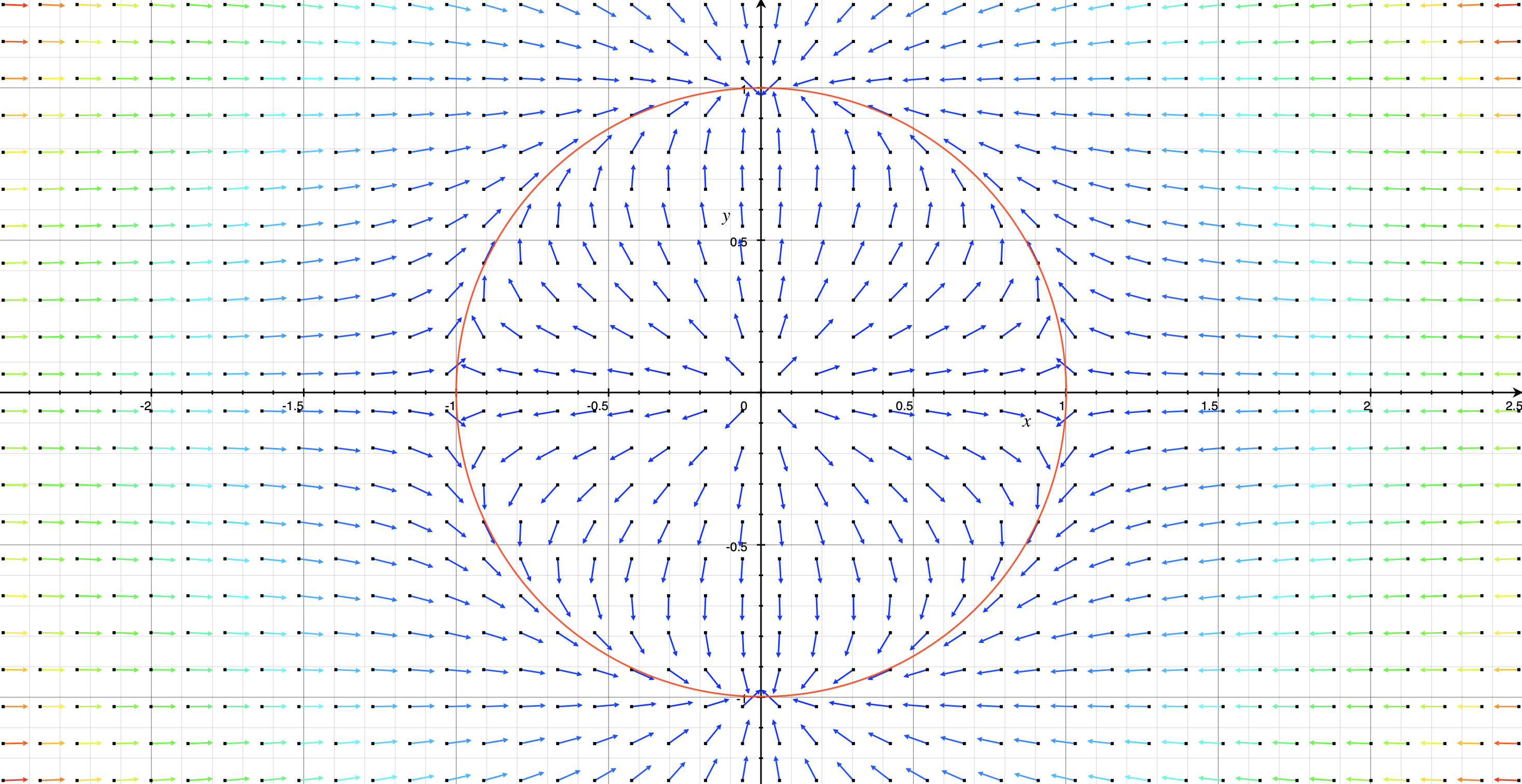}
    \caption{The vector field $X$. The locus where $\mathcal{L}_{\gamma;\beta}$ intersects the skeleton corresponds to the unit circle.}
    \label{fig:vectorfield}
\end{figure}

Near $0$, $X$ essentially behaves like the radial vector field $x\partial_x + y \partial_y$, and for $x \rightarrow +\infty$, $X$ essentially behaves like a horizontal vector field (colinear to $\partial_x$), see Figure~\ref{fig:vectorfield}. We further assume that
\begin{enumerate}
    \item[(4a)] $\underset{s \rightarrow 0}{\lim} \;\beta(s) = 0$ and $\beta$ is tangent to $X$ for $s$ near $0$,
    \item[(4b)] $\underset{s \rightarrow +\infty}{\lim}  f(s) = +\infty$ and $\beta$ is tangent to $X$ for $s$ sufficiently large.
\end{enumerate}
Then $\mathcal L_{\gamma ; \beta}$ is cylindrical at infinity. Indeed, it is enough to check that $\lambda_C$ restricted to $\mathcal L_{\gamma ; \beta}$ is zero for $s$ near $0$ and $s$ near $+\infty$. On $\mathcal L_{\gamma ; \beta}$, $\lambda_C$ restricts to $ \rho dr + d \theta$ and it vanishes if and only if the following equation is satisfied: $$ -\frac{f g'}{1-g^2} + \frac{fg'-gf'}{f^2 + g^2} = 0. $$
The latter is equivalent to $$g'f \left(1-f^2 - 2g^2\right) - f' g (1-g^2) = 0$$ and it is satisfied by the assumptions (4a) and (4b), i.e., we get $g'f'-f'g'=0$.

For a curve $\beta_1$ as in Figure \ref{fig:more_curves}, $\mathcal L_{\gamma ; \beta_1}$ is Hamiltonian isotopic to $\mathcal L_\gamma$ and is disjoint from it. For $\beta_w$ also as in Figure \ref{fig:more_curves}, $\mathcal L_{\gamma ; \beta_w}$ is a positive wrapping of $\mathcal L_\gamma$ and intersects $\mathcal L_{\gamma ; \beta_1}$ cleanly along two circles. These intersections can be resolved to obtain $\mathcal L_{\gamma ; \widetilde{\beta}}$ for some closed curve $\widetilde{\beta}$. For an appropriate choice of wrapping, $\widetilde{\beta}$ satisfies the assumptions (1), (2) and (3) above. To summarize, we obtained $\mathbb{T}_\gamma$ from $\mathcal L_\gamma$ and a slight push-off of $\mathcal L_\gamma$, by wrapping $\mathcal L_\gamma$ passed its push-off and resolving the two $S^1$-families of intersections. Therefore, it is natural to expect the following

\begin{conjecture}\label{conj:Tgamma}
As an object of $\mathcal{W}(V)$, $\mathbb{T}_\gamma$ is (split-)generated by $\mathcal L_\gamma$.
\end{conjecture}

In order to address the conjecture, we would need to implement an $S^1$-equivariant version of the fact that a Polterovich surgery induces a cone in the Fukaya category. Note that $\mathbb{T}_\gamma$ can be obtained from $\mathcal{L}_{\overline{\gamma}}$ by a similar procedure, where $\overline{\gamma}$ denotes $\gamma$ with the opposite orientation. This suggests that there might exist non trivial relations between cones over $\mathcal{L}_\gamma$ and cones over $\mathcal{L}_{\overline{\gamma}}$ in the wrapped Fukaya category.

\begin{remark}
This construction remains valid for a closed curve $\beta$ satisfying (1), (2) but not (3). The corresponding Lagrangian torus is incompressible, hence weakly exact, but not exact, and it is disjoint from the skeleton and contained in the prequantization end if $\beta$ is contained in the open unit disk.\footnote{If $\mathbb{D}$ denotes the open unit disk in the plane, $$\int_\mathbb{D} \frac{dx\wedge dy}{1-y^2} = 2\pi,$$ so $\mathcal{L}_{\gamma;\beta}$ is \emph{never} exact when $\beta$ is contained in $\mathbb{D}$.} We expect these tori to be Lagrangian (and perhaps Hamiltonian) isotopic to the ones from Section~\ref{sec:Mcduff_domains},\footnote{One could compare these tori by finding an explicit exact symplectomorphism between the symplectization of the prequantization contact form and the McDuff Liouville form in the prequantization end using Moser's trick. Unfortunately, the computations seem quite tedious.} and to be split-generated by $\mathcal L_\gamma$ in the wrapped Fukaya category of $V$ with Novikov coefficients.
\end{remark}

All of the (weakly) exact Lagrangian tori coming from the previous construction  either intersect the skeleton or are entirely contained in the component of the complement of the skeleton corresponding to the prequantization bundle end. In fact, the component of the complement of the skeleton corresponding to the cotangent bundle end cannot contain weakly exact Lagrangian tori:

\begin{lemma} \label{lemma:weakly}
If $\Sigma$ is a closed orientable surface of genus $g \geq 2$, then $\left(T^* \Sigma \setminus 0_\Sigma, \lambda_{\mathrm{can}} \right)$ does not contain any weakly exact Lagrangian tori.
\end{lemma}

\begin{proof}
We only sketch the main arguments and leave the details to the reader. Assume by contradiction that $T \subset \left(T^* \Sigma \setminus 0_\Sigma, \lambda_{\mathrm{can}} \right)$ is a weakly exact Lagrangian torus. The image of the morphism $i_T : \pi_1(T) \cong \Z^2 \rightarrow \pi_1(S^*\Sigma)$ induced by the inclusion $T \subset T^*\Sigma \setminus 0_\Sigma \cong \R \times S^*\Sigma$ is contained in a subgroup $H_\gamma = \langle \widetilde{\gamma}, f \rangle$ generated by the class of the lift of a loop $\gamma \subset \Sigma$ and the class $f$  a fiber. Indeed, there is a short exact sequence $$  1 \longrightarrow \Z \overset{i_*}{\longrightarrow} \pi_1(S^*\Sigma) \overset{\pi_*}{\longrightarrow} \pi_1(\Sigma) \longrightarrow 1$$ induced by the homotopy long exact sequence for the fiber bundle $S^1 \overset{i}{\hookrightarrow} S^*\Sigma \overset{\pi}{\rightarrow} \Sigma.$ Recall that a subgroup of $\pi_1(\Sigma)$ generated by at most two elements is either trivial, infinite cyclic, or free. Therefore, there exists a loop $\gamma$ on $\Sigma$ such that the image of the composition $\pi_* \circ i_T$ is contained in $\langle \gamma \rangle$ (where $\gamma$ is arbitrary if the image is trivial). The short exact sequence readily implies that the image of $i_T$ is contained in $H_\gamma$. 

It follows from the preceding discussion that $T$ lifts to a weakly exact Lagrangian torus in $(T^*C \setminus 0_C, \lambda_\mathrm{can})$, where $C = \R \times S^1$. Here, $\{0\} \times S^1$ projects onto the curve $\gamma$. Closing up the cylinder $C$ by quotienting the $\R$ direction by a sufficiently large translation, we obtain a \emph{displaceable} weakly exact Lagrangian torus $T'$ in $(T^* \mathbb{T}^2 \setminus 0_{\mathbb{T}^2}, \lambda_\mathrm{can})$. This implies that $T' \subset T^* \mathbb{T}^2$ is diplaceable, and by a theorem of Chekanov~\cite{C98} (see also~\cite{Oh97}), there exists $A > 0$ such that $T'$ bounds a nontrivial $J$-holomorphic disk with energy less than $A$ for any tame and convex almost complex structure $J$. Since $T'$ is weakly exact in $T^*\mathbb{T}^2 \setminus 0_{\mathbb{T}^2}$, these disks necessarily intersect the zero section. Performing a neck-stretching near the zero section, we obtain a sequence of holomorphic disks converging to a holomorphic building of genus $0$ such that one of its components is a holomorphic plane in $T^* \mathbb{T}^2$ asymptotic to a closed Reeb orbit and intersecting the zero section. This is impossible since the Reeb orbits in $S^* \mathbb{T}^2$ project to non-contractible curves on $\mathbb{T}^2$. 
\end{proof}

\begin{figure}
\centering
\begin{subfigure}{.5\textwidth}
  \centering
  \includegraphics[width=0.95\linewidth]{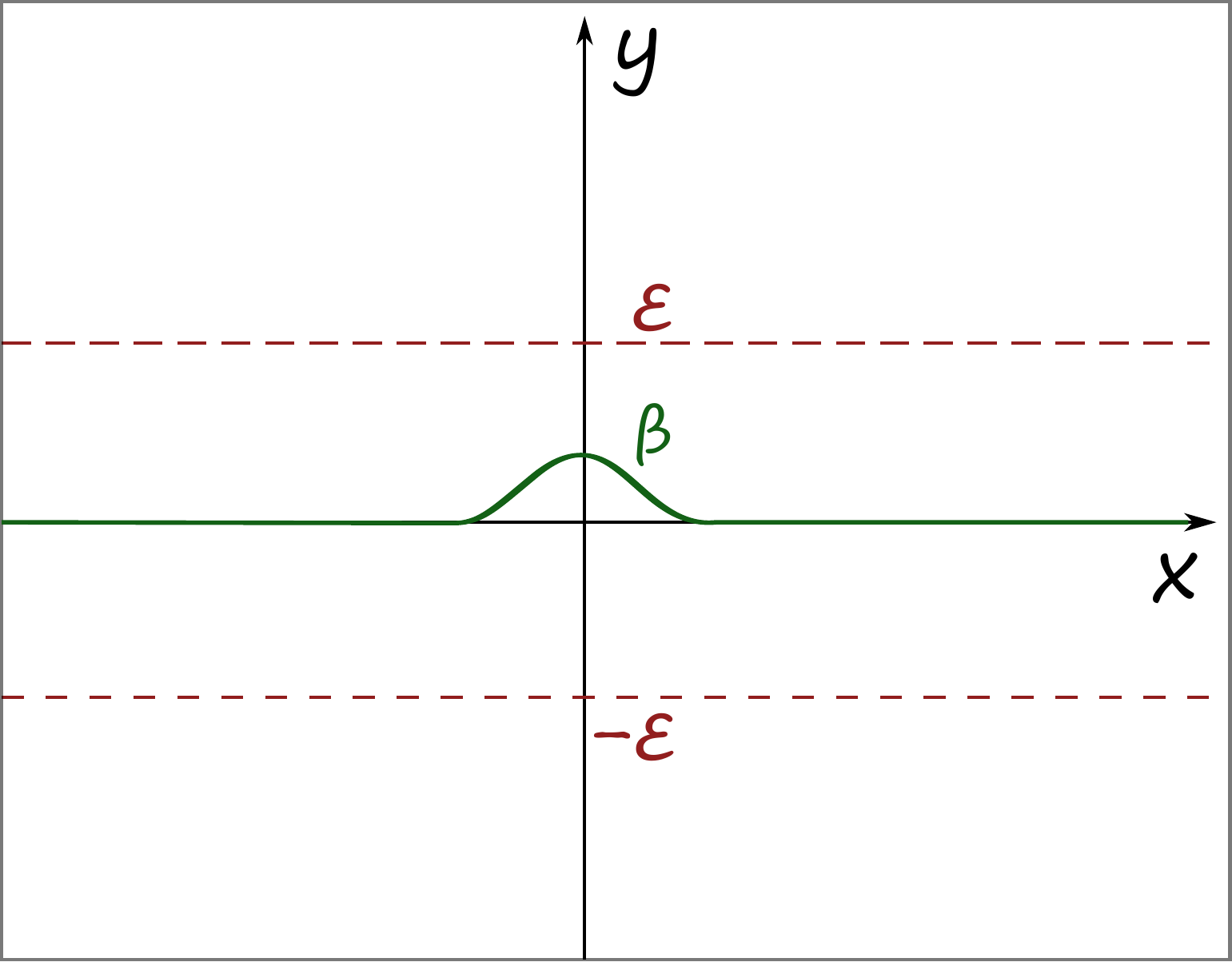}
  \caption{The curve $\beta$ for the U-shaped Lagrangian.}
  \label{fig:U_shaped}
\end{subfigure}%
\begin{subfigure}{.5\textwidth}
  \centering
  \includegraphics[width=0.95\linewidth]{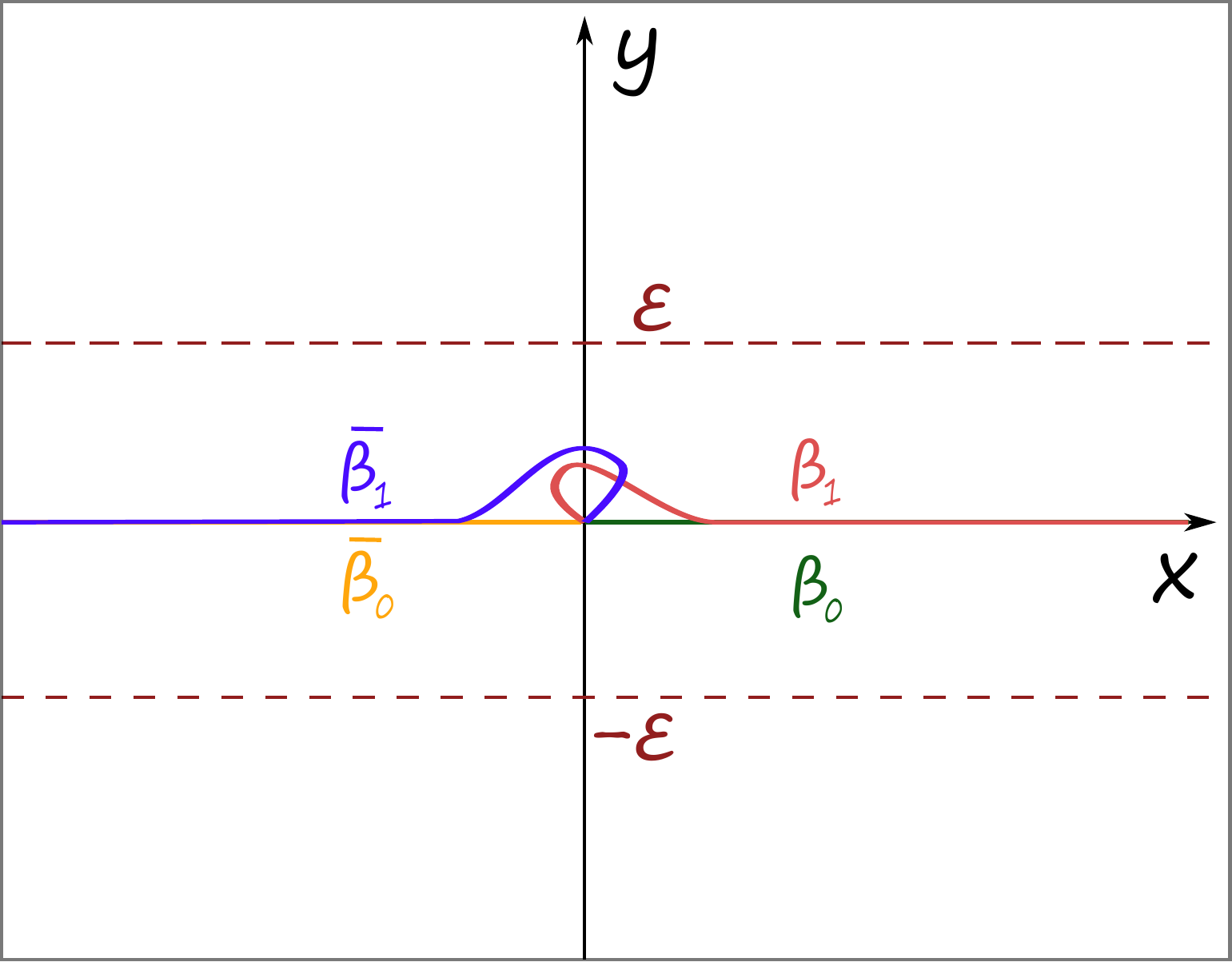}
    \caption{The curves $\beta_0,\beta_1$ and their opposite  $\overline{\beta}_0,\overline{\beta}_1$.}
    \label{fig:U_shaped_wrapped}
\end{subfigure}
\caption{The relevant curves for the U-shaped Lagrangian.}
\end{figure}

We end this section with one last construction. We now take $\beta : \mathbb{R} \rightarrow \R^2 \setminus \{0\}$ to be as in Figure \ref{fig:U_shaped}, with $\beta(s) = (s, 0)$ for $\vert s \vert \geq 1$. The corresponding Lagrangian cylinder $\mathcal L_{\gamma ; \beta}$ is exact and intersects the skeleton of $V$. It has two cylindrical ends in the cotangent bundle end of $V$ which coincide with $\mathcal L_\gamma \sqcup \mathcal L_{\overline{\gamma}}$, where $\overline{\gamma}$ denotes $\gamma$ with the opposite orientation. We call it a \emph{U-shaped Lagrangian} associated with $\gamma$. It can also be obtained from $\mathcal L_\gamma$ and $\mathcal L_{\overline{\gamma}}$ by wrapping them in opposite ways in the prequantization end in order to create an clean intersection along a circle, and then resolving this intersection. See Figure \ref{fig:U_shaped_wrapped} for the corresponding picture from the point of view of the $\beta$ curves, where the curve $\overline{\beta}_0(s)=(s,0)$ for $s\in \mathbb{R}_{<0}$ gives $\mathcal{L}_{\overline{\gamma}}=\mathcal{L}_{\overline{\gamma};\overline{\beta}_0}$, and $\beta_1,\overline{\beta}_1$ are the wrappings of $\beta_0$ and $\overline{\beta}_0$ respectively. As an object of $\mathcal{W}(V)$, we expect the U-shaped Lagrangian to be split-generated by $\mathcal L_\gamma$ and $\mathcal L_{\overline{\gamma}}$.

        \subsection{Closed Lagrangians in the torus bundle domains}

We now assume that $M$ is a $\mathbb{T}^2$-bundle over $S^1$ with monodromy given by a positive hyperbolic matrix $A \in \mathrm{SL}(2,\Z)$, $(V=\R \times M, \lambda)$ is a torus bundle manifold, and we show Theorem~\ref{thm:closed} from the Introduction.

It follows from the homotopy long exact sequence that $\pi_1(V) = \pi_1(M) \cong \Z^2 \rtimes_A \Z$ is the semi-direct product of $\Z^2$ with $\Z$ where $\Z$ acts on $\Z^2$ by $A$, and this group is torsion-free as it sits in the short exact sequence $$ 0 \rightarrow \Z^2 \rightarrow \Z^2 \rtimes_A \Z \rightarrow \Z \rightarrow 0.$$ Recall that if $K$ denotes the Klein bottle, $\pi_1(K) =\langle a, b \, | \, aba = b \rangle$.

\begin{lemma} \label{lemmalg} Let $G := \Z^2 \rtimes_A \Z$ and $H = \Z^2$, $\Z/2\Z$ or $\langle a, b \, | \, aba = b \rangle$. If $\phi : H \rightarrow G$ is a group morphism then either $\phi(H) \subseteq \Z^2 \rtimes 0$ or $\phi(H)$ is a cyclic subgroup of $G$.
\end{lemma}

\begin{proof} 
The case $H=\Z/2\Z$ is trivial since $G$ is torsion-free.

For the case $H= \langle a, b \, | \, aba = b \rangle$, let $\phi(a) := (v, n)$ and $\phi(b) := (w, m)$. The relation $aba=b$ imposes $n=0$ and $v + A^mv = 0$. Since $-1$ is not an eigenvalue of $A^m$ by the hyperbolicity of $A$, we get $v=0$ so $\phi(a)=0$. Therefore, $\phi(H)$ is the cyclic subgroup of $G$ generated by $\phi(b)$.

Now let's assume that $H = \Z^2$. It suffices to prove that if $\phi: H \hookrightarrow G$ is injective then its image lies in $G_0:=\Z^2 \rtimes 0$. We identify $H$ with a subgroup of $G$ isomorphic to $\Z^2$. Denoting $H_0 := H \cap G_0$, there is an injection $H \slash H_0 \hookrightarrow G \slash G_0 \cong \Z$ which implies that $H_0 \neq 0$. There are two cases left:
\begin{itemize}
\item Case 1: $H_0 \cong \Z$. Then $H$ is generated by two elements $(v, 0)$ and $(w,m)$ where $v \neq 0$ and $m \neq 0$. Since these elements commute, we get: $v + w = w +~A^m v$, so $A^mv = v$ and $1$ is an eigenvalue of $A^m$, but this contradicts the hyperbolicity of $A$.
\item Case 2: $H_0 \cong \Z^2$. Then $[H:H_0]$ is finite and must equal $1$ because of the injection $H \slash H_0 \hookrightarrow \Z$, so $H = H\cap H_0 \subseteq \Z^2 \rtimes 0$ as desired.
\end{itemize}
This concludes the proof.
\end{proof}

We now explicitly describe the cover $V'$ of $V$ associated with the cyclic subgroup generated by an element of the form $(v, 1) \in \Z^2 \rtimes_A \Z$, and the cover $V''$ of $V$ associated with the normal subgroup $\Z^2 \rtimes 0 \subset \Z^2 \rtimes_A \Z$. To this end, we construct $V'$ and $V''$ as quotients of $(\mathbb{R}^4, \lambda_0)$ and we show that those are exact symplectomorphic to subsets of suitable cotangent bundles.

Recall that $A \in \mathrm{SL}(2,\Z)$ is a hyperbolic matrix and $P \in \mathrm{SL}(2,\R)$ and $\nu \in \R \setminus \{0\}$ are such that 
\[ PAP^{-1} = \begin{pmatrix}
e^\nu & 0\\
0 & e^{-\nu}
\end{pmatrix} =: D_\nu.\]

For $v_0 \in \R^2$, let 
\[\begin{array}{crcl}
\phi_{\nu, v_0} :& \R^2 \times \R & \longrightarrow & \R^2 \times \R \\
	& (v, z) & \longmapsto & \left( D_\nu v+ v_0, z -\nu \right),
\end{array}\]
where $v = (x,y)$.
The quotient $\R^3 \slash \phi_{\nu, v_0}$ is an affine plane bundle over $S^1$ whose monodromy is given by $D_\nu + v_0$. The contact forms $\alpha_\pm$ are preserved by $\phi_{\nu, v_0}$ so they descend to contact forms on $\R^3 \slash \phi_{\nu, v_0}$, and $\lambda_0$ descends to a Liouville form $\lambda'$ on $V' :=\R \times \R^3 \slash \phi_{\nu, v_0}$. 

\begin{lemma} \label{symp1} $\left(V', \lambda'\right)$ is exact symplectomorphic to $\left( T^*(\R \times S^1), \lambda_{\mathrm{can}} \right)$.
\end{lemma}

\begin{proof} We first trivialize the affine plane bundle $\R^3 \slash \phi_{\nu, v_0} \rightarrow S^1 = \R \slash \nu \Z$. Let
\[\begin{array}{crcl}
\tilde{\theta} :& \R^2 \times \R & \longrightarrow & \R^2 \times \R \\
	& (v, z) & \longmapsto & \left( D_z^{-1} v+ w_0, z \right),
\end{array}\]
where $D_z := \begin{pmatrix}
e^z & 0\\
0 & e^{-z}
\end{pmatrix}$ and $w_0 := \left(I - D_\nu \right)^{-1} v_0$. This specific choice of $w_0$ ensures that $\tilde{\theta}(v, z-\nu) = \phi_{\nu, v_0} \circ \tilde{\theta}(v, z)$ so $\tilde{\theta}$ induces a diffeomorphism $\overline{\theta} : \R^2 \times S^1 \rightarrow \R^3 \slash \phi_{\nu, v_0}$ and a diffeomorphism $\theta : \R \times \R^2 \times S^1 \rightarrow V'$. We then have
\begin{align*}
\theta^*\lambda' &= \sinh(s)e^z \theta^*dx + \cosh(s) e^{-z} \theta^*dy\\
&= \sinh(s) dx + \cosh(s) dy -\left(\sinh(s) x - \cosh(s) y\right) dz \\
&= \begin{multlined}[t] d\left( \sinh(s)x + \cosh(s)y\right) \\- \left(\cosh(s) x + \sinh(s) y\right) ds - \left(\sinh(s) x - \cosh(s) y\right) dz.
\end{multlined}
\end{align*}
Let $f : (s,x,y,z) \mapsto \sinh(s)x + \cosh(s)y$ and consider the change of coordinates given by
\[\begin{array}{crcl}
\psi :& \R \times \R^2 \times S^1 & \longrightarrow&\R \times \R^2 \times S^1\\
	& (s,x,y,z) & \longmapsto & \left( s, \begin{pmatrix}
-\sinh(s) & \cosh(s)\\
-\cosh(s) & -\sinh(s)
\end{pmatrix} \cdot \begin{pmatrix}
x\\
y 
\end{pmatrix}, z \right) = (s,a,b,z).
\end{array}\]
In these new coordinates, we get
\[\left(\theta \circ \psi^{-1}\right)^*\lambda' = -d\big( f\circ \psi^{-1}\big) + a dz + b ds,\]
so after identifying $\R \times \R^2 \times S^1$ with $T^*(\R \times S^1)$, $\theta \circ \psi^{-1} : T^*(\R \times S^1) \rightarrow V'$ gives the desired exact symplectomorphism.
\end{proof}

The contact forms $\alpha_\pm$ and therefore the Liouville form $\lambda_0$ are translation invariant along the $\R^2_{x,y}$ direction. Quotienting by the lattice $P(\Z^2) \subset \R^2_{x,y}$ yields the exact symplectic manifold $(V'', \lambda'') = \left( \R \times (\R^2 \slash P(\Z^2)) \times \R, \lambda'' \right).$

\begin{lemma} \label{symp2} $\left(V'', \lambda''\right)$ is exact symplectomorphic to a subset of $\left(T^* \mathbb{T}^2, \lambda_{\mathrm{can}} \right)$ disjoint from the zero section.
\end{lemma}

\begin{proof} First of all, $P$ induces a diffeomorphism $\R^2 \slash \Z^2 \overset{\sim}{\rightarrow} \R^2 \slash P(\Z^2)$. We compute
\begin{align*}
P^* \lambda'' &= \sinh(s) e^{-z} P^*dx + \cosh(s) e^z P^*dy\\
&= P^{T}  \begin{pmatrix}
\sinh(s) e^{-z}\\
\cosh(s) e^z 
\end{pmatrix}  \begin{pmatrix}
dx\\
dy 
\end{pmatrix}.
\end{align*}
Let's consider the diffeomorphism
$$\begin{array}{crcl}
\psi_0 :& \R^2 & \longrightarrow & \R \times \R_{>0} \\
	& (s, z) & \longmapsto & \big(\sinh(s) e^{-z}, \cosh(s) e^z \big),
\end{array}$$
which induces a diffeomorphism $\psi := P^{T}\psi_0 : \R^2 \overset{\sim}{\rightarrow} P^{T} \left( \R \times \R_{>0}\right) \subset \R^2_{a,b}$. We then have
\[ \big(\psi^{-1} \big)^* \lambda'' = a dx + b dy,\]
so after identifying $\R_a \times \R^2 \slash \Z^2_{x,y} \times \R_b$ with $T^*\mathbb{T}^2$, $\psi$ induces an exact symplectomorphism $\left(V'', \lambda''\right) \rightarrow \left( T^*\mathbb{T}^2\setminus Z, \lambda_\mathrm{can} \right)$, where $Z \subset T^*\mathbb{T}$ is an closed subset containing the zero section.
\end{proof}

Since $A$ preserves the lattice $\Z^2$, $D_\nu$ preserves the lattice $P(\Z^2)$. Then, $\phi_{\nu, v_0}$ descends to a diffeomorphism $\overline{\phi}_{\nu, v_0}$ of $\R^2 \slash P(\Z^2) \times \R$, or alternatively $P(\Z^2)$ induces a well-defined action on the fibers of  $\R^3 \slash \phi_{\nu, v_0} \rightarrow S^1$. After quotienting $V'$ by the diffeomorphism induced by $\overline{\phi}_{\nu, v_0}$, or equivalently quotienting $V''$ by the action induced by $P(\Z^2)$, we obtain the torus bundle manifold $(V=\R \times M,\lambda)$.

\begin{proof}[Proof of Theorem \ref{thm:closed}] Let's assume by contradiction that there exists an embedded  exact Lagrangian $L \hookrightarrow (V, \lambda)$ in the torus bundle domain that is either a torus, a projective plane or a Klein bottle. By Lemma~\ref{lemmalg}, we can distinguish two cases:
	\begin{itemize}
	\item Case 1: the image of $\pi_1(L) \rightarrow \pi_1(V) = \Z^2 \rtimes_A \Z$ is contained in $\Z^2 \rtimes 0$. Then $L$ lifts to an embedded exact Lagrangian surface $L \hookrightarrow (V'', \lambda'')$, where $V''$ is the cover of $V$ corresponding to $\Z^2 \rtimes 0$. By Lemma~\ref{symp2}, we obtain an exact Lagrangian surface in $\left(T^* \mathbb{T}^2, \lambda_{\mathrm{can}} \right)$ disjoint from the zero section, but this is impossible by a theorem of Gromov~\cite{Gr85}.
	\item Case 2: the image of $\pi_1(L) \rightarrow \pi_1(V)$ is cyclic and is not contained in $\Z^2 \rtimes 0$. Then, it is generated by an element of the form $(v, n)$ with $n \neq 0$. The cover of $V$ corresponding to the subgroup $\langle(v, n)\rangle \subset \pi_1(V)$ is essentially of the form $\left(V', \lambda'\right)$ but with $A$ replaced by $A^n$. By Lemma~\ref{symp1}, we get a closed exact Lagrangian surface $L \hookrightarrow \left( T^*(\R \times S^1), \lambda_{\mathrm{can}} \right)$, but this is impossible by a theorem of Lalonde and Sikorav~\cite{LS91} asserting that if $U$ is an open manifold then $\left(T^*U, \lambda_\mathrm{can}\right)$ has no closed exact Lagrangian submanifold.
	\end{itemize}
	This finishes the proof.
\end{proof}

\section{Rabinowitz Floer cohomology and symplectic cohomology} 

In this section we prove Theorem~\ref{thm:RFH-SH} from the Introduction and its corollaries.

Let $V$ be a Liouville domain of dimension $2n$. We denote by $SH^*(V)$ its symplectic cohomology and by $RFH^*(V)$ its Rabinowitz Floer cohomology~\cite{CFO10,CO18} (cf. the remarks after Theorem~\ref{thm:RFH-SH} on gradings). Recall from \cite{CFO10} that there is a commutative diagram of exact sequences relating symplectic cohomology, Rabinowitz Floer cohomology, and symplectic homology $SH^*(V,\partial V)\cong SH_{2n-*}(V)$, of the form\footnote{Notice the directions of the vertical arrows.}
$$
    \begin{tikzcd}
        \dots\arrow[r] & SH^{*}(V,\partial V)\arrow[r,"\eps"] \arrow[d]&
        SH^{*}(V) \arrow[r,"\iota"] & RFH^{*}(V) \arrow[r,"\pi"]& SH^{*+1}(V,\partial V)\arrow[r]\arrow[d]&\dots\\
        \dots\arrow[r]&H^{*}(V,\partial V) \arrow[r,"\eps_0"]& H^{*}(V) \arrow[u]\arrow[r,"\iota_0"]& H^{*}(\partial V) \arrow[u,"\delta"]\arrow[r,"\pi_0"] & H^{*+1}(V,\partial V) \arrow[r] &\dots 
    \end{tikzcd}
$$

Suppose now that $\partial V$ is hypertight. Then $RFH^*(V)\cong RFH^*(\partial V)$ as a ring, where the latter is defined on the symplectization of $\partial V$; in particular, it does not depend on the filling $V$ (\cite{U19}). If in addition $\partial V$ has two components, it follows that the ring structure in Rabinowitz Floer cohomology splits nicely into two pieces.

\begin{proposition}[Splitting of Rabinowitz Floer cohomology]\label{prop:RFH}
Let $V=[-1,1]\times M$ be a $2n$-dimensional Liouville domain such that each boundary component $M_\pm=\{\pm 1\}\times V$ is hypertight. Then we have a ring isomorphism 
$$
RFH^*(V)\cong RFH^*(M_-)\oplus RFH^*(M_+).
$$
In particular, the unit in $RFH^*(V)$ is of the form $1=(1_-,1_+)$. Moreover, the commutative diagram above becomes
$$
    \begin{tikzcd}
        0 \arrow[r] &
        SH^{*}(V) \arrow[r,"\iota"] & RFH^*(M_-)\oplus RFH^*(M_+) \arrow[r,"\pi"]& SH^{*+1}(V,\partial V)\arrow[r]\arrow[d]& 0\\
        0 \arrow[r]& H^{*}(M) \arrow[u]\arrow[r,"\iota_0"]& H^{*}(M)\oplus H^*(M) \arrow[u,"\delta"]\arrow[r,"\pi_0"] & H^{*}(M) \arrow[r] &0
    \end{tikzcd}
$$
where $\iota_0$ is the diagonal embedding, and $\delta=\delta_-\oplus \delta_+$ is the direct sum of the corresponding maps $\delta_\pm: H^{*}(M_\pm) \rightarrow RFH^*(M_\pm)$.
\end{proposition}

\begin{proof}
The splitting $RFH^*(V)\cong RFH^*(M_-)\oplus RFH^*(M_+)$ follows from the preceding discussion. Next, let us compute the lower row in the commutative diagram for $V=I\times M$ with $I=[-1,1]$. The long exact sequence of the pair $(I,\p I)$ in singular cohomology is
$$
    \begin{tikzcd}
        0 \arrow[r] &
        H^{0}(I) \arrow[r,"\iota_0"] & H^0(\p I) \arrow[r,"\pi"]& H^{1}(I,\partial I)\arrow[r]\arrow[d,equal]& 0\\
        0 \arrow[r]& \Z \arrow[u,equal]\arrow[r]& \Z \oplus \Z \arrow[u,equal]\arrow[r] & \Z \arrow[r] &0
    \end{tikzcd}
$$
where the map $\Z\to\Z\oplus\Z$ is the diagonal embedding. By naturality of the K\"unneth formula~\cite[Ch.~5 Sec.~6 Thm.~1]{Sp66}, the long exact sequence of the pair $(V,\p V)$ is obtained by tensoring this one with $H^*(M)$,
$$
    \begin{tikzcd}
        0 \arrow[r] &
        H^{*}(V) \arrow[r,"\iota_0"] & H^*(\p V) \arrow[r,"\pi"]& H^{*+1}(V,\partial V)\arrow[r]\arrow[d,equal]& 0\\
        0 \arrow[r]& H^*(M) \arrow[u,equal]\arrow[r]& H^*(M) \oplus H^*(M) \arrow[u,equal]\arrow[r] & H^*(M) \arrow[r] &0 
    \end{tikzcd}
$$ 
where $\iota_0$ is the diagonal embedding. Since the map $\eps$ factors through $\eps_0=0$, we obtain $\eps=0$ and thus the commutative diagram in the proposition.
\end{proof}

Proposition~\ref{prop:RFH} provides an efficient tool for computing the product structure on symplectic cohomology in terms of that on Rabinowitz Floer cohomology: 

\begin{corollary}[Splitting of symplectic cohomology]\label{cor:SH-product}
Let $V=[-1,1]\times M$ be a $2n$-dimensional Liouville domain such that each boundary component $M_\pm=\{\pm 1\}\times V$ is hypertight, and the free homotopy classes of positively or negatively parametrized closed Reeb orbits on $M_+$ are distinct from those on $M_-$. With respect to the splitting as a $\Z$-module
$$
   SH^*(V) \cong SH^*_-(V)\oplus SH^*_0(V) \oplus SH^*_+(V),
$$ 
we denote the components of the product $\mu$ on $SH^*(V)$ by $\mu^{ab}_c := \pi_c \circ \mu \vert_{SH^*_a(V)\otimes SH^*_b(V)}$, where $a,b,c\in \{+,-,0\}$ and $\pi_c: SH^*(V)\rightarrow SH^*_c(V)$ is the natural projection.\footnote{Here, $c \in \{+, -, 0\}$ does not mean ``contractible''.} The only possibly nontrivial components of the product are the following:
\begin{itemize}
    \item $\mu^{00}_0$, the cup product on $SH^*_0(V)\cong H^*(M)$;
    \item $\mu^{--}_-,\mu^{-0}_-,\mu^{0-}_-$, determined by the product on $RFH^*(M_-)$;
     \item $\mu^{++}_+,\mu^{+0}_+,\mu^{0+}_+$, determined by the product on $RFH^*(M_+)$.
\end{itemize}
\end{corollary}

\begin{proof}
The hypothesis on the free homotopy classes implies that the injective ring map $\iota$ in Proposition~\ref{prop:RFH} sends $SH_-^*(V)$ to $RFH^*(M_-)$ and $SH_+^*(V)$ to $RFH^*(M_+)$. 
Due to the splitting $RFH^*(V)\cong RFH^*(M_-)\oplus RFH^*(M_+)$ as a ring, this implies that $\mu$ cannot have components involving both $-$ and $+$. 
Since all orbits on $M_\pm$ are non-contractible, $\mu$ also cannot have components involving exactly two $0$'s. This leaves only the possibilities listed in the corollary as well as $\mu^{--}_0$ and $\mu^{++}_0$.

To see that the latter two vanish, recall from Proposition~\ref{prop:RFH} that $\iota$ sends $SH^*_0(V)$ to both summands via the composition
$$
    \begin{tikzcd}
        SH^*_0(V)\cong H^{*}(M) \arrow[r,"\iota_0"]& H^{*}(M)\oplus H^*(M) \arrow[r,"\delta"] & RFH^*(M_-)\oplus RFH^*(M_+),
    \end{tikzcd}
$$
where $\iota_0$ is the diagonal map and $\delta=\delta_-\oplus\delta_+$. Consider $a,b\in SH^*_+(V)$ and write $\mu(a,b)=c_0+c_+\in SH^*_0(V)\oplus SH^*_+(V)$. Then $\iota\mu(a,b)=\delta_-c_0+d_+$ for some $d_+\in RFH^*(M_+)$. On the other hand, $\iota\mu(a,b)=\mu(\iota a,\iota b)\in RFH^*(M_+)$, which implies $\delta_-c_0=0$ and therefore $c_0=\pi_0 c_0 = 0$. This shows $\mu^{++}_0=0$, and $\mu^{--}_0=0$ follows analogously.

By Morse cohomology, $\mu^{00}_0$ equals the cup product on $SH^*_0(V)\cong H^*(M)$. Since $\iota$ is an injective ring map, the remaining components are determined by $RFH^*(M_\pm)$ as stated. 
\end{proof}

Proposition~\ref{prop:RFH} and Corollary~\ref{cor:SH-product} together yield Theorem~\ref{thm:RFH-SH} from the Introduction. As a consequence, we now prove Corollaries~\ref{cor:SH-McDuff} and~\ref{cor:SH-torusbundle}. 

\begin{proof}[Proof of Corollary~\ref{cor:SH-McDuff}]
By Lemma~\ref{lem:hypertight}, a McDuff domain $V=[-1,1]\times M$ with $M=S^*\Sigma$ satisfies the hypotheses of Corollary~\ref{cor:SH-product}. So it only remains to compute the summands and the components of the product with respect to the $\Z$-module splitting
$$
   SH^*(V) \cong SH^*_-(V)\oplus SH^*_0(V) \oplus SH^*_+(V).
$$ 
To compute the components $\mu^{--}_{-},\mu^{-0}_-,\mu^{0-}_-$, we perturb the Hamiltonian by a Morse function on the base $\Sigma$. Since $\pi_2(\Sigma)=0$, the $3$-punctured Floer spheres contributing to these components project onto gradient Y-trees in $\Sigma$ connecting critical points, whose counts give the cup product on $H^*(\Sigma)$. A further perturbation by Morse functions on the $S^1$-fibres over critical points yields the desired form. 

Since $RFH^*(S^*\Sigma, \xi_\mathrm{can})$ is independent of the filling, we can compute it using its filling by the unit disk cotangent bundle $D^*\Sigma$, so it equals the Rabinowitz loop cohomology $\widecheck{H}^*(\mathcal{L}\Sigma)$ defined in~\cite{CHO}. As shown there, its product is an amalgamation of the loop product on the positive action part, the cup product on the zero action part, and the Goresky--Hingston product on the negative action part. 
\end{proof}

\begin{proof}[Proof of Corollary~\ref{cor:SH-torusbundle}]
Recall that the Reeb vector fields $R_\pm= \frac{1}{2} \big(\pm e^{-z} v_x+e^{z} v_y \big)$ are tangent to the torus fibers, spanned by $v_x,v_y$ (two eigenvectors of $A$), where $z \in [0,\nu)$ (here, $e^{\pm\nu}$ are the eigenvalues of the hyperbolic monodromy matrix). Therefore the closed $R_\pm$-orbits arise whenever $R_\pm$ has rational slope, foliating the fibers corresponding to a subset $\Gamma$ of $[0,\nu)$ in bijection with $\mathbb{Q}\cap [0,1)$. Moreover, different fibers correspond to orbits in different homology classes, and so there is no Floer cylinder between them. Since for $z\in \Gamma$, the fiber over $z$ is a Morse-Bott $S^1$-family, after introducing a Morse function on $S^1$ that breaks the symmetry, we see that $SH^*_\pm(V)=\bigoplus_\Gamma H^*(S^1)$. This concludes the proof. 
\end{proof}

\section{Wrapped Floer cohomology of Anosov Liouville domains}\label{sec:WFH}

    \subsection{Computation of wrapped Floer cohomology}

In this section, we prove the first part of Theorem~\ref{thm:WFH} from the Introduction. We apply this to compute the wrapped Floer cohomology groups of the Lagrangians $\mathcal{L}_\Lambda$ in any Anosov Liouville domain, and in particular for the McDuff domains and the torus bundle domains. The products are studied below.

\begin{proof}[Proof of Theorem~\ref{thm:WFH}(1): splitting as $\Z$-modules]  Similarly as in Theorem~\ref{thm:RFH-SH}, for any Liouville domain of the form $V=[-1,1]\times M$, and any two exact Lagrangians $L_1, L_2 \subset V$ with Legendrian boundary, the wrapped Floer cochain complex splits as a $\Z$-module as
$$CW^*(L_1,L_2)=CW^*_-(L_1,L_2) \oplus CW^*_0(L_1,L_2) \oplus CW^*_+(L_1,L_2).$$ Here, the first summand is generated by (Hamiltonian chords corresponding to) $\alpha_-$-chords, the second one by intersection points of $L_1$ and $L_2$ lying in the interior of $V$, and the third one by (Hamiltonian chords corresponding to) $\alpha_+$-chords. 
By an open version of~\cite[Lemma 2.2, Lemma 2.3]{CO18}, there cannot be any Floer strips to an $\alpha_\pm$-chord coming from an $\alpha_\mp$-chord or an intersection point in the interior of $V$.
This means that the wrapped differential is of the form
$$
\partial_W=\left(\begin{array}{ccc}
  \partial_W^-   &  0 & 0 \\
  \partial_W^{-,0}   &  \partial_W^{0} & \partial_W^{+,0}\\
  0  &  0 & \partial_W^{+}
\end{array}\right).
$$
In the case of a $4$-dimensional Anosov Liouville domain and $L_1,L_2$ of the form $\mathcal{L}_\Lambda,\mathcal{L}_{\Lambda'}$, for $\Lambda\neq \Lambda'$, 
recall that these Lagrangians are actually disjoint. Thus $CW^*_0(\mathcal{L}_\Lambda,\mathcal{L}_{\Lambda'})=0$, and therefore
$$HW^*(\mathcal{L}_\Lambda,\mathcal{L}_{\Lambda'})=HW_-^*(\mathcal{L}_\Lambda,\mathcal{L}_{\Lambda'})\oplus HW_+^*(\mathcal{L}_\Lambda,\mathcal{L}_{\Lambda'})$$ if $\Lambda\neq \Lambda'$. If $\Lambda=\Lambda'$, then $\partial_W^0$ can be identified with the Morse differential for $\mathcal{L}_\Lambda$, and Lemma~\ref{lemma:disk} implies $\partial_W^{+,0}=\partial_W^{-,0}=0$ (cf. property $(C)$ in the proof of Theorem~\ref{thm:unit}).
\end{proof}

\begin{figure}
    \centering
    \includegraphics{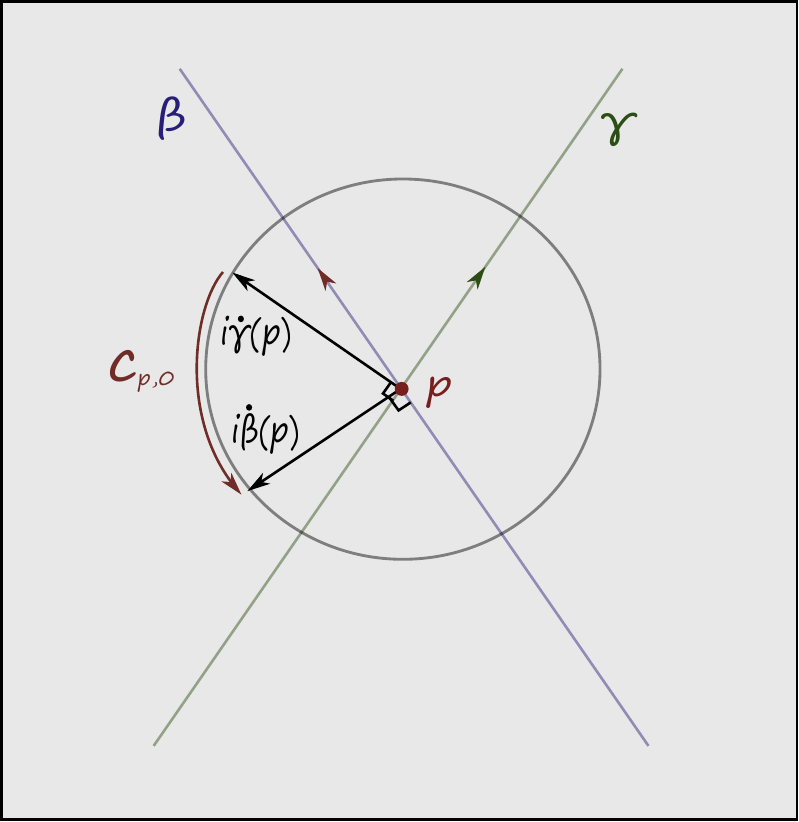}
    \caption{
    The chord $c_{p,0}$ from $\Lambda_\gamma$ to $\Lambda_\beta$, for $p\in \gamma \cap \beta$, joining $i\dot \gamma(p)$ to $i\dot \beta(p)$ along the $S^1$-fiber.}
    \label{fig:chordcp0}
\end{figure}

As particular cases of Theorem~\ref{thm:WFH}(1), we have the following computations as $\Z$-modules.

\begin{proposition}[McDuff domains: wrapped Floer cohomology] 
Consider a McDuff domain $V=[-1,1]\times S^*\Sigma$. Let $\Lambda_\gamma, \Lambda_\beta\subset M=S^*\Sigma$ be the positive conormal lifts of oriented closed geodesics $\gamma, \beta \subset \Sigma$, and $\mathcal{L}_\gamma,\mathcal{L}_\beta$ the corresponding Lagrangians. We denote by $\overline \beta$ the geodesic $\beta$ with opposite orientation.
    
    \smallskip
    
    \begin{enumerate}
    
    \item If $\gamma\neq \beta,\overline{\beta}$, then as a $\Z$-module we have
$$
HW^*(\mathcal{L}_\gamma,\mathcal{L}_\beta)\cong H^{>0}_*(\mathcal{L}(\gamma,\beta))\oplus \bigoplus_{p \in \gamma \cap \beta} \Z[t] \cdot c_p,
$$
where $\mathcal{L}(\gamma,\beta)$ is the space of free paths in $\Sigma$ 
from $\gamma$ to $\beta$, $H^{>0}_*(\mathcal{L}(\gamma,\beta))$ is the homology of $\mathcal{L}(\gamma,\beta)$ relative to the constant paths,  and $t^k c_p =: c_{p,k}$ is a generator corresponding to an intersection point $p \in \gamma \cap \beta$ obtained by the concatenation of paths $c_{p,k}=t^k*c_p$, where $t$ is the $S^1$-fiber and $c_p=c_{p,0}$ is the shortest positively oriented path on the $S^1$-fiber going from $\Lambda_\gamma$ to $\Lambda_\beta$ (see Figure \ref{fig:chordcp0}).

\begin{figure}
    \centering
    \includegraphics{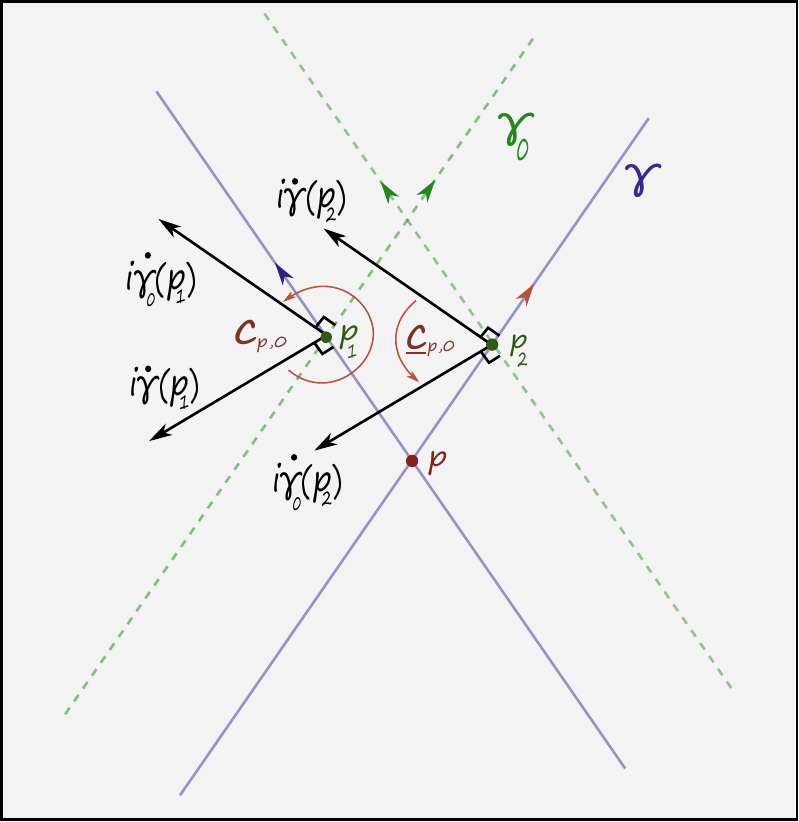}
    \caption{If $\gamma=\beta$ and $p$ is a self-intersection of $\gamma$, it resolves into two intersections $p_1$ and $p_2$ between $\gamma$ and a small push-off $\gamma_0$, respectively giving the chord $c_{p,0}$ and its complementary $\underline c_{p,0}$, both going from the conormal lifts $\Lambda_\gamma$ to $\Lambda_{\gamma_0}$. }
    \label{fig:pushoff}
\end{figure}

\begin{figure}
    \centering
    \includegraphics{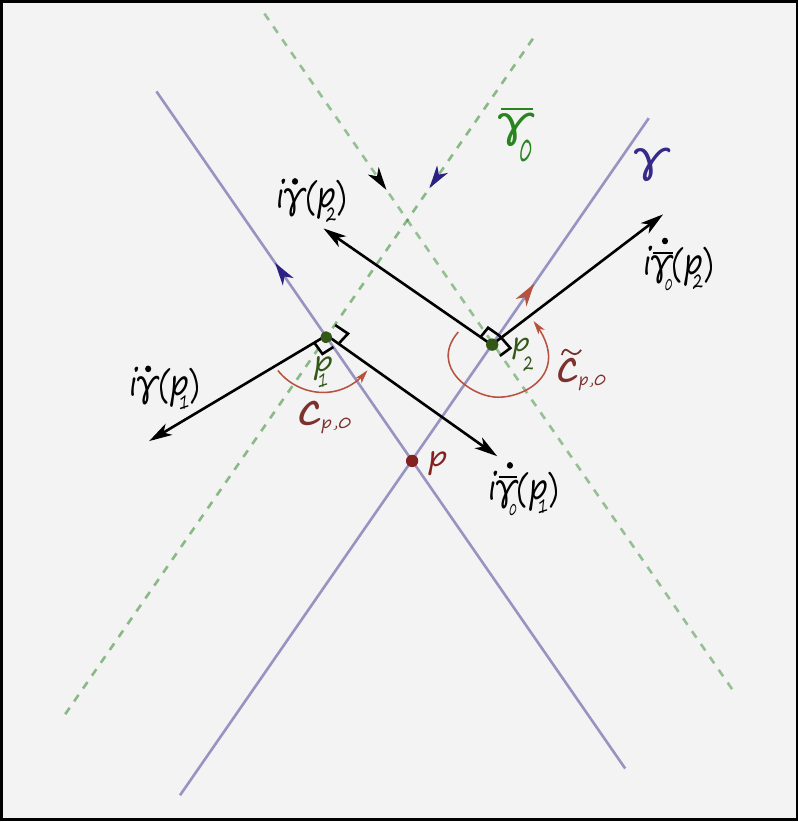}
    \caption{If $\beta=\overline{\gamma}$ and $p$ is a self-intersection of $\gamma$, the inverse $\overline{\gamma}_0$ of a small push-off $\gamma_0$ intersects $\gamma$ in $p_1,p_2$, and produces a chord $c_{p,0}$ from $i\dot{\gamma}(p_1)$ to $i\dot{\overline\gamma}_0(p_1)=-i\dot{\gamma}_0(p_1)$, and a chord $\widetilde{c}_{p,0}$ from $i\dot{\gamma}(p_2)$ to $i\dot{\overline\gamma}_0(p_2)=-i\dot{\gamma}_0(p_2)$.}
    \label{fig:pushoffbar}
\end{figure}

\smallskip

\item If $\gamma=\beta$, then
$$
HW^*(\mathcal{L}_\gamma,\mathcal{L}_\gamma)\cong H^{>0}_*(\mathcal{L}(\gamma,\gamma))\oplus H^*\big(S^1;\Z\big) \oplus \bigoplus_{p \in \Gamma_\gamma} \Z(t) \cdot c_p,$$
where $\Gamma_\gamma$ denotes the set of self-intersections of $\gamma$, and $t^k c_p =: c_{p,k}$ is defined via the concatenation $c_{p,k}=t^k*c_p$ for $k \in \Z$ (see Figure \ref{fig:pushoff}). 

\smallskip

\item If $\gamma=\overline{\beta}$, then
$$
HW^*(\mathcal{L}_\gamma,\mathcal{L}_{\overline{\gamma}})\cong H^{>0}_*(\mathcal{L}(\gamma,\gamma)) \oplus \bigoplus_{p \in \Gamma_\gamma} \left( \Z[t]\cdot c_p \oplus \Z[t] \cdot \widetilde{c}_p\right),$$
where $t^k \widetilde c_p = \widetilde c_{p,k}=t^k*\widetilde c_p$, and if $c_p$ joins positively $v \in S^*_p\Sigma$ to $w \in S^*_p\Sigma$, then $\widetilde c_p$ joins positively $-w$ to $-v$ (see Figure \ref{fig:pushoffbar}).
\end{enumerate}
\end{proposition}

\begin{proof}
Assume first that $\gamma\neq\beta,\overline{\beta}$.
We note that $\alpha_+$-Reeb chords $c_+$ between $\Lambda_\gamma$ and $\Lambda_\beta$ project to $\Sigma$ as \emph{binormal geodesics paths} (or \emph{binormal chords}) $c_+^\perp:[0,1]\rightarrow \Sigma$ between $\gamma$ and $\beta$, i.e., 
\begin{gather*}
\nabla_{\partial_t c_+^\perp}\partial_tc_+^\perp=0,\cr
c_+^\perp(0)\in \gamma, \qquad c_+^\perp(1)\in \beta,\cr
\partial_tc_+^\perp(0) \perp \gamma, \qquad \partial_tc_+^\perp(1) \perp \beta.
\end{gather*}
There is a unique such binormal chord in each free homotopy class of paths between $\gamma$ and $\beta$, i.e., in each connected component of $\mathcal{L}(\gamma,\beta)$. They are precisely the critical points of the length functional 
$$
L:\mathcal{P}(\gamma,\beta)\rightarrow \mathbb{R},\qquad
L(x)=\int_{[0,1]}|\dot{x}(t)|dt.
$$
Note that since $\gamma$ and $\beta$ cannot possibly intersect tangentially, every $c_+^\perp \in \mbox{crit}(L)$ has positive length, i.e., $L\big(c_+^\perp\big)>0$. We may then adapt the analysis of \cite{AS}, by considering suitable moduli spaces of Floer strips with mixed boundary conditions (see Figure \ref{fig:mixedmoduli}), and conclude that
$$
HW^*_+(\mathcal{L}_\gamma,\mathcal{L}_\beta)\cong H_*^{>0}(\mathcal{L}(\gamma,\beta)).
$$

\begin{figure}
    \centering
    \includegraphics[width=0.5\linewidth]{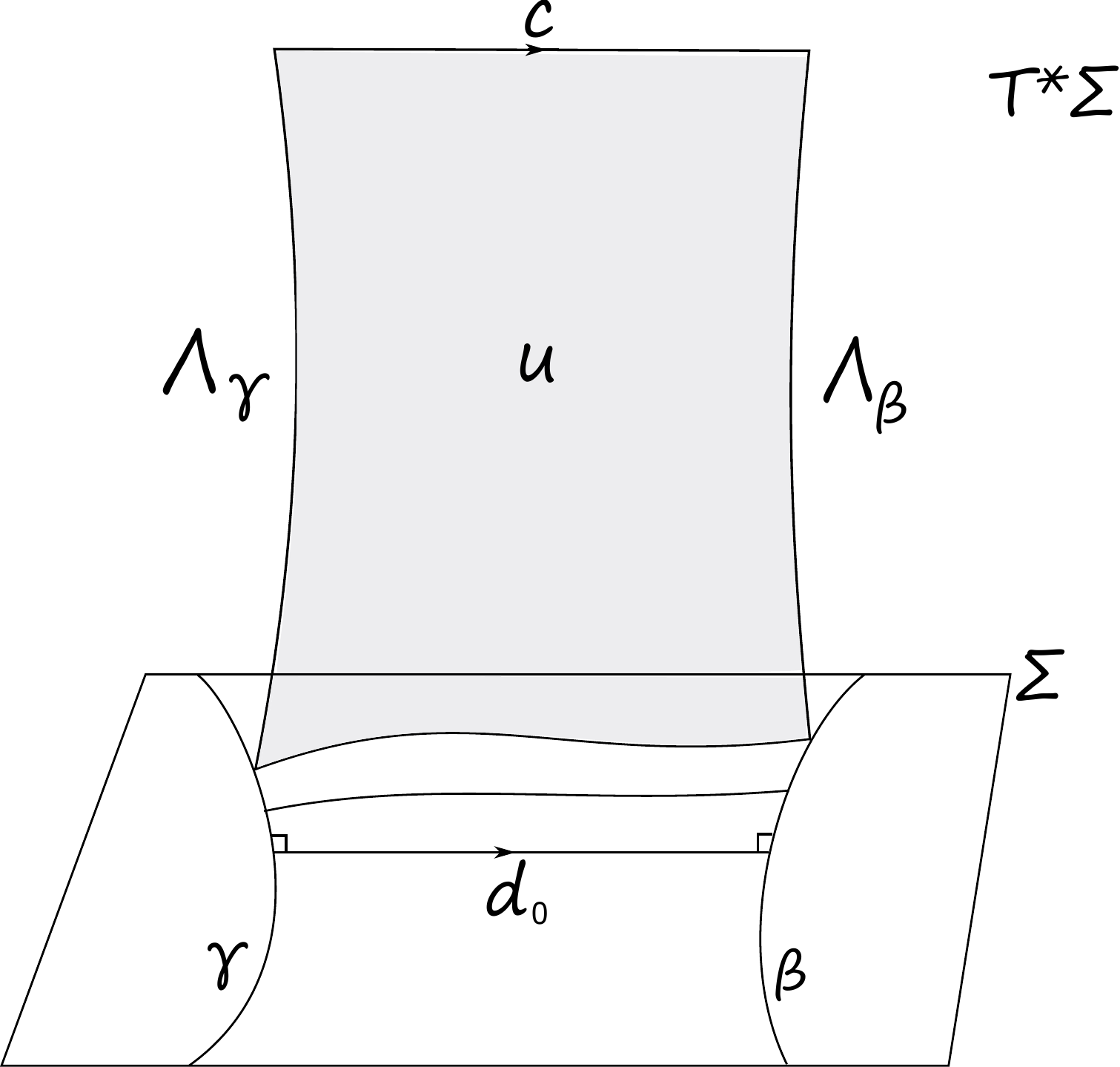}
    \caption{For $c$ a $\alpha_+$-chord between $\Lambda_\gamma$ and $\Lambda_\beta$, $d_0 \in \mbox{crit}(L)$ a binormal path between $\gamma$ and $\beta$, and the Hamiltonian $H(q,p)=|p|^2/2$, the counts of $H$-Floer solutions $u: [0,+\infty)\times [0,1]\rightarrow T^*\Sigma$ with Lagrangian boundary in $\mathcal{L}_\gamma$, $\mathcal{L}_\beta$ and $\Sigma$, which approach $c$ at infinity, and whose intersection with $\Sigma$ lies in the unstable manifold of $d_0$ (with respect to the gradient flow of $L$), gives a chain map inducing an isomorphism $HW^*_+(\mathcal{L}_\gamma,\mathcal{L}_\beta)\cong H_*^{>0}(\mathcal{L}(\gamma,\beta))$.}
    \label{fig:mixedmoduli}
\end{figure}

Similarly, $\alpha_-$-Reeb chords $c_-$ between $\Lambda_\gamma$ and $\Lambda_\beta$ project to $\Sigma$ as intersection points $p\in \gamma \cap \beta$. Since $\alpha_-$-orbits are all closed and correspond to fibers of the prequantization bundle, to each $p \in \gamma \cap \beta$ there corresponds a family of Reeb chords $c_{p,k}, k \in \N \cup\{0\}$, one for each iteration along the $S^1$-fiber, so that $c_{p,k}=t^k*c_p$ where $t$ denotes the fiber and $c_p$ the shortest positively oriented path along the fiber going from $\Lambda_\gamma$ to $\Lambda_\beta$. Therefore,
$$
CW^*_+(\mathcal{L}_\gamma,\mathcal{L}_\beta) = \bigoplus_{p \in \gamma \cap \beta} \bigoplus_{k=0}^\infty \Z\cdot c_{p,k} = \bigoplus_{p \in \gamma \cap \beta} \Z[t] \cdot c_p.$$

A Floer strip between two such chords projects to $\Sigma$ as a bigon between $\gamma$ and $\beta$, which does not exist since geodesics minimize the length in their free homotopy class. This implies
$$
HW^*_+(\mathcal{L}_\gamma,\mathcal{L}_\beta)\cong CW^*_+(\mathcal{L}_\gamma,\mathcal{L}_\beta),
$$
and the proposition follows for the case $\gamma\neq \beta,\overline{\beta}$. The remaining cases follow by perturbation of the self-intersections of $\gamma$, taking orientations into consideration, as depicted in Figures \ref{fig:pushoff} and \ref{fig:pushoffbar}.
\end{proof}

 \begin{proposition}[Torus bundle domains: wrapped Floer cohomology]\label{prop:torus_bundles_wfh} Consider a torus bundle domain $V=[-1,1]\times M$. Let $\mathcal{O},\mathcal{O}'\subset \mathbb{T}^2$ be two periodic orbits of the hyperbolic monodromy matrix $A$, and $\mathcal{L}_\mathcal{O},\mathcal{L}_{\mathcal{O}'}$ the corresponding Lagrangians. Then there is no differential in the wrapped Floer cochain complex between $\mathcal{L}_\mathcal{O},\mathcal{L}_{\mathcal{O}'}$ (save the Morse differential if $\mathcal{O}=\mathcal{O}'$), and we have the following.

\smallskip

\begin{figure}
    \centering
    \includegraphics[width=0.7 \linewidth]{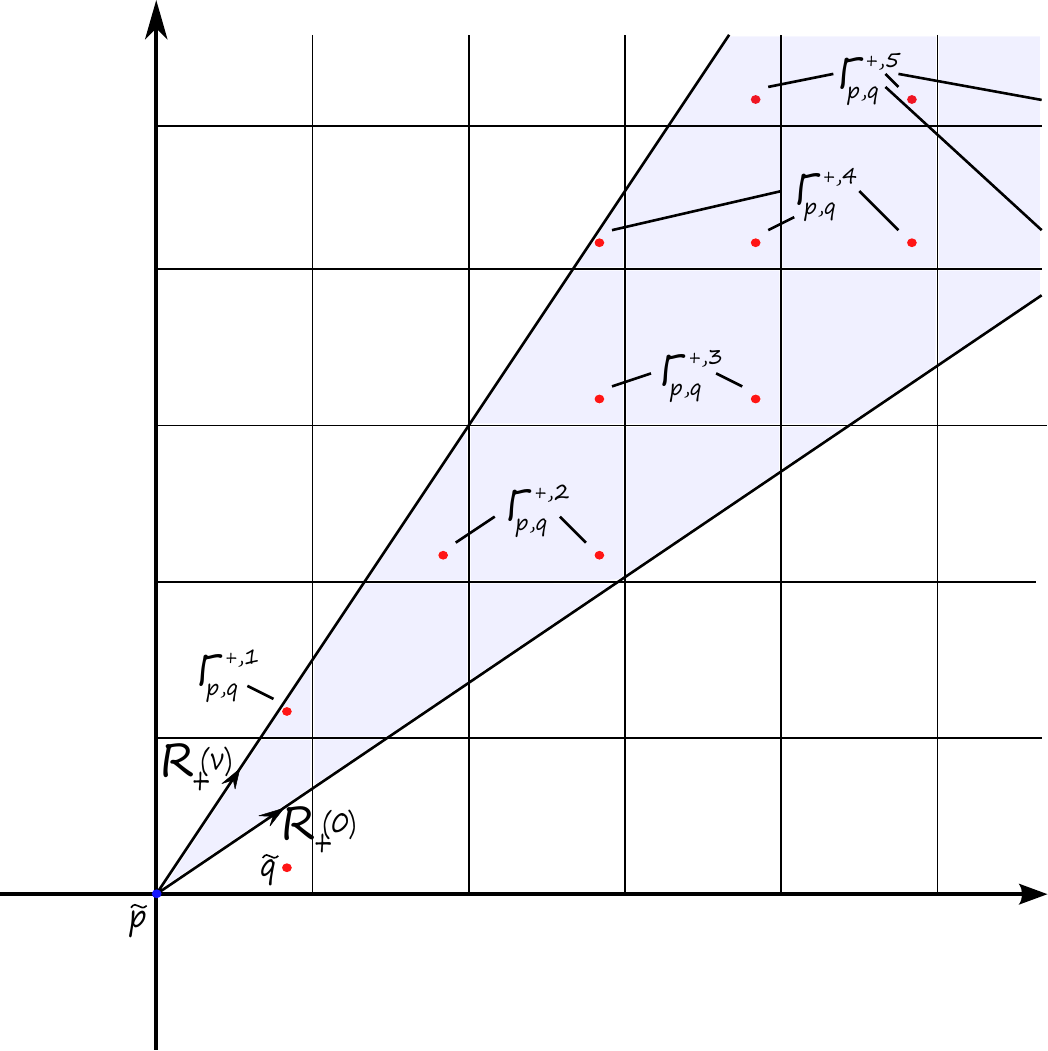}
    \caption{
    The linear path between $\tilde p$ and a translate of $\tilde q$ lying inside the cone $\mathbb{R}^{\geq 0}\langle v,w\rangle$, where $v=R_+(0),w=R_+(\nu)$, gives a $\alpha_+$-chord between $p$ and $q$, lying on a $\mathbb{T}^2$-fiber which depends on the slope of the path. In the figure, we depict the sets $\Gamma^{+,k}_{p,q}=\Gamma^{+,\leq k}_{p,q}\backslash \Gamma^{+,\leq k-1}_{p,q}$, where $k$ is the box-length of the corresponding chords. There is a similar description for the ``minus'' end.}
    \label{fig:chords}
\end{figure}

\begin{enumerate}
    \item If $\mathcal{O}\neq \mathcal{O}'$, then
$$
HW^*(\mathcal{L}_\mathcal{O},\mathcal{L}_{\mathcal{O}'})=\bigoplus_{p \in \mathcal{O},q\in \mathcal{O}'}\left(\bigoplus_{z^+\in \Gamma_{p,q}^+}\Z\cdot z^+ \oplus \bigoplus_{z^-\in \Gamma_{p,q}^-}\Z\cdot z^-\right), 
$$
where $\Gamma^\pm_{p,q}$ is the set of straight lines in $\mathbb{T}^2=\mathbb R^2/\mathbb{Z}^2$ from $p$ to $q$ with directions in the positive cone spanned by the Reeb vector fields $R_\pm(0)$ and $R_\pm(\nu)$ from~\eqref{reeb2}.

\smallskip

\item If $\mathcal{O}=\mathcal{O}'$, then
$$
HW^*(\mathcal{L}_\mathcal{O},\mathcal{L}_{\mathcal{O}})=H^*\big(S^1;\Z\big)\oplus\bigoplus_{p,q\in \mathcal{O}}\left(\bigoplus_{z^+\in \Gamma_{p,q}^+}\Z\cdot z^+ \oplus \bigoplus_{z^-\in \Gamma_{p,q}^-}\Z\cdot z^-\right).
$$
\end{enumerate}
\end{proposition}

\begin{remark} \label{rem:torus}
If $\tilde p,\tilde q \in \mathbb{R}^2$ are lifts of $p,q \in \mathbb{T}^2=\mathbb R^2/\mathbb{Z}^2$, then $\Gamma^\pm_{p,q}$ is in canonical one-to-one correspondence with the set of $\mathbb{Z}^2$-translates $\tilde q + (m,n)$ of $\tilde q$, with $m,n\in \mathbb{Z}$, which lie inside the positive cone based at $\tilde p$ and spanned by the Reeb vector fields $R_\pm(0)$ and $R_\pm(\nu)$ from~\eqref{reeb2}; see Figure \ref{fig:chords}. 

Note that $\Gamma^\pm_{p,q}=\bigcup_{k=0}^\infty \Gamma^{\pm,\leq k}_{p,q}$, where $\Gamma^{\pm,\leq k}_{p,q}$ is the set of translates $\tilde q + (m,n)$ with $\vert (m,n)\vert = \max\{\vert m \vert,\vert n \vert\}\leq k$. This filtration is equivalent (up to a factor of $\sqrt{2}$ and a bounded error) to the action filtration and $\#\Gamma^{\pm,\leq k}_{p,q}$ grows quadratically with $k$. 
\end{remark}

\begin{proof}
Put $\Lambda=\Lambda_\mathcal{O}$, $\Lambda'=\Lambda_{\mathcal{O}'}$ for two orbits $\mathcal{O},\mathcal{O}'$ of the hyperbolic monodromy matrix. We will first show that the differential $\partial_W$ vanishes identically if $\mathcal{O}\neq \mathcal{O}'$. 
As the differential has no mixed terms between the two boundary components of $V$, it suffices to check $\partial_W^+=0$ (the argument for $\partial_W^-$ is analogous).
Given a Floer strip $u$ between two $\alpha_+$-chords $c_1,c_2$, we may first project it on $M$. Moreover, since it is simply connected, we may lift it to the cover $\mathbb{R}^2\times \mathbb{R}$ of $M$, and then further project it to the $\mathbb{R}^2$-factor. This gives a loop $a=\widetilde c_1*\widetilde c_2^{-1}$ which is obtained by concatenating two paths $\widetilde{c}_1$ and the opposite of $\widetilde{c}_2$, lifting $c_1$ and $c_2$, respectively, each of which is linear and follows the direction of the Reeb field $R_+$. Then, the only possibility is that $\widetilde c_1=\widetilde c_2$ This implies that $c_1,c_2$ have the same slope, hence lie in the same $\mathbb{T}^2$-fiber, and therefore $c_1=c_2$. 
But then the strip is the trivial strip, which is not counted in the differential. 

It remains to describe the generators of $HW^*(\mathcal{L}_\mathcal{O},\mathcal{L}_{\mathcal{O}'})$. For this, note that each Reeb chord for $R_\pm$ from $\mathcal{O}$ to $\mathcal{O}'$ has a unique representative in $\mathbb{T}^2\times [0,\nu)$ by a linear path from some $p\in\mathcal{O}$ to some $q\in\mathcal{O}'$ in direction $R_\pm(z)$ from~\eqref{reeb2}, for some $z\in[0,\nu)$. 
This shows that the nonconstant generators are indexed by the filtered sets $\Gamma_{p,q}^{\pm}$ as described in the statement of the Theorem, for $p \in \mathcal{O}$ and $q\in \mathcal{O}'$. In the case $\mathcal{O}=\mathcal{O}'$ we have in addition the constant intersection points, which after a Morse perturbation give the $H^*\big(S^1;\Z\big)$ summand. This concludes the proof.
\end{proof}

\subsection{The product structure}

In this section, we study the product $\mathfrak{m}=\mathfrak{m}_2$ on the algebra 
$$A=\bigoplus_{\Lambda,\Lambda'} HW^*(\mathcal{L}_\Lambda,\mathcal{L}_{\Lambda'}),
$$
for any Anosov Liouville domain, and complete the proof of Theorem~\ref{thm:WFH}. We have a natural splitting of $A$ as a $\mathbb{Z}$-module
$$
A=I_-\oplus A_0\oplus I_+,
$$
where 
$$
I_\pm:=\bigoplus_{\Lambda,\Lambda'} HW_\pm^*(\mathcal{L}_\Lambda,\mathcal{L}_{\Lambda'}),\qquad
A_0:=\bigoplus_{\Lambda} HW_0^*(\mathcal{L}_\Lambda,\mathcal{L}_{\Lambda}),
$$
and we let
$$
A_{\pm}=A_0\oplus I_\pm.
$$
Note that $A_0=\bigoplus_{\Lambda} H^*\big(S^1;\Z\big)$. The following is an open-string analogue of the statement for symplectic cohomology of Theorem~\ref{thm:RFH-SH}.

\begin{proof}[Proof of Theorem~\ref{thm:WFH}(2): the ring structure]
The fact that the product on $A_0$ is as claimed follows from the standard fact that the isomorphism with Morse homology preserves the product structure.

\begin{figure}
    \centering    \includegraphics[width=0.8\linewidth]{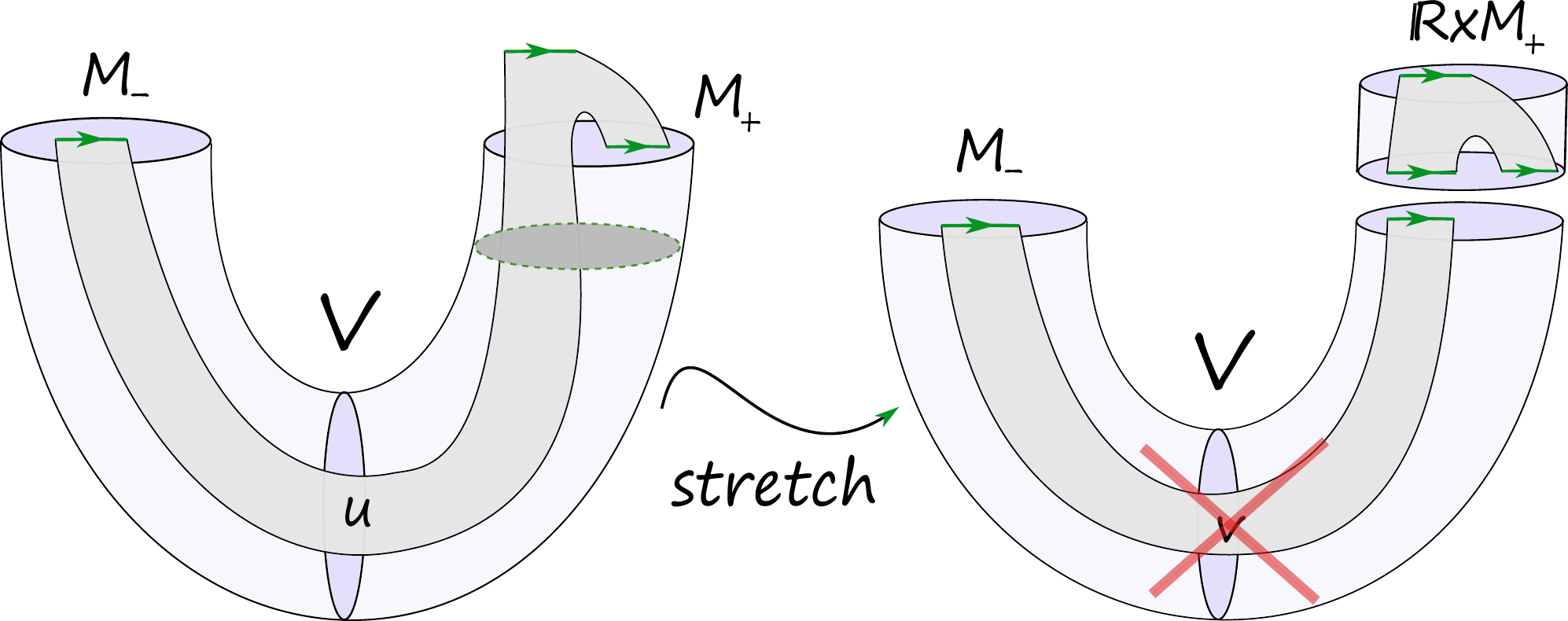}
    \caption{Stretching the neck gives a disk $v$ with only positive ends, contradicting the Topological Disk Lemma \ref{lemma:disk}.}
    \label{fig:stretch}
\end{figure}

To see the remaining claims, we argue as follows. Assume by contradiction that there exists a $3$-punctured disk $u:\mathcal{D}\rightarrow V$ contributing to $\mathfrak{m}$ which has mixed ends (i.e., at least one asymptotic in $M_+$, and at least one in $M_-$), two of which are positive, and one which is negative. 
After stretching the neck along contact-type slices slightly to the interior of $M_-$ and/or $M_+$, the resulting broken disk configuration contains a punctured disk $v$ with only positive ends asymptotic to Reeb chords in the contact-type slices; see Figure \ref{fig:stretch}. As in Step 1 of the proof of Theorem \ref{thm:unit}, the projection of $v$ to $M$ gives rise to a topological disk that contradicts Lemma \ref{lemma:disk}. A product of two chords in $M_+$ or $M_-$ with output in the constants is excluded similarly. This finishes the proof of Theorem~\ref{thm:WFH} of the Introduction.
\end{proof}

\smallskip

\textbf{McDuff domains: open-string products.} We now focus on the algebra $A$ corresponding to the McDuff domains. First, we compute the coefficients of $\mathfrak{m}:A\otimes A\rightarrow A$ on $I_+$, which is the free $\Z$-module with generators all the positive length binormal chords between $\gamma$ and $\beta$, for varying closed geodesics $\gamma$ and $\beta$. Consider $\pi_0\mathcal{L}(\gamma,\beta)$, the set of free homotopy classes of paths from $\gamma$ to $\beta$ (which is a  countable set). We identify elements in $\pi_0\mathcal{L}(\gamma,\beta)$ with their unique binormal chord representative, and view $CW_+^*(\mathcal{L}_\gamma,\mathcal{L}_\beta)$ as the free $\Z$-module over $\pi_0\mathcal{L}(\gamma,\beta)$. For $c$ a binormal chord, we denote by $\overline{c}$ the binormal chord with opposite orientation (so that if $c \in \pi_0\mathcal{L}(\gamma,\beta)$, then $\overline{c} \in \pi_0\mathcal{L}(\beta,\gamma)$). 


We first consider the product $\mathfrak{m}_+=\mathfrak{m}\vert_{I_+\otimes I_+}: I_+ \otimes I_+ \rightarrow I_+,$ defined by counting $3$-punctured Floer disks in the completion $\R\times M$ of $V=[-1,1]\times M$ with two inputs and one output asymptotic to Hamiltonian chords near $M_+=\{1\}\times M$. By the maximum principle, these Floer disks are 
contained in $\R_+\times M$. By a deformation argument as in~\cite{BO09}, their counts agree with the counts of index $1$ holomorphic disks in the symplectization $\mathbb{R}_+\times S^*\Sigma$ with two positive punctures and one negative puncture asymptotic to Reeb chords in the unit cotangent bundle $S^*\Sigma$. Similarly, we have a coproduct $\widecheck{\mathfrak{m}}_+: I_+ \rightarrow I_+ \otimes I_+$, counting index $1$ holomorphic disks in $\mathbb{R}_+\times S^*\Sigma$ with one positive puncture and two negative ones.  
We also have an open-string cobracket $\Delta: I_+\rightarrow I_+ \otimes I_+,$ defined as follows. Let $\gamma_0,\gamma_1,\gamma_2\subset \Sigma$ be geodesics, and consider a binormal chord $c \in \pi_0\mathcal{L}(\gamma_0,\gamma_2)$. Let $x_1,\dots, x_n$ be the intersection points of $c$ with $\gamma_1$. At each $x_i$, $c$ may be split into two (oriented) paths $p_1^i \in \mathcal{L}(\gamma_0,\gamma_1)$ and $p_2^i \in \mathcal{L}(\gamma_1,\gamma_2)$, and we let $c_1^i$ and $c_2^i$ be the corresponding binormal chord representatives. We define
$$
\Delta^{\gamma_1}(c):=\sum_{i=1}^n c_1^i \otimes c_2^i 
$$
and $\Delta:=\sum_\gamma \Delta^\gamma$. For $c_1,c_2,c_3$ binormal chords, we denote by $\langle \mathfrak{m}_+(c_1,c_2),c_3\rangle$ the coefficient in front of $c_3$ of $\mathfrak{m}_+(c_1,c_2)$. Similarly, we define $\langle \widecheck{\mathfrak{m}}_+(c_1), c_2\otimes c_3\rangle$, and $\langle \Delta(c_1),c_2\otimes c_3\rangle$, as the coefficients of $\widecheck{\mathfrak{m}}_+(c_1)$ (resp.\ $\Delta(c_1)$) in front of $c_2\otimes c_3$.

\begin{figure}
    \centering
    \includegraphics[width=1\linewidth]{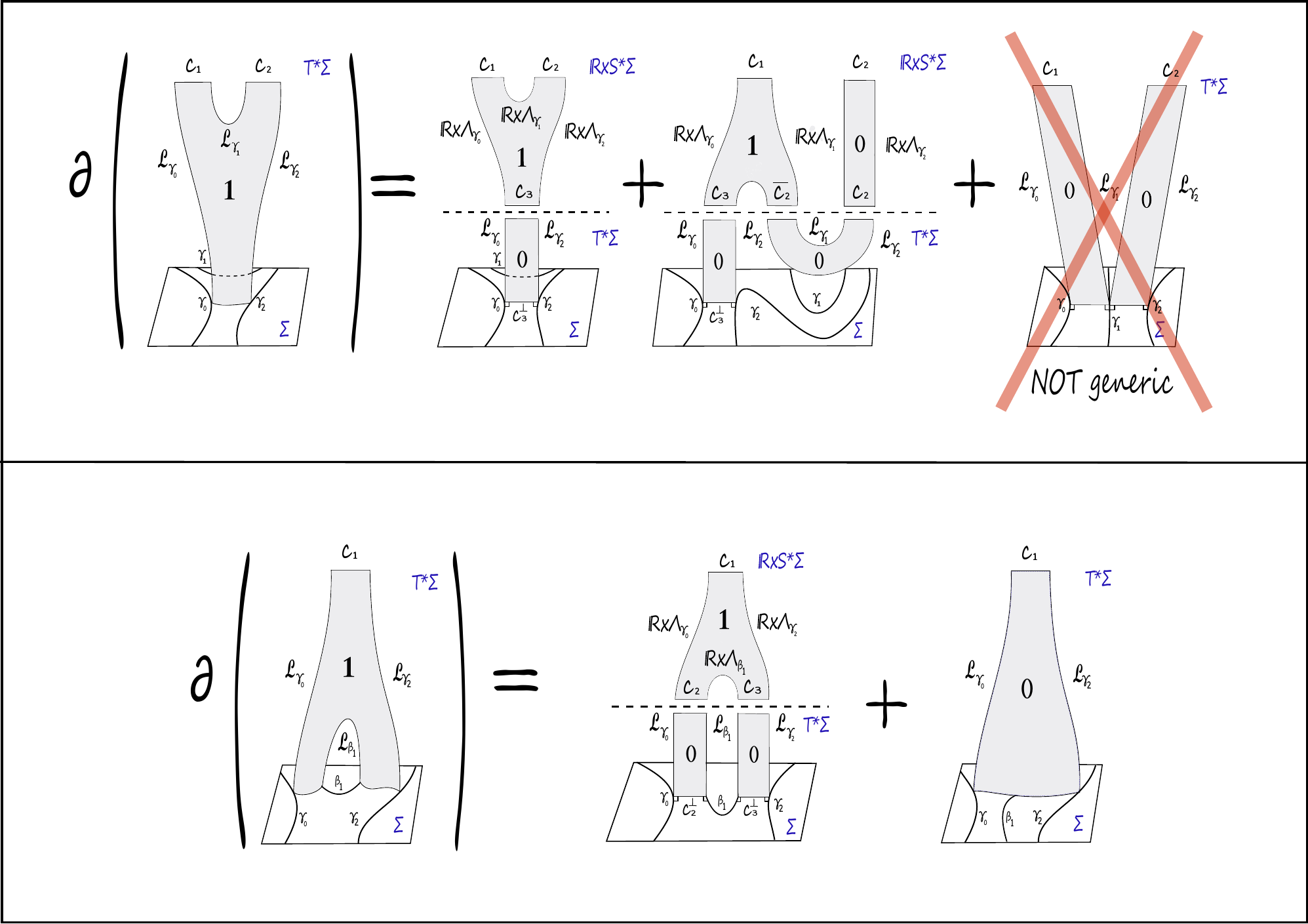}
    \caption{The degenerations of $1$-dimensional moduli spaces of curves with Lagrangian boundary on the zero section gives the relations between the structure constants of $\mathfrak{m}_+$ and $\Delta$.}
    \label{fig:moduli}
\end{figure}

The following reduces the computation of $\mathfrak{m}_+$ to that of the string operation $\Delta$:

\begin{proposition}
 For binormal chords $c_1,c_2,c_3$, the following formulas hold with suitable signs:
 \begin{align*}
 \langle \mathfrak{m}_+(c_1,c_2),c_3 \rangle 
 &= \pm \langle \widecheck{\mathfrak{m}}_+(c_1), c_3\otimes\overline{c_2}\rangle  \pm \langle \widecheck{\mathfrak{m}}_+(c_2), \overline{c_1}\otimes c_3\rangle,\cr
 \langle \widecheck{\mathfrak{m}}_+(c_1),c_2\otimes c_3\rangle
 &= \langle \Delta(c_1),c_2\otimes c_3\rangle\,.
 \end{align*}
\end{proposition}

\begin{proof}
The proof is a straightforward adaptation of that in \cite{CL09} (see also \cite{Sch}). One counts the boundary configurations in a $1$-dimensional moduli space of curves with boundary on the zero section to obtain the above relations; see Figure \ref{fig:moduli}. We omit further details. 
\end{proof}

\begin{figure}
    \centering
    \includegraphics[width=0.6\linewidth]{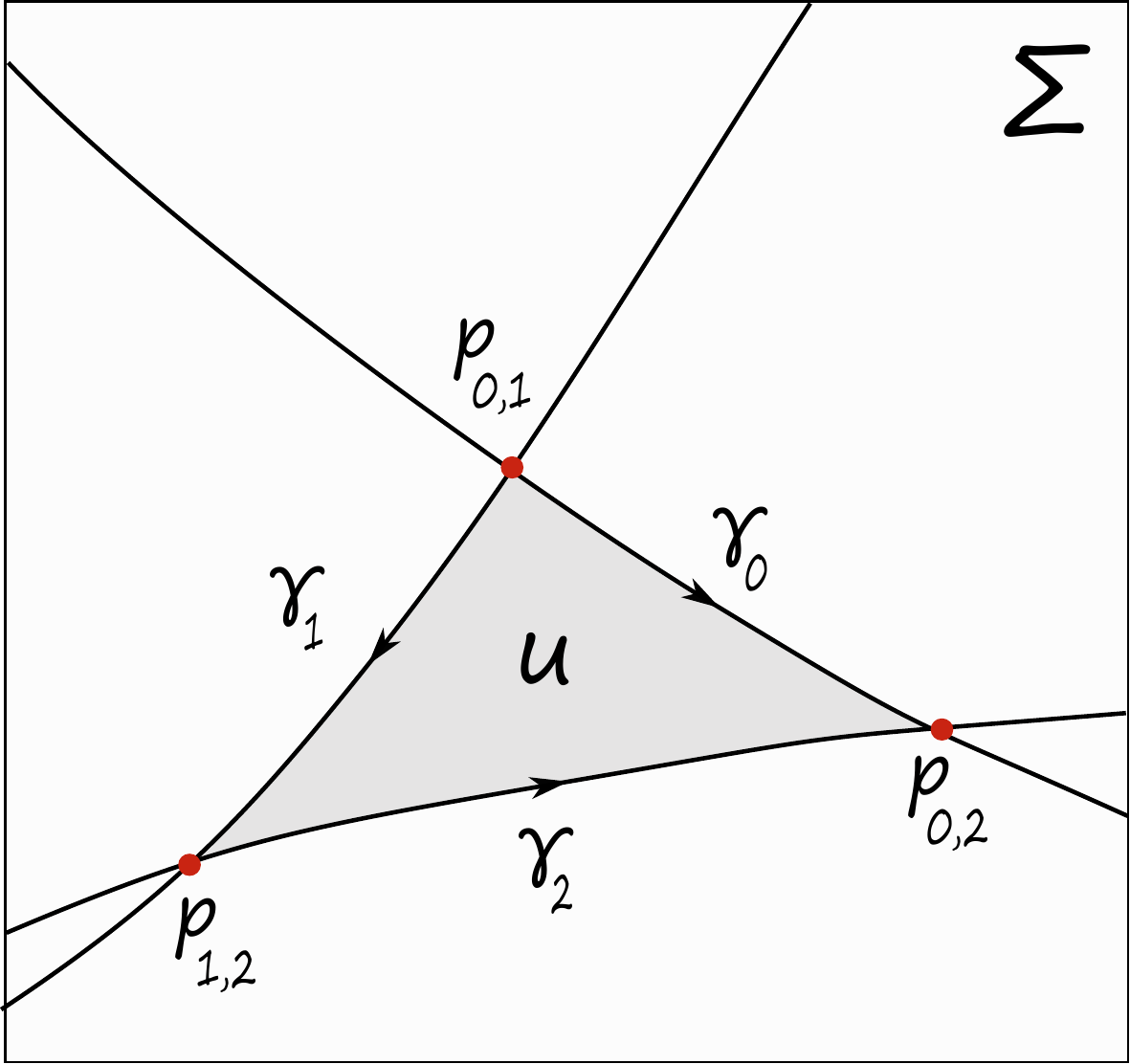}
    \caption{The product on $I_-$ is given by counts of immersed geodesic triangles on $\Sigma$.}
    \label{fig:geodesic_triangle}
\end{figure}

We now consider the product $\mathfrak{m}_-=\mathfrak{m}\vert_{I_-\otimes I_-}: I_-\otimes I_-\rightarrow I_-$, where $I_-$ is generated by $c_{p,k}$ for  intersection points $p$ of geodesics $\gamma, \beta$ for varying $\gamma,\beta$, and $k\in \mathbb{N}\cup \{0\}$.
Again, after a deformation we may assume that  $\mathfrak{m}_-$ counts index $1$ holomorphic curves in the symplectization $\R \times S^*\Sigma$ of the prequantization circle bundle $S^*\Sigma$ with two positive punctures and one negative puncture. We can also arrange that the almost complex structure on $J$ on $\R \times S^*\Sigma$ is such that the projection $\R \times S^*\Sigma \rightarrow \Sigma$ is $(J, j)$-holomorphic, where $j$ is the complex structure on the Riemann surface $\Sigma$. For simplicity, we will only consider the product between three Lagrangian cylinders corresponding to geometrically distinct geodesics.

\begin{figure}
    \centering
    \includegraphics[width=0.59\linewidth]{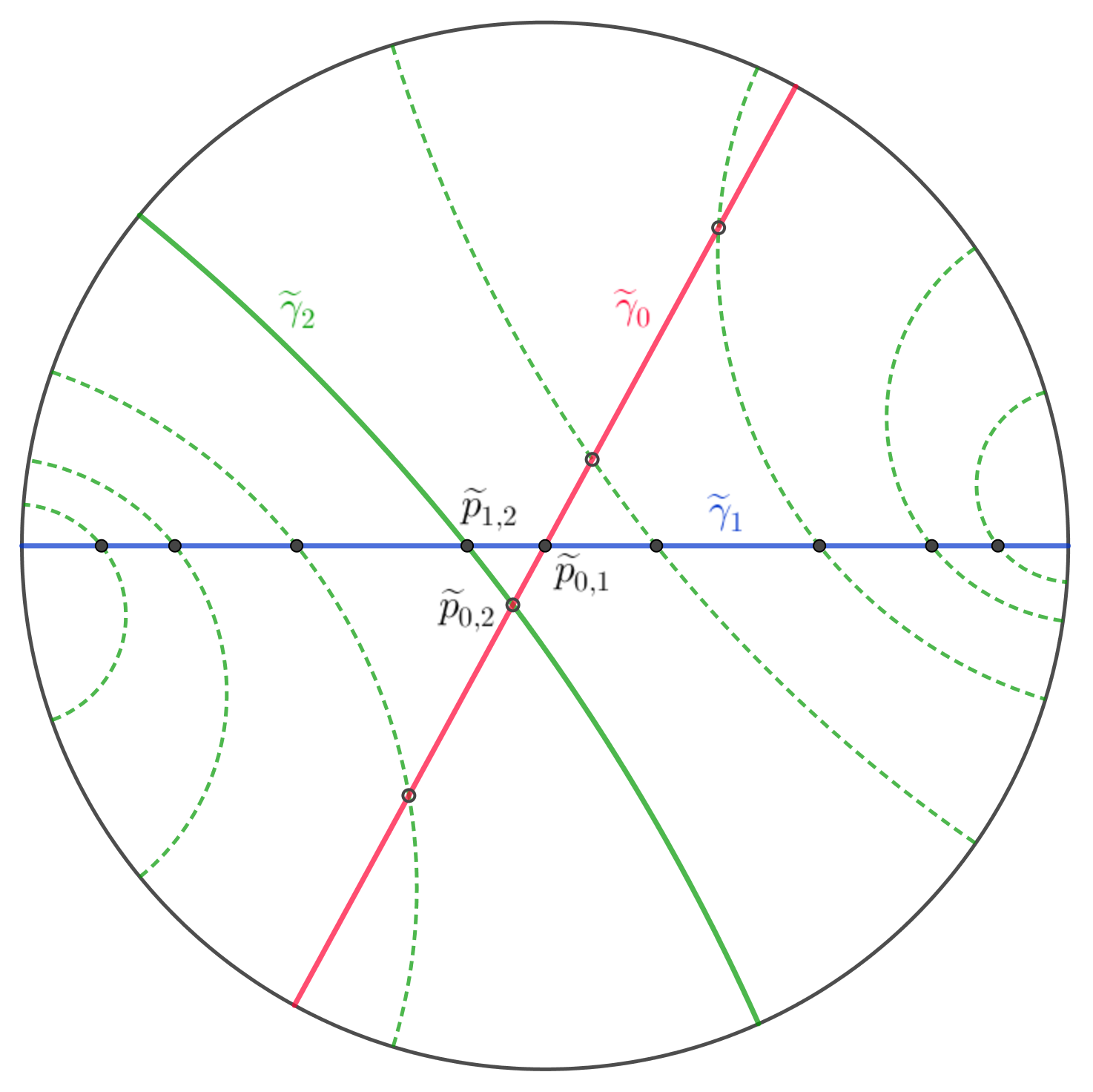}
    \caption{Geodesic triangles between lifts of $\gamma_0$, $\gamma_1$ and $\gamma_2$. The dashed geodesics are hyperbolic translates of $\widetilde{\gamma}_2$ along $\widetilde{\gamma}_1$. In this example, there are 4 geodesic triangles contributing to the product.}
    \label{fig:geodesics}
\end{figure}

\begin{proposition}
Let $\gamma_0,\gamma_1,\gamma_2$ be three geometrically distinct closed oriented geodesics on $\Sigma$ intersecting at $p_{0,1} \in \gamma_0 \cap \gamma_1$, $p_{1,2} \in \gamma_1 \cap \gamma_2$ and $p_{0,2} \in \gamma_0 \cap \gamma_2$, together with $k_{0,1}, k_{1,2}, k_{0,2} \in \N\cup\{0\}$. The coefficient $$\mathfrak{m}_-^{012} := \left\langle \mathfrak{m}_- \big( t^{k_{0,1}} c_{p_{0,1}}, t^{k_{1,2}} c_{p_{1,2}} \big), t^{k_{0,2}} c_{p_{0,2}} \right\rangle \in \Z$$ is zero unless $k_{0,2} = k_{0,1} + k_{1,2}$, in which case it is equal to the number of immersed geodesic triangles with cyclically ordered vertices $p_{0,1}, p_{1,2}, p_{0,2}$ and edges on the $\gamma_i$'s as in Figure~\ref{fig:geodesic_triangle}, counted with appropriate signs.
\end{proposition}

\begin{proof}
Let $\mathfrak{T}$ be the moduli space of holomorphic triangles defining $\mathfrak{m}_-^{012}$, and $\mathcal{T}$ be the moduli space of immersed geodesic triangles on $\Sigma$ as above. If $u \in \mathfrak{T}$ is a holomorphic triangle, it projects to a geodesic triangle on $\Sigma$ with the required properties. This defines a natural map $\pi : \mathfrak{T} \rightarrow \mathcal{T}$. We will only show that $\pi$ is a bijection and we leave it to the reader to work out the relevant orientations and signs.

We can lift $u \in \mathfrak{T}$ to a holomorphic triangle in $\R \times S^*\mathbb{H}$ whose boundaries lie in the Lagrangian cylinders over the unit conormals of lifts $\widetilde{\gamma}_i$ in $\mathbb{H}$ of the geodesics $\gamma_i$. By~\eqref{eq:preinH}, the contact structure on $S^*\mathbb{H} \cong \mathbb{H} \times S^1$ becomes $$\widetilde{\alpha}_\mathrm{pre} = \frac{1}{y} dx + d\varphi,$$ where $\varphi$ denotes $S^1$ coordinate. Under the biholomorphism $\mathbb{H} \cong \mathbb{D}$, it further becomes $$\widetilde{\alpha}_\mathrm{pre} = \frac{2r^2}{1-r^2}d\theta + dt,$$ where $(r, \theta)$ are the usual polar coordinates on $\mathbb{D}$ and $t$ denotes the $S^1$ coordinate on $\mathbb{D} \times S^1 \cong S^*\mathbb{D}$. Note that this is simply the $S^1$-contactization of the Liouville manifold $\big(\mathbb{D}, \lambda_{\mathbb{D}}\big)$ where $\lambda_\mathbb{D} := \frac{2r^2}{1-r^2}d\theta$. We can further lift $u$ to the symplectization of the $\R$-contactization of $\big(\mathbb{D}, \lambda_{\mathbb{D}}\big)$, with boundary conditions on the Lagrangian cylinders $L_i$ over suitable lifts of the unit conormals of the geodesics $\widetilde{\gamma}_i$. These Legendrian lifts in $\mathbb{D} \times \R$ only depend on the $k_{i,j}$'s up to a global translation in the $t$-direction (which coincides with the Reeb direction). This readily implies that $k_{0,2}= k_{0,1} + k_{1,2}$. By~\cite[Theorem 2.1.]{DR16}, for a suitable choice of almost complex structure on $\R \times \mathbb{D} \times \R$, there is a one-to-one correspondence between
\begin{enumerate}
\item SFT-type holomorphic triangles in $\R \times \mathbb{D} \times \R$ between the $L_i$'s, and
\item Holomorphic triangles in $\mathbb{D}$ between the $\widetilde{\gamma}_i$'s.
\end{enumerate}
Since the $\widetilde{\gamma}_i$'s bound a unique holomorphic triangle in $\mathbb{D}$ (which is embedded and cut-out transversally), this implies that $\pi$ is injective. This also shows that $\pi$ is surjective, as any immersed geodesic triangle $\Delta \in \mathcal{T}$ lifts to an embedded geodesic triangle $\widetilde{\Delta} \subset \mathbb{D}$ between lifts $\widetilde{\gamma}_i$ of the $\gamma_i$'s. By the uniformization theorem, $\widetilde{\Delta}$ can be identified with a holomorphic map from a disk with three boundary punctures to $\mathbb{D}$, which is cut-out transversally. By the above correspondence, it lifts to a SFT-type holomorphic triangle between the Lagagrangians $L_i$ which is cut-out transverally and which projects to a holomorphic triangle in $\R \times S^*\Sigma$ contributing to the coefficient $\mathfrak{m}_-^{012}$.
\end{proof}

\begin{figure}
    \centering
    \includegraphics[width=0.65\linewidth]{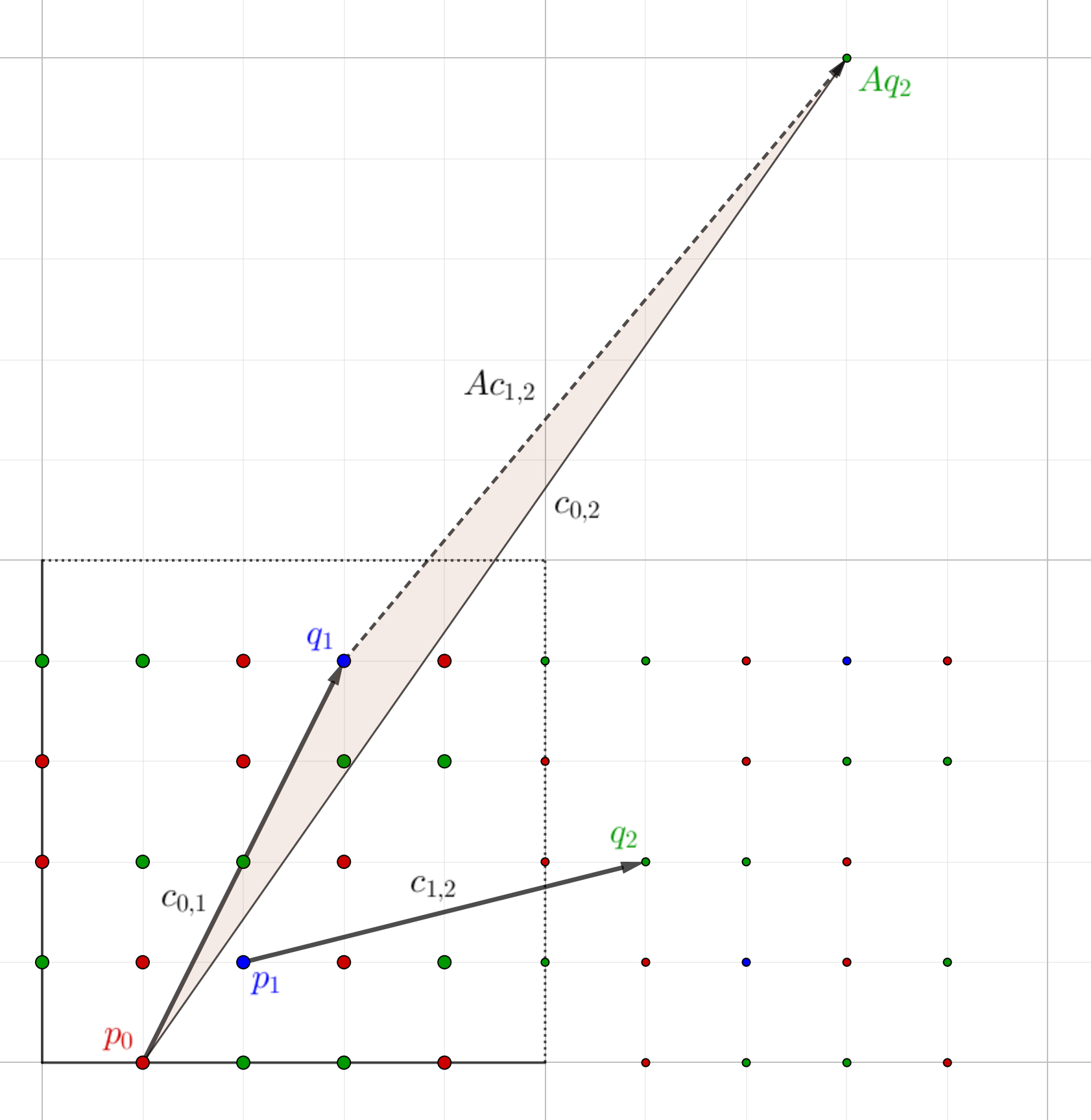}
    \caption{Torus bundle domain: three orbits $\textcolor{red}{\mathcal{O}_0}$, $\textcolor{blue}{\mathcal{O}_1}$ and $\textcolor{green}{\mathcal{O}_2}$, together with Reeb chords $c_{0,1}$ and $c_{1,2}$, and a potential candidate for a homolorphic triangle contributing to the product $\mathfrak{m}\big(c_{0,1},c_{1,2} \big)$.}
    \label{fig:torustriangle}
\end{figure}

\begin{remark}
It can further be shown that the number of immersed geodesic triangles in $\mathcal{T}$ only depends on the length of $\gamma_1$, the angle between the $\gamma_0$ and $\gamma_1$ at $p_{0,1}$, the angle between $\gamma_1$ and $\gamma_2$ at $p_{1,2}$, and the distance between $p_{0,1}$ and $p_{1,2}$. To see it, it suffices to lift $\gamma_0$ and $\gamma_1$ to $\mathbb{D}$ such that $p_{0,1}$ is lifted to $0 \in \mathbb{D}$. The possible lifts of $\gamma_2$ all differ by the action of the hyperbolic translation of length $\ell(\gamma_1)$ along $\widetilde{\gamma}_1$, see Figure~\ref{fig:geodesics}.
\end{remark}

\begin{remark}[Torus bundle domains: open-string products] In the case of a torus bundle domain, consider three periodic orbits $\mathcal{O}_0,\mathcal{O}_1,\mathcal{O}_2 \subset \mathbb{T}^2$ of the monodromy matrix corresponding to three closed orbits $\Lambda_0, \Lambda_1, \Lambda_2$ of the Anosov flow. Using the notations of Proposition \ref{prop:torus_bundles_wfh} and Remark~\ref{rem:torus}, consider a chord $c_{0,1} \in HW^*(\mathcal{L}_{\mathcal{O}_0},\mathcal{L}_{{\mathcal{O}_1}})$ from $p_0\in \mathcal{O}_0$ to $q_1\in \mathcal{O}_1$ of slope $z_{0,1}^\pm\in \Gamma_{p_0,q_1}^{\pm}$, and a chord $c_{1,2} \in HW^*(\mathcal{L}_{\mathcal{O}_1},\mathcal{L}_{\mathcal{O}_2})$ from $p_1 \in \mathcal{O}_1$ to $q_2\in \mathcal{O}_2$ of slope $z_{1,2}^\pm\in \Gamma_{p_1,q_2}^{\pm}$. Recall that the slope is measured with respect to the fundamental domain $\mathbb{T}^2 \times [0, \nu)$, and we can view $z^{\pm}_{i,j}$ as elements of $[0,\nu)$. We don't necessarily have $q_1 = p_1$, but there exists $k_1 \in \Z$ such that $A^{k_1} p_1 = q_1$, and the set of such exponents is of the form $\mathcal{K}_1 = k_1 + \vert \mathcal{O}_1\vert \cdot \Z$. Note that the direction of $A^k c_{1,2}$ viewed as a chord in $\mathbb{T} \times [k\nu, (k+1)\nu)$ becomes $R_{\pm}\big(z_{1,2}^\pm + k\nu \big)$. It easily follows that for each $k \in \mathcal{K}_1$, there exists a unique chord $c_{0,2} \in HW^*(\mathcal{L}_{\mathcal{O}_0},\mathcal{L}_{\mathcal{O}_2})$ from $p_0$ to $A^{k} q_2$ such that there exists a topological disk with three boundary punctures mapping to $M$ with punctures asymptotic to $c_{0,1}$, $c_{1,2}$ and $c_{0,2}$, and boundary components in $\Lambda_0$, $\Lambda_1$ and $\Lambda_2$, see Figure~\ref{fig:torustriangle}. Here, we are using that the slope of the lift $\widetilde{R}_\pm$ of $R_\pm$ to $\mathbb{T}^2 \times \R_z$ is a monotonous function of $z \in \R$. Moreover, this topological disk is unique up to homotopy. However, we do not know whether some of these disks can be realized as Floer solutions contributing to the product $\mathfrak{m}_\pm\big(c_{0,1},c_{1,2}\big)$.
\end{remark}


\section{Discussion and future directions}\label{sec:discussion}

We now list some questions that remain open or unexplored. Fix an Anosov Liouville domain $V=[-1,1]\times M$. Recall that we have a distinguished family of exact cylindrical Lagrangians $\{\mathcal{L}_\Lambda\}_\Lambda$, indexed by the simple closed orbits of the corresponding Anosov flow on $M$, and $\mathcal{W}_0(V)$ (the \emph{orbit category} of the flow) is the full subcategory of the wrapped Fukaya category on these objects.

\medskip

\textbf{Bridge to Anosov theory.} As noted in the Introduction, the symplectic invariants of $V$ are homotopy invariants of the Anosov flow, so it is natural to wonder whether they can be used to address questions in Anosov theory. As the orbit category is only defined up to quasi-equivalence,\footnote{An $A_\infty$-functor $F : \mathcal{A} \rightarrow \mathcal{B}$ between $A_\infty$-categories is a quasi-equivalence if the cohomological functor $H^*F : H^* \mathcal{A} \rightarrow H^*\mathcal{B}$ is an equivalence of (ordinary) categories. More generally, two $A_\infty$-categories are quasi-equivalent if there exists a zig-zag of quasi-equivalences between them.} its quasi-equivalence class is then a homotopy invariant of the flow. Conversely, while it seems too strong to expect that it recovers the homotopy class of the flow, we can ask whether it recovers at least its \emph{topological equivalence class}:\footnote{Two flows are topologically equivalent, or orbit equivalent, if there exists a \emph{homeomorphism} sending the oriented trajectories of one to the oriented trajectories of the other, without necessarily preserving the parametrization. It follows from Anosov's structural stability that if two Anosov flows are homotopic through Anosov flows, then they are topologically equivalent.}

\begin{question}
If two smooth Anosov flows are topologically equivalent, are their orbit categories quasi-equivalent? Reciprocally, if two smooth (transitive) Anosov flows have quasi-equivalent orbit categories, are they topologically equivalent?
\end{question}

By a recent result of Barthelm\'e and Mann \cite{BM21}, for the case of $\mathbb{R}$-covered Anosov flows,\footnote{An Anosov flow on a $3$-manifold $M$ is $\R$-covered if its stable \emph{leaf space} (or equivalently, its unstable leaf space) is homeomorphic to $\R$. Here, the (un)stable leaf space is the quotient of the universal cover $\widetilde{M}$ of $M$ by the relation ``being on the same leaf of the lift $\widetilde{\mathcal{F}}^{\star}$ of $\mathcal{F}^{\star}$'', for $\star \in \{u, s\}$.} spectral rigidity reduces the above question to whether the quasi-equivalence class of the orbit category recovers the set of free homotopy classes represented by the closed orbits. 

Along the same lines as discussed above, the following seems to be an open question in Anosov theory, communicated to us by Rafael Potrie. 

\begin{question}
If two Anosov flows are topologically homotopic (i.e., topologically equivalent via an equivalence which is homotopic to the identity), are they homotopic as Anosov flows? 
\end{question}

Perhaps one can detect a negative answer to this question via the symplectic invariants of the associated domain, e.g., via the orbit category. For instance, it seems to be an open question whether, on a unit tangent bundle over a surface, the space of Riemannian metrics with Anosov geodesic flow is connected. Similarly, there could potentially be examples of such metrics whose geodesic flows may not be connected, in the space of Anosov flows, to an Anosov flow coming from a metric with negative curvature (the space of which is well-known to be connected). Potential candidates are the surfaces embedded in $\mathbb{R}^3$ constructed by Donnay and Pugh \cite{DP03}, which necessarily have regions of positive and negative curvature. Moreover, the set of connected components of the space of Anosov flows on the unit tangent bundle of a surface $\Sigma$ has a subset that is in one-to-one correspondence with $H^1(\Sigma;\mathbb{Z})$, and is therefore infinite (see \cite{M13}). On the other hand, there is a unique topological equivalence class \cite{Gh84}. 

In this article, we have computed the symplectic invariants of the prototypical \emph{algebraic} Anosov flows. In order to understand the invariants of examples that go beyond these basic cases, a starting point is to understand their behavior under surgeries. Indeed, Handel and Thurston \cite{HT80}, and Frank and Williams \cite{FW79}, gave the first examples of non-algebraic Anosov flows, via surgery methods. This was generalized by Goodman \cite{Go81} to a surgery along a curve in the vicinity of a periodic orbit. Moreover, Fried \cite{F83} produced \emph{topological} Anosov flows via Dehn surgery also along a periodic orbit. Recently, Shannon \cite{Sh20} proved that Fried's surgery is topologically equivalent to Goodman's. On the symplectic side, this raises the following question, which arose from conversations with Surena Hoozori:

\begin{question}
What is the effect of Fried--Goodman surgery on the symplectic invariants (e.g., on the orbit category, or on the symplectic cohomology)?
\end{question}

The spectral decomposition theorem of Smale \cite{Sm67} applies in particular to Anosov flows, as they satisfy Axiom A, i.e., the non-wandering set of the flow can be decomposed as a finite union of compact invariant sets --the \emph{basic sets}--, on each of which the flow is transitive. Brunella \cite{B93} has further shown the existence of incompressible tori which are transverse to the flow and separate the basic sets in the spectral decomposition, by choosing a Lyapunov function for the flow. Transitive Anosov flows can also admit transverse tori, such as the famous Bonatti--Langevin example~\cite{BR94}. By analogy to the case of a suspension of a toral Anosov map, in which the fibers are incompressible transverse tori and also weakly exact Lagrangian after taking the product with $\R$, one can ask:

\begin{question}
Do the Brunella tori of a non-transitive Anosov flow, and more generally tori transverse to an Anosov flow, give rise to weakly exact Lagrangian tori?
\end{question}

\medskip

\textbf{The open-closed map.} In Theorem~\ref{thm:unit}, we showed that the contractible component $\mathcal{OC}_0^c$ of the open-closed map restricted to the orbit category is neither surjective nor injective (its kernel has in fact infinite rank). The non-contractible component $\mathcal{OC}_0^{nc}$ seems highly nontrivial and might contain interesting information about the underlying Anosov flow. It is natural to ask:

\begin{question}
What are the image and kernel of the non-contractible component of the open-closed map $\mathcal{OC}_0^{nc} : \HH_{*-2}^{nc} \rightarrow SH^*_{nc}(V)$? 
\end{question}

We do not expect $\mathcal{OC}_0^{nc}$ to be injective in general (see Remark~\ref{rem:non_contractible}). When the underlying Anosov flow is transitive, its simple closed orbits form a dense subset of $M$. Hence, for every closed Reeb orbit $\gamma_\pm$, there exist a finite collection of Lagrangian cylinders $\mathcal{L}_{\Lambda_1}, \dots, \mathcal{L}_{\Lambda_k}$ and Reeb chords $c^{\pm}_1 : \Lambda_1 \rightarrow \Lambda_2, \dots, c^{\pm}_k : \Lambda_k \rightarrow \Lambda_1$ such that $\gamma^\pm$ is freely homotopic to a loop obtained by concatenating the chords $c_i$'s and paths in the $\mathcal{L}_{\Lambda_i}$'s in an alternating way. In other words, there are no obvious topological obstructions to the existence of holomorphic annuli in the chain level definition of $\mathcal{OC}^{nc}_O$ hitting a generator of $SC^*_{nc}$ corresponding to $\gamma_\pm$. This raises the following

\begin{question}
If the underlying Anosov flow is transitive, is $\mathcal{OC}_0^{nc}$ surjective?
\end{question}

\medskip

\textbf{Split-generation.} The following are questions pertaining to the wrapped Fukaya category $\mathcal{W}(V)$.

\begin{question}
We have the following alternative. Either
\begin{itemize}
    \item $\mathcal{W}_0(V)$ split-generates $\mathcal{W}(V)$, and therefore $\mathcal{W}(V)$ is not homologically smooth; or
    \item $\mathcal{W}_0(V)$ does not split-generate $\mathcal{W}(V)$, i.e., there exists an exact Lagrangian (closed or with cylindrical ends) in $V$ that is not split-generated by the Lagrangian cylinders.
\end{itemize}
Which one is true?
\end{question}

While we have not established split-generation, it is still possible that the Lagrangians $\mathcal L_\Lambda$ determine $\mathcal{W}(V)$ in the following weaker sense: 
\begin{question}
Is the restricted Yoneda functor
$$
\mathcal{W}(V) \rightarrow \mathcal{W}_0(V)-\mathrm{Mod}
$$
full and faithful?
\end{question}
Split-generation would imply an affirmative answer to this question, 
but the converse does not hold. For example, the zero-section $M \subset T^*M$ does not split-generate $\mathcal{W}(T^*M)$  (since it is a proper object in a non-proper $A_\infty$-category), but if $M$ is simply-connected, then the restricted Yoneda embedding to $C^*(M)-\mathrm{Mod}$ is full and faithful by Koszul duality; see \cite{Ad56, LS20}.
\smallskip

Assume now that $V=[-1,1]\times M$ is a torus bundle domain. Let $\mathcal{W}_0'(V)$ be the wrapped Fukaya category of $V$ (with coefficients in the \emph{Novikov ring} $\Lambda_\Z$) containing $\mathcal{W}_0(V) \otimes \Lambda_\Z$ and the relatively exact Lagrangian tori $\mathbb{T}_{s,z}$ given by the fibers of the torus fibration on $V$. The open-closed map $\mathcal{OC}_0 : \HH_{*-2}(\mathcal{W}_0(V))\otimes \Lambda_\Z \rightarrow SH^*(V)\otimes \Lambda_\Z$ factors as

\begin{center}
\begin{tikzcd}
 &\HH_{*-2}(\mathcal{W}_0(V))\otimes \Lambda_\Z \arrow{d} \arrow{r}{\mathcal{OC}_0} & SH^*(V)\otimes \Lambda_\Z \\
 & \HH_{*-2}(\mathcal{W}_0'(V)) \arrow{ur}[swap]{\mathcal{OC}_0'}
\end{tikzcd}
\end{center}
Recalling that $\mathcal{OC}_0$ does not hit the unit, we conjecture the following stronger claim:

\begin{conjecture} $\mathcal{OC}'$ does not hit the unit. Moreover, $\mathcal{W}_0'(V)$ is not homologically smooth.
\end{conjecture}

This conjecture would be true if the following question admits an affirmative answer: 

\begin{question}
Are the fibers $\mathbb{T}_{s,z}$ split-generated by the exact cylinders $\mathcal{L}_\Lambda$?
\end{question}

\smallskip

We now discuss split-generation for the McDuff domain. As pointed out to us by Mohammed Abouzaid, the Lagrangian foliation by cotangent fibers on $T^*\Sigma$ restricts to a foliation on the McDuff domain by punctured cotangent fibers, which are still Lagrangian (note that the magnetic term in the twisted symplectic form is horizontal, and therefore vanishes along the fibers). Abouzaid \cite{A11a} proved that a cotangent fiber generates the wrapped Fukaya category of $T^*\Sigma$. However, the punctured cotangent fibers in the McDuff domain are \emph{not} exact Lagrangians, as their boundary on the prequantization end is a Reeb orbit and thus not Legendrian. It is conceivable that they form the objects of an enlarged category, by taking into account the fact that the geometry of a twisted cotangent bundle is geometrically bounded (cf.\ \cite{Groman}). Assuming one can enlarge the wrapped Fukaya category to include such Lagrangians as objects, the following question is natural. 

\begin{question}
Do the punctured cotangent fibers generate an ``enlarged'' wrapped Fukaya category of the McDuff domain? 
\end{question}

\medskip

\textbf{Structure of the orbit category.} For a general Anosov Liouville domain with orbit category $\mathcal{W}_0$, Theorem~\ref{thm:WFH} implies that the cohomology category $H^* \mathcal{W}_0$ can be interpreted as a suitable fiber product of categories. Indeed, we can define $A_\infty$-categories $\mathcal{W}_0^-$, $\mathcal{W}_0^+$ and $\mathcal{W}_0^c$ whose objects are the same as $\mathcal{W}_0$ and whose morphism chain complexes are generated by chords in $M_-$ and interior intersection points, chords in $M_+$ and interior intersection points, and only interior intersection points, respectively. We can also define \emph{non-unital} $A_\infty$-categories $\mathcal{I}^-$ and $\mathcal{I}^+$ whose objects are the same as $\mathcal{W}_0$ and whose morphism chain complexes are generated by chords in $M_-$ only and chords in $M_+$ only, respectively. There are faithful functors $\mathcal{I}^\pm \rightarrow \mathcal{W}_0^\pm$  making $H^* \mathcal{I}^\pm$ ideals in $H^*\mathcal{W}_0^\pm$, in the sense that composing a morphism in $H^* \mathcal{I}^\pm$ with any morphism in $H^*\mathcal{W}_0^\pm$ gives a morphism in $H^* \mathcal{I}^\pm$. Moreover, the quotients\footnote{Here, we are considering the naive quotient at the level of morphisms. It is not a localization.} $H^* \mathcal{W}_0^\pm \slash H^* \mathcal{I}^\pm$are isomorphic to $H^*\mathcal{W}^c_0$ and $H^*\mathcal{W}_0$ becomes the (strict) fiber product
$$\begin{tikzcd}
H^*\mathcal{W}_0 \arrow[r] \arrow[d]	\arrow[dr, phantom, "\lrcorner", very near start]	&		H^*\mathcal{W}_0^- \arrow[d]\\
H^*\mathcal{W}_0^+	\arrow[r]               &	    H^*\mathcal{W}^c_0
\end{tikzcd}$$
where the functors $H^*\mathcal{W}_0^\pm \rightarrow H^*\mathcal{W}^c_0$ are the quotients by $H^* \mathcal{I}^\pm$. Note that $H^*\mathcal{W}_0^{c}$ is isomorphic to a disjoint union $$H^*\mathcal{W}_0^{c} = \bigsqcup_{\Lambda} \Z[x] \slash x^2$$ indexed by the simple closed orbits of the underlying Anosov flow.

Similarly, $\mathcal{I}^\pm \subset \mathcal{W}^{\pm}_0$ are ``$A_\infty$-ideals'', in the sense that any $A_\infty$-product in $\mathcal{W}^{\pm}_0$ involving a morphism in $\mathcal{I}^\pm$ gives a morphism in $\mathcal{I}^\pm$, and the quotients $\mathcal{W}^{\pm}_0 \slash \mathcal{I}^\pm$ are quasi-isomorphic to $\mathcal{W}^c_0$.\footnote{Here, we are considering a quotient by a \emph{non-full} and \emph{non-unital} subcategory.}

\begin{question}
Is $\mathcal{W}_0$ the homotopy fiber product of 
$$\begin{tikzcd}
                            &		\mathcal{W}_0^- \arrow[d]\\
\mathcal{W}_0^+	\arrow[r]    &	    \mathcal{W}^c_0
\end{tikzcd}$$
where the $A_\infty$-functors $\mathcal{W}_0^\pm \rightarrow \mathcal{W}^c_0$ are induced by the quotients by $\mathcal{I}^\pm$?
\end{question}

The subtlety here is to understand products in $\mathcal{W}_0$ involving both chords in $M_+$ and in $M_-$ as inputs. Using a similar neck-stretching argument as in the proof of Theorem~\ref{thm:WFH} (2) and using Lemma~\ref{lemma:disk}, one can show that for given inputs, there exist suitable Floer data so that the corresponding product vanishes. We expect that a careful filtration argument should lead to a positive answer to the previous question. These (homotopy) fiber product descriptions of $\mathcal{W}_0$ and $H^*\mathcal{W}_0$ mirror the geometric construction of $V$ as ``gluing two symplectizations along the skeleton''.

We expect that the \emph{Rabinowitz wrapped Fukaya category} of Ganatra--Gao--Venkatesh (see~\cite{Ga20}) should also enjoy a natural splitting when restricted to the Lagrangian cylinders $\mathcal{L}_\Lambda$, that should mirror the splitting of Rabinowitz Floer homology in Theorem~\ref{thm:RFH-SH}.

\newpage


\printbibliography

@article{AS,
    author = "Abbondandolo, Alberto and Schwarz, Matthias",
    title = "{On the Floer homology of cotangent bundles}",
    journal = "Communications on Pure and Applied Mathematics",
    volume = "59",
    number = "2",
    pages = "254--316",
    year = "2006"
}

@article{A10,
    author = "Abouzaid, Mohammed",
    title = "{A geometric criterion for generating the Fukaya category}",
    journal = "Publications Math\'{e}matiques de l'IH\'{E}S",
    volume = "112",
    pages = "191--240",
    year = "2010"
}

@article{A11a,
    author = "Abouzaid, Mohammed",
    title = "{A Cotangent Fibre Generates the Fukaya Category}",
    journal = "Advances in Mathematics",
    volume = "228",
    number = "2",
    pages = "894--939",
    year = "2011"
}

@article{A11b,
    author = "Abouzaid, Mohammed",
    title = "{A Topological Model for the Fukaya Categories of Plumbings}",
    journal = "Journal of Differential Geometry ",
    volume = "87",
    number = "1",
    pages = "1--80",
    year = "2011"
}

@article{Ad56,
    author = "Adams, John Frank",
    title = "{On the cobar construction}",
    journal = "Proceedings of the National Academy of Sciences of the United States of America",
    volume = "42",
    pages = "409--412",
    year = "1956"
}

@online{BM21,
    author = "Barthelm\'{e}, Thomas and Mann, Kathryn",
    title = "{Orbit equivalences of $\R$-covered Anosov flows and applications (appendix with Jonathan Bowden)}",
    eprinttype  = {arXiv},
    eprint  = {2012.11811},
    year = "2021"
}

@article{BR94,
    author = "Bonatti, Christian and Langevin, R\'{e}mi",
    title = "{Un exemple de flot d'Anosov transitif transverse \`{a} un tore et non conjugu\'{e} \`{a} une suspension}",
    journal = "Ergodic Theory and Dynamical Systems",
    volume = "14",
    number = "4",
    pages = "633--43",
    year = "1994"
}

@article{BO09,
    author = "Bourgeois, Fr\'{e}d\'{e}ric and Oancea, Alexandru",
    title = "{An exact sequence for contact- and symplectic homology}",
    journal = "Inventiones Mathematicae",
    volume = "174",
    number = "3",
    pages = "611--680",
    year = "2009"
}

@online{BC21,
    author = "Breen, Joseph and Christian, Austin",
    title = "{Mitsumatsu's Liouville domains are stably Weinstein}",
    eprinttype   = {arXiv},
    eprint       = {2109.07615v1},
    year = "2021"
}

@article{B93,
    author = "Brunella, Marco",
    title = "{Separating the basic sets of a nontransitive Anosov flow}",
    journal = "Bulletin of the London Mathematical Society",
    volume = "25",
    number = "5",
    pages = "487--490",
    year = "2005"
}

@book{C07,
    title = {Foliations and the Geometry of $3$-Manifolds},
    author = {Calegari, Danny},
    series = {Oxford Mathematical Monographs},
    year = {2007},
    publisher = {Clarendon; Oxford University Press}
}

@article{CDGG,
    author = "Chantraine, Baptiste and Dimitroglou Rizell, Georgios and Ghiggini, Paolo and Golovko, Roman",
    title = "{Geometric generation of the wrapped Fukaya category of Weinstein manifolds and sectors}",
    journal = "Annales Scientifiques de l'\'Ecole Normale Sup\'erieure (to appear)",
    eprinttype   = {arXiv},
    eprint       = {1712.09126},
    year        ="2022"
}

@article{C98,
    author = "Chekanov, Yuri",
    title = "{Lagrangian Intersections, Symplectic Energy, and Areas of Holomorphic Curves}",
    journal = "Duke Mathematical Journal",
    volume = "95",
    number = "1",
    pages = "213--26",
    year = "1998"
}

@inbook{CL09,
    author = "Cieliebak, Kai and Latschev, Janko",
    title = "{The role of string topology in symplectic field theory}",
    publisher = "American Mathematical Society",
    series = "CRM Proc. Lecture Notes",
    booktitle = "New perspectives and challenges in symplectic field theory",
    year = "2009",
    pages = "113--146"
}

@article{CO18,
    author = "Cieliebak, Kai and Oancea, Alexandru",
    title = "{Symplectic homology and the Eilenberg-Steenrod axioms. Appendix written jointly with Peter Albers}",
    journal = "Algebraic {\&} Geometric Topology",
    volume = "18",
    number = "4",
    pages = "1953--2130",
    year = "2018"
}

@article{CFO10,
    author = "Cieliebak, Kai and Fauenfelder, Urs and Oancea, Alexandru",
    title = "{Rabinowitz Floer homology and symplectic homology}",
    journal = "Annales Scientifiques de l'\'Ecole Normale Sup\'erieure",
    volume = "43",
    number = "4",
    pages = "957--1015",
    year = "2010"
}

@online{CHO,
    author = "Cieliebak, Kai and Hingston, Nancy and Oancea, Alexandru",
    title = "{Poincar\'e duality for loop spaces}",
    eprinttype   = {arXiv},
    eprint       = {2008.13161},
    year = "2022"
}

@article{DR16,
    author = "Dimitroglou Rizell, Georgios",
    title = "Lifting pseudo-holomorphic polygons to the symplectisation of $P \times \mathbb{R}$ and applications",
    journal = "Quantum Topology",
    volume = "7",
    number = "1",
    pages = "29--105",
    year = "2016"
}

@inbook{DP03,
    author = "Donnay, Victor and Pugh, Charles",
    title = "{Anosov geodesic flows for embedded surfaces}",
    publisher = "Soci\'{e}t\'{e} Math\'{e}matique de France",
    series = "Ast\'{e}rique",
    booktitle = "Geometric methods in dynamics. II",
    volume = "287",
    year = "2003",
    pages = "61--69"
}

@article{EG91,
    author = "Eliashberg, Yakov and Gromov, Mikhael",
    title = "{Convex Symplectic Manifolds}",
    journal = "Proceedings of Symposia in Pure Mathematics",
    volume = "52",
    number = "2",
    pages = "135--162",
    year = "1991"
}

@book{ET98,
    title = {Confoliations},
    author = {Eliashberg, Yakov and Thurston, William},
    series = {University Lecture Series},
    volume = {13},
    year = {1998},
    publisher = {American Mathematical Society}
}

@book{FH19,
    title = {Hyperbolic Flows},
    author = {Fisher, Todd and Hasselblatt, Boris},
    series = {Zurich lectures in advanced mathematics},
    year = {2019},
    publisher = {European Mathematical Society}
}

@inbook{FW79,
    author = "Franks, John and Williams, Bob",
    title = "{Anomalous Anosov flows}",
    publisher = "Springer, Berlin",
    series = "Lecture Notes in Mathematics",
    booktitle = "Global theory of dynamical systems (Proc. Internat. Conf., Northwestern Univ., Evanston, Ill., 1979)",
    volume = "819",
    year = "1980",
    pages = "158–174"
}

@article{F82,
    author = "Fried, David",
    title = "{Flow equivalence, hyperbolic systems and a new zeta function for flows}",
    journal = "Commentarii Mathematici Helvetici",
    volume = "57",
    number = "1",
    pages = "237--59",
    year = "1982"
}

@article{F83,
    author = "Fried, David",
    title = "{Transitive Anosov flows and pseudo-Anosov maps}",
    journal = "Topology",
    volume = "22",
    number = "3",
    pages = "299–303",
    year = "1983"
}

@online{G13,
    author = "Ganatra, Sheel",
    title = "{Symplectic Cohomology and Duality for the Wrapped Fukaya Category}",
    eprinttype   = {arXiv},
    eprint       = {1304.7312},
    year = "2013"
}

@online{G19,
    author = "Ganatra, Sheel",
    title = "{Automatically generating Fukaya categories and computing quantum cohomology}",
    eprinttype   = {arXiv},
    eprint       = {1605.07702},
    year = "2019"
}

@misc{Ga20,
    author = "Ganatra, Sheel",
    title = "{Rabinowitz wrapped Fukaya categories and the categorical formal punctured neighborhood of infinity}",
    howpublished = "Lecture in the workshop ``Current Advances in Mirror Symmetry''",
    url  = "https://people.math.harvard.edu/~auroux/schms2020notes/201205-Ganatra.pdf",
    year = "2020"
}

@article{GPS1,
    author = "Ganatra, Sheel and Pardon, John and Shende, Vivek",
    title = "{Covariantly Functorial Wrapped Floer Theory on Liouville Sectors}",
    journal = "Publications Math\'{e}matiques de l'IH\'{E}S",
    volume = "131",
    number = "1",
    pages = "73--200",
    year = "2020 "
}

@online{GPS2,
    author = "Ganatra, Sheel and Pardon, John and Shende, Vivek",
    title = "{Sectorial Descent for Wrapped Fukaya Categories}",
    eprinttype   = {arXiv},
    eprint       = {1809.03427},
    year = "2022"
}

@online{GPS3,
    author = "Ganatra, Sheel and Pardon, John and Shende, Vivek",
    title = "{Microlocal Morse Theory of Wrapped Fukaya Categories}",
    eprinttype   = {arXiv},
    eprint       = {1809.08807},
    year = "2022"
}

@article{Ge95,
    author = "Geiges, Hansj{\"o}rg",
    title = "{Examples of Symplectic 4-Manifolds with Disconnected Boundary of Contact Type}",
    journal = "Bulletin of the London Mathematical Society",
    volume = "27",
    number = "3",
    pages = "278--280",
    year = "1995"
}

@article{Gh84,
    author = "Ghys, \'Etienne",
    title = "{Flots d'Anosov sur les $3$-vari\'et\'es fibr\'ees en cercles.} ({French})",
    journal = "Ergodic Theory and Dynamycal Systems ",
    volume = "4",
    number = "1",
    pages = "67--80",
    year = "1984"
}

@article{Gh92,
    author = "Ghys, \'Etienne",
    title = "{D\'eformations de flots d'Anosov et de groupes fuchsiens.} ({French})",
    journal = "Annales de l'Institut Fourier",
    volume = "42",
    number = "1--2",
    pages = "209--47",
    year = "1992"
}

@article{Gh93,
    author = "Ghys, \'Etienne",
    title = "{Rigidit\'e diff\'erentiable des groupes fuchsiens.} ({French})",
    journal = "Publications Math\'{e}matiques de l’IH\'{E}S",
    volume = "78",
    number = "1",
    pages = "163--85",
    year = "1993"
}

@inbook{Go81,
    author = "Goodman, Sue",
    title = "{Dehn surgery on Anosov flows}",
    publisher = "Springer, Berlin",
    series = "Lecture Notes in Mathematics",
    booktitle = "Geometric dynamics (Rio de Janeiro, 1981)",
    volume = "1007",
    year = "1983",
    pages = "300--307"
}

@article{Groman,
    author = "Groman, Yoel",
    title = "{Floer theory and reduced cohomology on open manifolds}",
    journal = "Geometry {\&} Topology (to appear)",
    eprinttype   = {arXiv},
    eprint       = {1510.04265},
    year = "2022"
}

@article{Gr85,
    author = "Gromov, Mikhael",
    title = "{Pseudo holomorphic curves in symplectic manifolds}",
    journal = "Inventiones Mathematicae",
    volume = "82",
    number = "2",
    pages = "307--347",
    year = "1985"
}

@article{HT80,
    author = "Handel, Michael and Thurston, William",
    title = "{Anosov flows on new three manifolds}",
    journal = "Inventiones Mathematicae",
    volume = "59",
    number = "2",
    pages = "95--103",
    year = "1980"
}

@online{H22a,
    author = "Hozoori, Surena",
    title = "{Symplectic Geometry of Anosov Flows in Dimension 3 and Bi-Contact Topology}",
    eprinttype   = {arXiv},
    eprint       = {2009.02768},
    year = "2022"
}

@misc{H,
  author = "Hozoori, Surena",
  howpublished = "Personal communication"
}

@book{Hu97,
    title = {Gromov’s Compactness Theorem for Pseudo-holomorphic Curves},
    author = {Hummel, Christoph },
    series = {Progress in Mathematics},
    volume = {151},
    year = {1997},
    publisher = {Birkh{\"a}user Verlag, Basel}
}

@article{LS91,
    author = "Lalonde, Fran{\c{c}}ois and Sikorav, Jean-Claude",
    title = "{Sous-vari\'et\'es lagrangiennes et lagrangiennes exactes des fibr\'es cotangents.} ({French})",
    journal = "Commentarii Mathematici Helvetici",
    volume = "66",
    number = "1",
    pages = "18--33",
    year = "1991"
}

@article{LS20,
    author = "Lazarev, Oleg and Sylvan, Zachary",
    title = "{Prime-localized subdomains}",
    journal = "Geometry {\&} Topology (to appear)",
    eprinttype   = {arXiv},
    eprint       = {2009.09490},
    year = "2020"
}

@misc{Mas22,
    author = "Massoni, Thomas",
    title = "{Anosov flows and Liouville pairs in dimension three (in preparation)}",
    year = "2022"
}

@article{MNW13,
    author = "Massot, Patrick and Niederkr{\"u}ger, Klaus and Wendl, Chris",
    title = "{Weak and strong fillability of higher dimensional contact manifolds}",
    journal = "Inventiones Mathematicae",
    volume = "192",
    number = "2",
    pages = "287--373",
    year = "2013"
}

@article{M13,
    author = "Matsumoto, Shigenori",
    title = "{The space of (contact) Anosov flows on 3-manifolds}",
    journal = "Journal of Mathematical Sciences of The University of Tokyo",
    volume = "20",
    number = "2",
    pages = "445--460",
    year = "2013"
}

@article{McD91,
    author = "McDuff, Dusa",
    title = "{Symplectic manifolds with contact type boundaries}",
    journal = "Inventiones Mathematicae",
    volume = "103",
    number = "3",
    pages = "651--671",
    year = "1991"
}

@article{M95,
    author = "Mitsumatsu, Yoshihiko",
    title = "{Anosov flows and non-Stein symplectic manifolds}",
    journal = "Annales de l'Institut Fourier",
    volume = "45",
    number = "5",
    pages = "1407--1421",
    year = "1995"
}

@article{Oh97,
    author = "Oh, Yong-Geun ",
    title = "{Gromov-Floer theory and disjunction energy of compact Lagrangian embeddings}",
    journal = "Mathematical Research Letters",
    volume = "4",
    number = "6",
    pages = "895--905",
    year = " 1997"
}

@article{RS17,
    author = "Ritter, Alexander and Smith, Ivan ",
    title = "{The monotone wrapped Fukaya category and the open-closed string map}",
    journal = " Selecta Mathematica",
    volume = "23",
    number = "1",
    pages = "533--642",
    year = "2017"
}

@article{R08,
    author = "Rouquier, Rapha{\"{e}}l",
    title = "{Dimensions of Triangulated Categories}",
    journal = "Journal of K-Theory",
    volume = "1",
    number = "2",
    pages = "193--256",
    year = "2008"
}

@phdthesis{Sch,
    author = "Schoenfeld, Eric",
    title = "{Higher Symplectic Field Theory Invariants for Cotangent Bundles of Surfaces}",
    school = "Stanford University",
    year = "2009"
}

@inbook{Sei02,
    author = "Seidel, Paul",
    title = "{Fukaya categories and deformations}",
    publisher = "Higher Ed. Press, Beijing",
    booktitle = "Proceedings of the International Congress of Mathematicians",
    volume = "II",
    year = "2002",
    pages = "351--360"
}

@inbook{Sei09,
    author = "Seidel, Paul",
    title = "{Symplectic homology as Hochschild homology}",
    publisher = "American Mathematical Society",
    booktitle = "Algebraic geometry—Seattle 2005",
    series = "Proceedings of Symposia in Pure Mathematics",
    volume = "80",
    year = "2009",
    pages = "351--360"
}

@phdthesis{Sh20,
    author = "Shannon, Mario",
    title = "{Dehn surgeries and smooth structures on $3$-dimensional transitive Anosov flows}",
    school = "Universit\'{e} de Bourgogne",
    url = "https://tel.archives-ouvertes.fr/tel-02951219/document"
}

@article{Sm67,
    author = "Smale, Stephen",
    title = "{Differentiable dynamical systems}",
    journal = "Bulletin of the American Mathematical Society",
    volume = "73",
    number = "6",
    pages = "747-817",
    year = "1967"
}

@book{Sp66,
    title = {Algebraic Topology},
    author = {Spanier, Edwin},
    year = {1966},
    publisher = {McGraw-Hill}
}

@article{U19,
    author = "Uebele, Peter ",
    title = "{Periodic Reeb flows and products in symplectic homology}",
    journal = "Journal of Symplectic Geometry",
    volume = "17",
    number = "4",
    pages = "1201--1250",
    year = "2019"
}

@article{V74,
    author = "Verjovsky, Alberto",
    title = "{Codimension one Anosov flows}",
    journal = "Bol. Soc. Mat. Mexicana",
    volume = "19",
    number = "2",
    pages = "49--77",
    year = "1974"
}

@book{W97,
    title = {Introduction to Homological Algebra},
    author = {Weibel, Charles},
    series = {Reprint. 1997},
    year = {2003},
    publisher = {Cambridge University Press}
}

@article{W91,
    author = "Weinstein, Alan",
    title = "{Contact Surgery and Symplectic Handlebodies}",
    journal = "Hokkaido Mathematical Journal",
    volume = "20",
    number = "2",
    pages = "241--51",
    year = "1991"
}

\end{document}